\documentclass{lms}

\usepackage[T1]{fontenc}
\usepackage[utf8]{inputenc}
\usepackage{hyperref}
\hypersetup{
  colorlinks,
  linkcolor={red!50!black},
  citecolor={green!40!black},
  urlcolor={blue!80!black}
}

\usepackage{amsmath,amssymb,stmaryrd}
\usepackage{tikz-cd}
\usepackage{enumitem}

\setlist{itemsep=1pt,topsep=4pt}
\setitemize{label=\tiny\raisebox{1ex}{\textbullet}}

\newcounter{cstcount}                     
\newcommand{\newC}{\refstepcounter{cstcount} \ensuremath{C_{\thecstcount}}}
\newcommand{\oldC}[1]{\ensuremath{C_{\ref*{#1}}}}

\newcommand{\A}{\mathbb{A}}

\newcommand{\Q}{\mathbb{Q}}
\newcommand{\R}{\mathbb{R}}
\newcommand{\C}{\mathbb{C}}
\newcommand{\Z}{\mathbb{Z}}
\newcommand{\N}{\mathbb{N}}

\newcommand{\V}{\mathcal{V}}

\newcommand{\Proj}{\mathbb{P}}
\newcommand{\Half}{\mathbb{H}}
\newcommand{\Qbar}{\overline{\Q}}
\newcommand{\Zhat}{\widehat{\Z}}
\newcommand{\Order}{\mathcal{O}}
\newcommand{\eps}{\varepsilon}
\newcommand{\Lcal}{\mathcal{L}}
\newcommand{\Mcal}{\mathcal{M}}
\newcommand{\Scal}{\mathcal{S}}
\newcommand{\Tcal}{\mathcal{T}}
\newcommand{\CC}{\mathcal{C}}

\newcommand{\Frac}{\mathcal{F}}
\newcommand{\linv}{\tfrac{1}{\ell}}
\newcommand{\U}{\mathcal{U}}

\newcommand{\st}{\ |\ }
\newcommand{\defby}{\ :\ }
\newcommand{\from}{\colon}
\newcommand{\embed}{\hookrightarrow}
\newcommand{\mat}[4]{\left(\begin{matrix}#1&#2\\#3&#4\end{matrix}\right)}
\newcommand{\tmat}[4]{\left(\begin{smallmatrix}#1&#2\\#3&#4\end{smallmatrix}\right)}
\newcommand{\floor}[1]{\left\lfloor #1 \right\rfloor}
\newcommand{\Zint}[1]{\llbracket #1\rrbracket}

\newcommand{\abs}[1]{\left|#1\right|}
\newcommand{\act}[2]{#1\cdot#2}
\newcommand{\bpar}[1]{\bigl(#1\bigr)}

\newcommand{\dual}{\vee}

\DeclareMathOperator{\Aut}{Aut}
\DeclareMathOperator{\End}{End}
\DeclareMathOperator{\GL}{GL}
\DeclareMathOperator{\Sp}{Sp}
\DeclareMathOperator{\GSp}{GSp}
\DeclareMathOperator{\SL}{SL}

\DeclareMathOperator{\Sh}{Sh}

\DeclareMathOperator{\Tr}{Tr}
\DeclareMathOperator{\wt}{wt}
\DeclareMathOperator{\Res}{Res}

\DeclareMathOperator{\h}{\mathit{h}}
\DeclareMathOperator{\Hbar}{\overline{\mathit{h}}}

\DeclareMathOperator{\GC}{GC}
\DeclareMathOperator{\SGC}{SGC}
\DeclareMathOperator{\Heckedeg}{\mathit{d}}
\DeclareMathOperator{\isogdeg}{\mathit{l}}

\newtheorem{thm}{Theorem}[section]
\newtheorem{prop}[thm]{Proposition}
\newtheorem{lem}[thm]{Lemma}
\newtheorem{cor}[thm]{Corollary}
\newnumbered{defn}[thm]{Definition}
\newnumbered{ex}[thm]{Example}
\newnumbered{rem}[thm]{Remark}

\newcommand{\claimemph}[1]{\vspace{2mm} \noindent \emph{#1}}

\title{Degree and height estimates for modular equations\\ on PEL
  Shimura varieties}
\shorttitle{Degree and height estimates for modular equations}
\author{Jean Kieffer}

\classno{11G18, 11F32, 11G50, 11F46}

\begin{document}

\maketitle

\begin{abstract}
  We define modular equations in the setting of PEL Shimura varieties
  as equations describing Hecke correspondences, and prove upper
  bounds on their degrees and heights. This extends known results
  about elliptic modular polynomials, and implies complexity bounds
  for number-theoretic algorithms using these modular equations. In
  particular, we obtain tight degree bounds for modular equations of
  Siegel and Hilbert type for abelian surfaces.
\end{abstract}


\section{Introduction}
\label{sec:intro}

Modular equations encode the presence of isogenies between polarized
abelian varieties. An example is given by the elliptic modular
polynomial~$\Phi_\ell$, where~$\ell$ is a prime: this bivariate
polynomial vanishes on the $j$-invariants of $\ell$-isogenous elliptic
curves \cite[§11.C]{cox_PrimesFormNy2013}, and can be used to detect
and compute such isogenies
\cite{elkies_EllipticModularCurves1998}. Elliptic modular polynomials
are used for instance in the SEA algorithm to count points on elliptic
curves over finite fields~\cite{schoof_CountingPointsElliptic1995},
and in multi-modular methods to compute class polynomials of imaginary
quadratic fields~\cite{sutherland_ComputingHilbertClass2011}; being
able to compute isogenies also has applications in cryptography.
Analogues of~$\Phi_\ell$ for principally polarized abelian surfaces,
called Siegel and Hilbert modular equations in dimension~$2$, have
recently been defined and
computed~\cite{milio_QuasilinearTimeAlgorithm2015,milio_ModularPolynomialsHilbert2020,martindale_HilbertModularPolynomials2020},
and are of similar interest.

In the first part of this paper, we define modular equations in the
general setting of PEL Shimura varieties of finite level; these
varieties are moduli spaces for abelian varieties with polarization,
endomorphisms, and level structure, hence the name.  Choose connected
components~$\Scal$ and~$\Tcal$ of such a Shimura variety of
dimension~$n\geq 1$; they have a canonical model over a certain number
field~$L$. Choose coordinates on~$\Scal$ and~$\Tcal$ that are defined
over~$L$.  Let~$H_\delta$ be an absolutely irreducible Hecke
correspondence defined by an adelic element~$\delta$ of the underlying
reductive group, and let~$\Heckedeg(\delta)$ be the degree
of~$H_\delta$. In the modular interpretation, $H_\delta$ parametrizes
isogenies of a certain degree~$\isogdeg(\delta)$ between abelian
varieties with PEL structure.  Then the \emph{modular equations of
  level~$\delta$} are a family of~$n+1$ univariate
polynomials~$(\Psi_{\delta,m})_{1\leq m\leq n+1}$ with coefficients in
the function field~$L(\Scal)$ of~$\Scal$, of degree at
most~$d(\delta)$, describing~$H_\delta$ on~$\Scal\times \Tcal$. This
definition includes all the examples of modular polynomials cited
above, and provides a unified context to study them.

For each~$1\leq m\leq n+1$, the coefficients of~$\Psi_{\delta,m}$ can
be seen as multivariate rational fractions with coefficients in~$L$.
From an algorithmic point of view, two quantities are of interest:
first, the total degree of these fractions; and second, their height,
which measures the size of their coefficients. For instance, if
$\Frac\in \Q(Y_1,\ldots,Y_n)$, write $\Frac=P/Q$ where
$P,Q\in \Z[Y_1,\ldots,Y_n]$ are coprime; then the height~$\h(\Frac)$
of~$\Frac$ is defined as the maximum of~$\log\abs{c}$, where~$c$ runs
through the nonzero coefficients of~$P$ and~$Q$.

Our main result gives upper bounds on the degrees and heights of the
coefficients of modular equations on a given PEL Shimura variety in
terms of~$\Heckedeg(\delta)$ and~$\isogdeg(\delta)$. This provides complexity
bounds for algorithms involving these modular equations.

\begin{thm}
  \label{thm:main}
  Let~$\Scal$ and~$\Tcal$ be connected components of a simple PEL
  Shimura variety of type~(A) or~(C) of finite level and
  dimension~$n\geq 1$, with underlying reductive group~$G$. Let~$L$ be
  the field of definition of~$\Scal$ and~$\Tcal$, and choose
  coordinates on~$\Scal$ and~$\Tcal$ that are defined over~$L$. Then
  there exist constants~$\newC\label{cst:main-degree}$
  and~$\newC\label{cst:main-height}$ such that the following holds.
  Let~$H_\delta$ be an absolutely irreducible Hecke correspondence
  on~$\Scal\times \Tcal$ defined by an adelic element~$\delta$ of~$G$;
  let~$\Heckedeg(\delta)$ be the degree of~$H_\delta$, and
  let~$\isogdeg(\delta)$ be the degree of the isogenies described
  by~$H_\delta$ in the modular interpretation.  Let~$\Frac$ be a
  multivariate rational fraction over~$L$ occuring as a coefficient of
  one of the modular equations~$\Psi_{\delta,m}$ for
  $1\leq m\leq n+1$. Then
  \begin{enumerate}
  \item \label{item:degree-bound} The total degree of~$\Frac$ is
    bounded above by~$\oldC{cst:main-degree} \Heckedeg(\delta)$.
  \item \label{item:height-bound} The height of~$\Frac$ is bounded
    above
    by~$\oldC{cst:main-height} \Heckedeg(\delta) \max\{1, \log \isogdeg(\delta)\}$.
  \end{enumerate}
\end{thm}

This result generalizes known bounds on the size of the elliptic
modular polynomial~$\Phi_\ell$, which has degree~$\ell+1$ in both
variables.  We have $\h(\Phi_\ell) \sim 6 \ell \log \ell$ as~$\ell$
tends to infinity
\cite{cohen_CoefficientsTransformationPolynomials1984}, and explicit
bounds can be given
\cite{broker_ExplicitHeightBound2010}.
Since~$\Heckedeg(\delta) = \ell+1$ and~$\isogdeg(\delta)=\ell$ in this
case, Theorem~\ref{thm:main} seems optimal up to the value of the
constants.

In the case of Siegel and Hilbert modular equations in dimension~$2$,
this result is new, and we can provide explicit values for the
constants~$\oldC{cst:main-degree}$ and~$\oldC{cst:main-height}$. In
particular, the degree bounds that we obtain match exactly with
experimental data.

The strategy to prove part~\ref{item:degree-bound} of
Theorem~\ref{thm:main} is to exhibit a particular modular form that
behaves as the denominator of~$\Psi_{\delta,m}$, and to control its
weight; then, we show that rewriting quotients of modular forms in
terms of the chosen coordinates transforms bounded weights into bounded
degrees.  The proof of part~\ref{item:height-bound} is inspired by
previous works
on~$\Phi_\ell$~\cite{pazuki_ModularInvariantsIsogenies2019}. We prove
height bounds on \emph{evaluations} of modular equations at certain
points using well-known results on the Faltings height of isogenous
abelian
varieties~\cite{faltings_EndlichkeitssaetzeFuerAbelsche1983}. Then we
use a general tight relation between the height of a rational fraction
over a number field and the height of its evaluations at sufficiently
many points, proved by the author in a separate
paper~\cite{kieffer_UpperBoundsHeights2020}.

This paper is organized as follows. In Section~\ref{sec:pel}, we
recall the necessary background on PEL Shimura varieties. In
Section~\ref{sec:modeq}, we define the modular equations associated
with a choice of PEL setting and absolutely irreducible Hecke
correspondence, and explain how we recover the Siegel and Hilbert
modular equations in dimension~$2$ as special
cases. Sections~\ref{sec:degree-bound} and~\ref{sec:height-bound} are
devoted to the proof of the degree and height bounds respectively.

\section{Background on PEL Shimura varieties}
\label{sec:pel}


Our presentation is based on Milne's expository
notes~\cite{milne_IntroductionShimuraVarieties2005}, which serve as a
general reference for this section. These notes are themselves based
on Deligne's reformulation of Shimura's
works~\cite{deligne_TravauxShimura1970}. We use the following
notation: if~$G$ is a connected reductive algebraic group over~$\Q$,
then
\begin{itemize}
\item $G^{\mathrm{der}}$ is the derived group of~$G$,
\item $Z$ is the center of~$G$,
\item  $G^{\mathrm{ad}} = G/Z$ is the adjoint group of~$G$,
\item $T = G/G^\mathrm{der}$ is the largest abelian quotient of~$G$,
\item $\nu\from G\to T$ is the natural quotient map,
\item $G^{\mathrm{ad}}(\R)_+$ is the connected component of 1 in
$G^{\mathrm{ad}}(\R)$ for the real topology,
\item $G(\R)_+$ is the
preimage of~$G^{\mathrm{ad}}(\R)_+$ in~$G(\R)$, and finally
\item  $G(\Q)_+ = G(\Q)\cap G(\R)_+$. 
\end{itemize}
We write~$\A_f$ for the ring of finite adeles of~$\Q$.

\subsection{Simple PEL Shimura varieties of type~(A) or~(C)}
\label{sub:pel}


\paragraph{PEL data.}
Let~$(B,*)$ be a finite-dimensional simple~$\Q$-algebra with positive
involution. The center~$F$ of~$B$ is a number field; let~$F_0\subset F$ be the
subfield of invariants under~$*$. For simplicity, we make the
technical assumption that~$B$ is either of type~(A) or~(C)
\cite[Prop.~8.3]{milne_IntroductionShimuraVarieties2005}: this means
that for every embedding~$\theta$ of~$F_0$ in an algebraic
closure~$\Qbar$ of~$\Q$, the algebra with
positive involution~$(B\otimes_{F_0,\theta}\Qbar, *)$ is isomorphic to a
product of factors of the form, respectively,
\begin{itemize}
\item[(A)] $M_n(\Qbar)\times M_n(\Qbar)$ with $(a,b)^* = (b^t, a^t)$,
  or
\item[(C)] $M_n(\Qbar)$ with $a^* = a^t$.
\end{itemize}


Let~$(V,\psi)$ be a faithful symplectic~$(B,*)$-module. This means
that~$V$ is a finite-dimensional $\Q$-vector space equipped with a
faithful $B$-module structure and a nondegenerate alternating
$\Q$-bilinear form~$\psi$ such that for all~$b\in B$ and for
all~$u,v\in V$,
\begin{displaymath}
  \psi(b^* u, v) = \psi(u, bv).
\end{displaymath}

Let~$\GL_B(V)$ denote the group of automorphisms of~$V$ respecting the
action of~$B$, and let~$G$ be its reduced algebraic subgroup defined by
\begin{displaymath}
  \begin{aligned}
    G(\Q) &= \bigl\{g\in \GL_B(V) \st \psi(gx, gy) = \psi(\mu(g)x, y)
    \text{ for some } \mu(g)\in F_0^\times \bigr\}.
  \end{aligned}
\end{displaymath}
The group~$G$ is connected and reductive, and its derived group is
$G^\mathrm{der} = \ker(\mu) \cap \ker(\det)$
\cite[Prop.~8.7]{milne_IntroductionShimuraVarieties2005}.  We warn the
reader that our~$G$ is denoted by~$G_1$ in
\cite[§8 of the 2017 version]{milne_IntroductionShimuraVarieties2005}. In Milne's
terminology, our~$G$ will define a Shimura variety (so that the results
of~\cite[§5]{milne_IntroductionShimuraVarieties2005} apply), but not
strictly speaking a PEL Shimura variety. This choice of reductive
group will allow us to consider more Hecke correspondences later on.


Let~$x$ be a complex structure on~$V(\R)$, meaning an endomorphism
of~$V(\R)$ such that $x^2 =\nolinebreak -1$. We say that~$x$ is
\emph{positive for~$\psi$} if it commutes with the action of~$B$ and
if the bilinear form~$(u,v)\mapsto \psi\bigl(u, x(v)\bigr)$ on~$V(\R)$
is symmetric and positive definite. In particular, $x\in G(\R)$
and~$\mu(x)=1$. Such a complex structure~$x_0$ exists
\cite[Prop.~8.14]{milne_IntroductionShimuraVarieties2005}.
Define~$X_+$ to be the orbit of~$x_0$ under the action of~$G(\R)_+$ by
conjugation; the space~$X_+$ is a hermitian symmetric domain
\cite[Cor.~5.8]{milne_IntroductionShimuraVarieties2005}. We call the
tupe $(B, *, V, \psi, G, X_+)$ a \emph{simple PEL Shimura datum of
  type~(A) or~(C)}, or simply a \emph{PEL datum}. To simplify
notations, we abbreviate PEL data as pairs~$(G,X_+)$, the underlying
data~$(V,\psi)$ and~$(B,*)$ being implicit.


\paragraph{PEL Shimura varieties.}
Let $(G,X_+)$ be a PEL datum as above, let~$K$ be a compact open
subgroup of~$G(\A_f)$, and let~$K_\infty$ be the stabilizer of~$x_0$
in~$G(\R)_+$. The \emph{PEL Shimura variety} associated with~$(G,X_+)$
of level~$K$ is the double quotient
\begin{equation}
  \label{eq:Sh-CC}
  \begin{aligned}
    \Sh_K(G,X_+)(\C) &= G(\Q)_+\backslash (X_+\times G(\A_f)) /K \\
    &= G(\Q)_+ \backslash (G(\R)_+ \times G(\A_f)) / K_\infty\times K.
  \end{aligned}
\end{equation}
Actually, this quotient will be the set of~$\C$-points of the Shimura
variety, hence the notation.  In the first line of~\eqref{eq:Sh-CC},
the group~$G(\Q)_+$ acts on both~$X_+$ and~$G(\A_f)$ by conjugation
and left multiplication respectively, and~$K$ acts on $G(\A_f)$ by
right multiplication. When the context is clear, we omit~$(G,X_+)$
from the notation. The set~$\Sh_K(\C)$ is given the quotient topology
obtained from the real topology on~$G(\R)_+$ and the adelic topology
on~$G(\A_f)$.


In order to describe~$\Sh_K(\C)$ more explicitly, we study its
connected components.  The projection to the second factor induces a
map with connected fibers from~$\Sh_K(\C)$ to the double
quotient~$G(\Q)_+\backslash G(\A_f) / K$, which is finite
\cite[Lem.~5.12]{milne_IntroductionShimuraVarieties2005}. Let~$\CC$ be
a set of representatives in~$G(\A_f)$ for this double quotient. The
connected component~$\Scal_c$ of~$\Sh_K(\C)$ indexed by~$c\in \CC$ can
be identified with~$\Gamma_c\backslash X_+$, where
$\Gamma_c = G(\Q)_+\cap\, cKc^{-1}$ is an arithmetic subgroup of
$\Aut(X_+)$
\cite[Lem.~5.13]{milne_IntroductionShimuraVarieties2005}. Thus, the
Shimura variety~$\Sh_K(\C)$ has a natural structure of a complex
analytic space, and is an algebraic variety by the theorem of Baily
and Borel \cite[Thm.~3.12]{milne_IntroductionShimuraVarieties2005}.


Since~$G^\mathrm{der}$ is simply connected, by \cite[Thm.~5.17 and
Lem.~5.20]{milne_IntroductionShimuraVarieties2005} (the assumption
that~$K$ is sufficiently small is not actually needed there), the
map~$\nu$ induces an isomorphism
\begin{displaymath}
  G(\Q)_+ \backslash G(\A_f)/K \simeq \nu(G(\Q)_+) \backslash T(\A_f) / \nu(K).
\end{displaymath}
Therefore the set of connected components of~$\Sh_K(\C)$ is a finite
abelian group. Moreover, each connected component is itself a Shimura
variety with underlying group~$G^\mathrm{der}$
\cite[Rem.~5.23]{milne_IntroductionShimuraVarieties2005}.


A fundamental theorem states that~$\Sh_K(G,X_+)$ exists as an
algebraic variety defined over the \emph{reflex field} $E(G,X_+)$,
which is a number field contained in~$\C$, depending only on the PEL
datum~\cite[§12-14]{milne_IntroductionShimuraVarieties2005}. The field
of definition of the individual connected components of~$\Sh_K(\C)$
depends on~$K$, and is a finite abelian extension of~$E(G,X_+)$.

\subsection{The modular interpretation}
\label{sub:modular-int}

Our motivation in constructing PEL Shimura varieties is to obtain
moduli spaces of complex abelian varieties with polarization,
endomorphism, and level structures. This \emph{modular interpretation}
of PEL Shimura varieties is usually formulated in terms of isogeny
classes of abelian varieties
\cite[Thm.~8.17]{milne_IntroductionShimuraVarieties2005}. In order to
obtain a modular interpretation in terms of \emph{isomorphism} classes
of abelian varieties in the spirit
of~\cite[§2.6.2]{carayol_MauvaiseReductionCourbes1986}, we fix
\begin{itemize}
\item a PEL datum~$(G,X_+)$,
\item a lattice~$\Lambda_0\subset V$,
\item a compact open subgroup~$K\subset G(\A_f)$ which stabilizes the
  lattice~$\widehat{\Lambda}_0 = \Lambda_0\otimes\Zhat\subset
  V(\A_f)$, and
\item a set~$\CC\subset G(\A_f)$ of representatives for the finite
  double quotient~$G(\Q)_+\backslash G(\A_f) / K$.
\end{itemize}

By definition, a lattice in~$V$ is a subgroup of~$V(\Q)$ generated by
a $\Q$-basis of~$V$, hence a free~$\Z$-module of rank~$\dim V$. If~$p$
is a prime number, then a lattice in~$V(\Q_p)$ is a subgroup of the
form~$\bigoplus_{i\in I} \Z_p e_i$ where~$(e_i)_{i\in I}$ is a
$\Q_p$-basis of~$V(\Q_p)$. Finally, a lattice in~$V(\A_f)$ is a
product of lattices in~$V(\Q_p)$ for each~$p$ that are equal
to~$V(\Z_p)$ for all~$p$ but finitely many. Recall that the
local-global principle for lattices holds: the map
$\Lambda\mapsto \widehat{\Lambda} = \Lambda\otimes\Zhat$ is a
bijection between lattices in~$V$ and lattices in~$V(\A_f)$, and its
inverse is intersection with~$V(\Q)$. The assumption that~$K$
stabilizes~$\widehat{\Lambda}_0$ does not imply a loss of generality,
because every compact open subgroup of~$G(\A_f)$ stabilizes some
lattice in~$V(\A_f)$.

To complete the setup, let~$\Order$ be the largest order in~$B$
stabilizing~$\Lambda_0$.  We keep the notation of~§\ref{sub:pel}: for
every~$c\in \CC$, we write~$\Gamma_c = G(\Q)_+\cap cKc^{-1}$, and we
denote by~$\Scal_c = \Gamma_c\backslash X_+$ the connected component
of~$\Sh_K(\C)$ associated with~$c$.


We define a \emph{polarized lattice} to be a pair~$(\Lambda,\phi)$
where~$\Lambda$ is a free~$\Z$-module of finite rank
and~$\phi\from \Lambda\times\Lambda\to\Z$ is a nondegenerate
alternating form. Given a polarized lattice~$(\Lambda,\phi)$, we can
extend~$\phi$ to the $\Q$-vector space~$\Lambda\otimes\Q$, and we define
\begin{displaymath}
  \Lambda^\perp = \{v\in \Lambda\otimes\Q\st \forall w\in \Lambda,\ \phi(v,w)\in\Z\}.
\end{displaymath}
Then~$\Lambda^\perp/\Lambda$ is a finite abelian group called the
\emph{polarization type} of~$(\Lambda,\phi)$. We say that~$\phi$ is a
\emph{principal polarization} on~$\Lambda$ if~$\Lambda^\perp=\Lambda$.


\paragraph{A modular interpretation in terms of lattices.} Using the
data above, we define a standard polarized lattice for every connected
component of~$\Sh_K(\C)$ as follows.

\begin{defn}
  \label{def:standard-lattice}
  For each~$c\in \CC$, we define
  \begin{displaymath}
    \widehat{\Lambda}_c = c(\widehat{\Lambda}_0) \quad\text{and}\quad
    \Lambda_c = \widehat{\Lambda}_c\cap V(\Q).
  \end{displaymath}
  The action of~$c$, or any other element of~$G(\A_f)$, on adelic
  lattices is easily defined locally at each prime.  Since~$c$
  respects the action of~$B$ on~$V(\A_f)$, the order~$\Order$ is again
  the stabilizer of~$\widehat{\Lambda}_c$, and thus
  of~$\Lambda_c$. Let $\lambda_c\in\Q_+^\times$ be such that the
  nondegenerate alternating form~$\psi_c = \lambda_c\psi$ satisfies
  $\psi_c(\Lambda_c\times\Lambda_c)=\Z$. We call $(\Lambda_c,\psi_c)$
  with its structure of~$\Order$-module the \emph{standard polarized
    lattice} associated with~$(\Lambda_0, c)$.
\end{defn}


Choose~$c\in \CC$, and let $(\Lambda_c,\psi_c)$ be the standard
polarized lattice associated with~$(\Lambda_0,c)$. We consider tuples
$(\Lambda,x,\iota,\phi,\eta K)$ where
\begin{itemize}
\item $\Lambda$ is a free~$\Z$-module of rank\, $\dim V$,
\item $x\in \End(\Lambda\otimes\R)$ is a complex structure
  on~$\Lambda\otimes\R$,
\item $\iota$ is an embedding~$\Order\embed \End_\Z(\Lambda)$,
\item $\phi\from\Lambda\times\Lambda\to\Z$ is a nondegenerate
  alternating $\Z$-bilinear form on~$\Lambda$,
\item $\eta K$ is a $K$-orbit of $\Zhat$-linear isomorphisms of
  $\Order$-modules~$\widehat{\Lambda}_0 \to \Lambda\otimes\Zhat$,
\end{itemize}
satisfying the following condition of compatibility
with~$(\Lambda_c,\psi_c)$:
\begin{itemize}
\item[($\star$)] There exists an isomorphism of $\Order$-modules
  $a\from \Lambda \to \Lambda_c$, carrying~$\eta K$ to~$c K$ and~$x$
  to an element of~$X_+$, such that
  \begin{displaymath}
    \exists \zeta \in \mu(\Gamma_c),\ \forall u, v\in\Lambda,\ \phi(u,v) = \psi_c\bigl(\zeta a(u),
    a(v)\bigr).
  \end{displaymath}
\end{itemize}
For short, we will call such a tuple a \emph{lattice with PEL
  structure defined by $(\Lambda_0,c)$}, or simply a \emph{lattice
  with PEL structure} when the dependency on~$(\Lambda_0,c)$ is
understood.


An \emph{isomorphism} between lattices with PEL structure
$(\Lambda, x,\iota,\phi,\eta K)$
and~$(\Lambda', x', \iota', \phi', \eta' K)$ is an isomorphism of
$\Order$-modules $f\from \Lambda\to\Lambda'$ that sends~$x$ to~$x'$,
sends~$\eta K$ to~$\eta' K$, and such that
$\phi(u,v) = \phi'\bigl(\zeta f(u), f(v)\bigr)$ for
some~$\zeta\in\mu(\Gamma_c)$.

For every lattice with PEL structure $(\Lambda,x,\iota,\phi,\eta K)$,
the compatibility condition~($\star$) implies in particular that the
complex structure~$x$ is positive for~$\phi$, the adjunction
involution defined by~$\phi$ coincides with~$*$ on~$B$, the action
of~$B$ on~$\Lambda\otimes\Q$ leaves the complex structure~$x$
invariant, and the polarized lattices~$(\Lambda,\phi)$
and~$(\Lambda_c,\psi_c)$ have the same polarization type.


\begin{prop}
  \label{prop:modular-int-lattices} Let~$c\in \CC$, and
  let~$\mathcal{Z}_c$ be the set of isomorphism classes of lattices
  with PEL structure defined by~$(\Lambda_0,c)$. Then the map
  \begin{displaymath}
    \begin{matrix}
      \mathcal{Z}_c & \longrightarrow & \Scal_c \\
      (\Lambda,x,\iota,\phi,\eta K) &\longmapsto &[a x a^{-1},\, c]
      &\text{ where~$a$ is as in~$(\star)$}
    \end{matrix}
  \end{displaymath}
  is well-defined and bijective. The inverse map is
  \begin{displaymath}
    [x, c] \mapsto (\Lambda_c, x, \iota, \psi_c, cK).
  \end{displaymath}
   where~$\iota$ is the natural action of~$\Order$ on~$\Lambda_c$.
\end{prop}

\begin{proof}
  The proof is direct and omitted; the details are similar to
  \cite[Prop.~6.3]{milne_IntroductionShimuraVarieties2005}.
\end{proof}


\paragraph{A modular interpretation in terms of isomorphism classes of
  abelian varieties.}
Giving an abelian variety~$A$ over~$\C$ is the same as giving the
lattice~$\Lambda = H_1(A,\Z)$ and a complex structure on the universal
covering~$\Lambda\otimes\R$ of~$A$. Under this identification,
endomorphisms of~$A$ correspond to endomorphisms of $\Lambda$ that
respect the complex structure.  Moreover, giving a polarization on~$A$
is the same as giving a nondegenerate alternating
form~$\phi\from \Lambda\times\Lambda\to\Z$ such that the bilinear form
$(u,v)\mapsto \phi(u, iv)$ is symmetric and positive definite. The
\emph{polarization type} of~$A$ is the polarization type
of~$(\Lambda,\phi)$.

Recall that for every prime number~$p$, the Tate module~$T_p(A)$
is defined as the projective limit of the torsion subgroups~$A[p^n]$
as~$n$ tends to infinity:
\begin{displaymath}
  T_p(A) = \varprojlim A[p^n] = \varprojlim \Lambda/p^n\Lambda = \Lambda\otimes \Z_p.
\end{displaymath}
Therefore $\Lambda\otimes\Zhat$ is canonically isomorphic to the
global Tate module~$\widehat{T}(A)$ of~$A$, defined as
\begin{displaymath}
  \widehat{T}(A) = \prod_{p \text{ prime }} T_p(A).
\end{displaymath}

Fix~$c\in\CC$, and let~$(\Lambda_c,\psi_c)$ be the standard polarized
lattice associated with~$(\Lambda_0,c)$. We define a \emph{complex
  abelian variety with PEL structure defined by~$(\Lambda_0,c)$} to be
a tuple~$(A,\phi,\iota,\eta K)$ where
\begin{itemize}
\item $(A,\phi)$ is a complex polarized abelian variety of
  dimension~$\dim V$,
\item $\iota$ is an embedding~$\Order\embed\End(A)$,
\item $\eta K$ is a $K$-orbit of $\Zhat$-linear isomorphisms of
  $\Order$-modules~$\widehat{\Lambda}_0 \to \widehat{T}(A)$,
\end{itemize}
satisfying the following condition of compatibility with~$(\Lambda_c,\psi_c)$:
\begin{itemize}
\item[($\star\star$)] There exists an isomorphism of $\Order$-modules
  $a\from H_1(A,\Z) \to \Lambda_c$, carrying~$\phi$ to~$\psi_c$,
  carrying~$\eta K$ to~$c K$, and such that the complex structure
  induced by~$a$ on~$V(\R)$ belongs to~$X_+$.
\end{itemize}

If $(A,\phi,\iota,\eta K)$ is a complex abelian variety with PEL
structure defined by~$(\Lambda_0,c)$, then condition~($\star\star$)
implies that~$A$ and~$(\Lambda_c,\psi_c)$ have the same polarization
type, and that the Rosati involution on $\End(A)\otimes\Q$ (which is
adjunction with respect to~$\phi$) restricts to~$*$ on~$B$.


An \emph{isomorphism} between complex abelian varieties with PEL
structure $(A,\phi,\iota,\eta K)$ and $(A',\phi',\iota',\eta' K)$ is
an isomorphism of complex polarized abelian varieties
$f\from(A,\phi)\to(A,'\phi')$ respecting the action of~$\Order$ and
sending $\eta K$ to~$\eta ' K$.

The difference with the setting of
Proposition~\ref{prop:modular-int-lattices} is that isomorphisms of
complex abelian varieties with PEL structure must respect the
polarizations exactly, rather than up to an element
of~$\mu(\Gamma_c)$.  In general,~$\mu(\Gamma_c)\neq \{1\}$, but there
is the following workaround. If~$\eps\in F^\times$ lies in the center
of~$B$, then multiplication by~$\eps$ defines an element in the center
of~$G(\Q)$. Therefore it makes sense to define
\begin{displaymath}
  \mathcal{E}_K = \{\eps \in F^\times\st
  \eps\in K \} = \{\eps\in F^\times \st\eps\in\Gamma_c\}, \quad \text{for every } c\in G(\A_f).
\end{displaymath}


\begin{prop}
  \label{prop:modular-int-AV} Let~$c\in \nolinebreak\CC$, and
  let~$(\Lambda_c,\psi_c)$ be the standard polarized lattice
  associated with~$(\Lambda_0,c)$. If
  $\mu(\mathcal{E}_K) = \mu(\Gamma_c)$, then the map
  \begin{displaymath}
    [x,c]\longmapsto \bigl(V(\R)/\Lambda_c, \psi_c, \iota, cK\bigr),
  \end{displaymath}
  where~$V(\R)$ is seen as a complex vector space via~$x$, and~$\iota$
  is the action of~$\Order$ on~$V(\R)/\Lambda_c$ induced by the action
  of~$B$ on~$V(\R)$, is a bijection between~$\Scal_c$ and the set of
  isomorphism classes of complex abelian varieties with PEL structure
  defined by~$(\Lambda_0,c)$.
\end{prop}

\begin{proof}
  When defining~$\mathcal{Z}_c$ as in
  Proposition~\ref{prop:modular-int-lattices}, we can
  impose~$\zeta = 1$ in condition~$(\star)$ and strengthen the notion
  of isomorphism between lattices with PEL structure to respect the
  polarizations exactly. Indeed, multiplying the isomorphism~$a$
  by~$\eps\in \mathcal{E}_K$ leaves everything invariant except the
  alternating form, which is multiplied by~$\mu(\eps)$. The result
  follows then from the equivalence of categories between lattices and
  complex abelian varieties outlined above.
\end{proof}


\begin{rem}
  The group~$\mu(\mathcal{E}_K)$ always has finite index
  in~$\mu(\Gamma_c)$. Indeed, if~$\Z_{F_0}^\times$ denotes the unit
  group of~$F_0$, then
  \begin{displaymath}
    \mu(\mathcal{E}_K)\subset \mu(\Gamma_c) \subset \Z_{F_0}^\times
  \end{displaymath}
  and~$\mu(\mathcal{E}_K)$ contains a subgroup of finite index
  in~$\Z_{F_0}^\times$, namely all the squares of elements
  in~$\Z_{F_0}^\times \cap K$.  By
  \cite[Thm.~1]{chevalley_DeuxTheoremesArithmetique1951}, there exists
  a compact open subgroup~$M$ of~$\mu(K)$ such that
  $\Z_{F_0}^\times \cap M = \mu(\mathcal{E}_K)$. Define
  $K' = K\cap \mu^{-1}(M)$. Then~$\mathcal{E}_{K'} = \mathcal{E}_K$,
  and for every~$c\in G(\A_f)$, we have
  \begin{displaymath}
    G(\Q)_+\cap c K' c^{-1} = \{\gamma\in \Gamma_c \st \mu(\gamma) \in \mu(\mathcal{E}_K)\}.
  \end{displaymath}
  Therefore the hypothesis of Proposition~\ref{prop:modular-int-AV}
  will be satisfied for the smaller level subgroup~$K'$.
\end{rem}


When considering the classical modular curves as Shimura varieties
associated with the reductive group~$G = \GL_2$ acting on~$V = \Q^2$,
we can take~$\Lambda_0 = \Z^2$ and~$\psi = \tmat{0}{1}{-1}{0}$. Then
Proposition~\ref{prop:modular-int-AV} applies, and we let the reader
check that we recover the usual modular interpretation of modular
curves in terms of complex elliptic curves with level structure.

\subsection{Modular forms on PEL Shimura varieties}
\label{sub:forms}

Our definition of modular equations will involve choices of
coordinates on connected components of PEL Shimura varieties. These
coordinates, also called modular functions, are obtained as quotients
of modular forms. This section briefly presents modular forms on PEL
Shimura varieties without going into technical details.

Let~$(G,X_+)$ be a PEL datum, and let~$K_\infty\subset G(\R)_+$ be the
stabilizer of a fixed complex structure~$x_0\in X_+$. Attached to this
data is a certain canonical character
of~$K_\infty$~\cite[§1.8]{baily_CompactificationArithmeticQuotients1966},
denoted by~$\rho\from K_\infty\to\C^\times$. Let~$K$ be a compact open
subgroup of~$G(\A_f)$. A \emph{modular form} of weight~$w\in\Z$ on
$\Sh_K(G,X_+)(\C)$ is a function
\begin{displaymath}
  f\from G(\Q)_+ \backslash \bigl(G(\R)_+ \times G(\A_f)\bigr) / K \to \C
\end{displaymath}
that satisfies suitable growth and holomorphy
conditions \cite[Prop.~3.2]{milne_CanonicalModelsMixed1990}, and
such that
\begin{displaymath}
  \forall x\in G(\R)_+,\ \forall g\in G(\A_f),\ \forall k_\infty\in
  K_\infty,\ 
  f([xk_\infty, g]) = \rho(k_\infty)^w f([x,g]).
\end{displaymath}
The weight of~$f$ is denoted by~$\wt(f)$. We also say that $f$ is
\emph{of level $K$}.

Let~$\Scal$ be a connected component of~$\Sh_K(\C)$, or a union of
these, and let~$L$ be its field of definition. A \emph{modular form of
  weight~$w$ on~$\Scal$} is the restriction of a modular form of
weight~$w$ on~$\Sh_K(\C)$ to the preimage of~$\Scal$ in
$G(\Q)_+\backslash \bigl(G(\R_+) \times G(\A_f)\bigr) / K$ by the
natural projection. There is a canonical notion of modular forms
on~$\Scal$ being defined over~$L$
\cite[Chap.~III]{milne_CanonicalModelsMixed1990}.  A \emph{modular
  function} on~$\Scal$ is the quotient of two modular forms of the
same weight, the denominator being nonzero on each connected component
of~$\Scal$.

The following result is well known; since we did not find a precise
reference in the literature, we present a short proof.

\begin{thm}
  \label{thm:mf-algebraic} Let~$\Scal$ be a connected component of the
  Shimura variety~$\Sh_K(\C)$, and let~$L$ be its field of definition.
  Then the graded $L$-algebra of modular forms on~$\Scal$ defined
  over~$L$ is finitely generated, and there exists a weight~$w\geq 1$
  such that modular forms of weight~$w$ defined over~$L$ realize a
  projective embedding of~$\Scal$. Every element of the function
  field~$L(\Scal)$ is a quotient of two modular forms of the same
  weight defined over~$L$.
\end{thm}

\begin{proof}
  Choose an element~$c\in\CC\subset G(\A_f)$ defining the connected
  component~$\Scal$, so that $\Scal = \Gamma_c\backslash X_+$ where
  $\Gamma_c = G(\Q)_+\cap c K c^{-1}$. Assume first that the level
  subgroup~$K$ of~$G(\A_f)$ is sufficiently small, so that~$\Gamma_c$
  is torsion-free. Then, by the Baily--Borel theorem
  \cite[Thm.~10.11]{baily_CompactificationArithmeticQuotients1966},
  there exists an ample line bundle~$\Mcal_\C$ on~$\Scal$ such that
  for every~$w\geq 1$, the algebraic sections
  of~$\Mcal_\C^{\otimes w}$ are exactly the modular forms of
  weight~$w$ on~$\Scal$.

  In fact, $\Mcal_\C$ is the inverse determinant of the tangent bundle
  on~$\Scal$
  \cite[Prop.~7.3]{baily_CompactificationArithmeticQuotients1966}.
  Since~$\Scal$ has a model over~$L$, there is a line bundle~$\Mcal$
  on~$\Scal$ defined over~$L$ such that
  $\Mcal \otimes_L \C = \Mcal_\C$. This is a particular case of a
  general result on the rationality of automorphic vector bundles
  \cite[Chap.~III, Thm.~4.3]{milne_CanonicalModelsMixed1990}. For
  every $w\geq 1$, the $L$-vector space modular forms of weight~$w$
  on~$\Scal$ defined over~$L$ is~$H^0(\Scal, \Mcal^{\otimes
    w})$. Since~$\Mcal\otimes_L\C$ is ample, $\Mcal$ is ample too, and
  this implies the conclusions of the theorem.

  In general, we can always find a level subgroup~$K'$ of finite index
  in~$K$ such that the arithmetic subgroups $G(\Q)_+\cap c K' c^{-1}$
  for $c\in G(\A_f)$ are torsion free
  \cite[Prop.~3.5]{milne_IntroductionShimuraVarieties2005}, and we can
  assume that~$K'$ is normal in~$K$. Let~$\Scal'$ be a connected
  component of~$\Sh_{K'}(\C)$ lying over~$\Scal$, and let~$L'$ be its
  field of definition. Then the conclusions of the theorem hold
  for~$\Scal'$. We can identify the modular forms on~$\Scal$ defined
  over~$L$ with the modular forms on~$\Scal'$ defined over~$L'$ that
  are invariant under the action of a subgroup of~$K/K'$. Therefore
  the conclusions of the theorem also hold for~$\Scal$ by Noether's
  theorem \cite{noether_EndlichkeitssatzInvariantenEndlicher1915} on
  invariants under finite groups.
\end{proof}

We can also consider modular forms that are symmetric under certain
automorphisms of~$\Sh_K$. Let~$\Sigma$ be a finite group of
automorphisms of~$V$ as a $\Q$-vector space that leaves the symplectic
form~$\psi$ invariant, and also acts on~$B$ in such a way that
\begin{displaymath}
  \forall u\in V,\ \forall b\in B,\ \forall \sigma\in\Sigma, \
  \sigma(bu) = \sigma(b)\sigma(u).
\end{displaymath}
This implies that the elements of~$\Sigma$ commute with the
involution~$*$, and hence leave~$F_0$ stable. Under these assumptions,
each~$\sigma\in \Sigma$ induces an automorphism of~$G$ defined
over~$\Q$, also denoted by~$\sigma$. Assume further that these
automorphisms leave~$G(\R)_+$, $X_+$, $K$, $K_\infty$, $\nu$ and the
character~$\rho$ invariant. Then~$\Sigma$ can be seen as a finite
group of automorphisms of~$\Scal$, and one can check as
in~\cite[Thm.~13.6]{milne_IntroductionShimuraVarieties2005} that these
automorphisms are defined over~$L$.  Then for every modular form~$f$
of weight~$w$ on~$\Scal$ defined over~$L$, and
every~$\sigma\in\Sigma$, the function
\begin{displaymath}
  \act{\sigma}{f}\defby [x,g] \mapsto f([\sigma^{-1}(x), \sigma^{-1}(g)])
\end{displaymath}
is a modular form of weight~$w$ on~$\Scal$ defined over~$L$. We say
that~$f$ is \emph{symmetric} under~$\Sigma$ if $\act{\sigma}{f} = f$ for
every $\sigma\in\Sigma$.

\begin{prop}
  \label{prop:mf-symmetric}
  Let~$\Sigma$ be a finite group of automorphisms of~$G$ as
  above. Then the graded $L$-algebra of symmetric modular forms
  on~$\Scal$ defined over~$L$ is finitely generated, and every
  symmetric modular function on~$\Scal$ defined over~$L$ is the
  quotient of two symmetric modular forms of the same weight defined
  over~$L$.
\end{prop}

\begin{proof}
  This results from Theorem~\ref{thm:mf-algebraic} and another
  application of Noether's theorem.
\end{proof}

\subsection{Hecke correspondences}
\label{sub:hecke}


We fix a PEL datum~$(G,X_+)$ as above, as well as a compact open
subgroup~$K\subset G(\A_f)$.  Let~$\delta \in G(\A_f)$, and let
$K' = K\cap\,\delta K\delta^{-1}$. Consider the diagram
\begin{equation}
  \label{diag:hecke}
  \begin{tikzcd}
    \Sh_{K'}(\C) \ar[d, "p_1"] \ar[rr, "R(\delta)"] &&
    \Sh_{\delta^{-1} K'\delta}(\C) \ar[d, "p_2"] \\
    \Sh_K(\C) && \Sh_K(\C)
  \end{tikzcd}
\end{equation}
where the map~$R(\delta)$ is $[x,g]\mapsto [x,g \delta]$, and~$p_1$
and~$p_2$ are the natural projections. This diagram defines a
correspondence~$H_\delta$ in $\Sh_K\times \Sh_K$, called the
\emph{Hecke correspondence} of level~$\delta$, consisting of all pairs
of the form~$\bigl(p_1(x), p_2(R(\delta)x)\bigr)$ for $x\in
\Sh_{K'}$. Hecke correspondences are algebraic: the
diagram~\eqref{diag:hecke} is the analytification of a diagram
existing at the level of algebraic varieties. Moreover, Hecke
correspondences are defined over the reflex field
\cite[Thm.~13.6]{milne_IntroductionShimuraVarieties2005}.


We define the \emph{degree} of~$H_\delta$ to be the index
\begin{displaymath}
  \Heckedeg(\delta) = [K:K'] = [K:K\cap\,\delta K\delta^{-1}].
\end{displaymath}
This index is finite as both~$K$ and~$K'$ are compact open subgroups
of~$G(\A_f)$, and is the degree of the map~$\Sh_{K'}\to \Sh_K$.  One
can also consider~$H_\delta$ as a map from~$\Sh_K$ to its
$\Heckedeg(\delta)$-th symmetric power, sending~$z\in \Sh_K$ to the set
$\{z'\in \Sh_K\st (z,z')\in H_\delta\}$.

It is easy to see how~$H_\delta$ behaves with respect to connected
components: if~$z$ lies in the connected component indexed by
$t\in T(\A_f)$, then its images lie in the connected component indexed
by~$t\, \nu(\delta)$.

We call the Hecke correspondence~$H_\delta$ \emph{absolutely
  irreducible} if for every connected component~$\Scal$ of~$\Sh_K(\C)$
with field of definition~$L$, the preimage of~$\Scal$ in~$\Sh_{K'}$ is
absolutely irreducible as a variety defined over~$L$ (or equivalently,
connected as a variety over~$\C$). A sufficient condition
for~$H_\delta$ to be absolutely irreducible is
that~$\nu(K') = \nu(K)$.


\paragraph{Modular interpretation of Hecke correspondences.}
In the modular interpretation, Hecke correspondences describe
isogenies of a certain type between polarized abelian varieties.
Let~$\Lambda_0$, $\CC$, and~$\Order$ be as in~§\ref{sub:modular-int},
and write
\begin{displaymath}
  K = \bigsqcup_{i=1}^{\Heckedeg(\delta)} \kappa_i K',
\end{displaymath}
where~$\kappa_i\in G(\A_f)$ for each~$1\leq i\leq \Heckedeg(\delta)$.
Let~$c\in \CC$, denote by~$\Scal_c$ the connected component
of~$\Sh_K(\C)$ indexed by~$c$, and consider the lattice with PEL
structure~$(\Lambda_c, x, \iota, \psi_c, cK)$ associated with a
point~$[x,c]\in \Scal_c$ by
Proposition~\ref{prop:modular-int-lattices}.


In order to construct the lattices associated with~$[x,c]$ via the
Hecke correspondence~$H_\delta$, we partition the orbit~$c K$ into the
$K'$-orbits~$c \kappa_i K'$ for~$1\leq i\leq \Heckedeg(\delta)$. Each
element $c \kappa_i \delta\in G(\A_f)$ is then a $\Zhat$-linear
embedding of $\Order$-modules $\widehat{\Lambda}_0\embed V(\A_f)$; it
is well defined up to right multiplication by~$\delta^{-1}K'\delta$,
hence by~$K$. Let~$\Lambda_i\subset V(\Q)$ be the lattice such
that~$\Lambda_i\otimes\Zhat$ is the image of this embedding. There is
still a natural action of~$\Order$ on~$\Lambda_i$. The
decomposition~$c \kappa_i\delta K = q_i c' K$, with~$q_i\in G(\Q)_+$
and $c'\in \CC$, is well defined, and the element~$c'$ does not depend
on~$i$.

\begin{prop}
  \label{prop:hecke} Let~$\delta\in G(\A_f)$,
  let~$z=[x,c]\in \Scal_c$, and construct~$\Lambda_i,q_i,c'$ as above.
  Then the image of~$z$ under the Hecke correspondence~$H_\delta$ in
  the modular interpretation of Proposition~\ref{prop:modular-int-AV}
  is given by the~$\Heckedeg(\delta)$ isomorphism classes of tuples
  with representatives
  \begin{displaymath}
    \Bigl (\Lambda_i, x,
    \dfrac{\lambda_{c'}}{\lambda_c}\psi_c
    \bigl(\mu(q_i^{-1})\,\cdot\,,\cdot \bigr),
    c \kappa_i\delta K \Bigr)
    \qquad \text{for}\quad 1\leq i\leq \Heckedeg(\delta).
  \end{displaymath}
\end{prop}

\begin{proof}
  By construction, the images of~$[x,c]$ via the Hecke correspondence
  are the points~$[q_i^{-1} x, c']$ of~$\Sh_K(\C)$.  The
  relation~$c \kappa_i \delta K = q_i c' K$ shows that the
  map~$q_i^{-1}$ sends the lattice~$\Lambda_i$ to~$\Lambda_{c'}$. This
  map also respects the action of~$\Order$, and sends the complex
  structure~$x$ to~$q_i^{-1}x$. Finally, it sends the
  polarization~$(u,v)\mapsto \psi_c(u,v)$ on~$\Lambda_i$ to
  $(u,v)\mapsto \psi_c\bigl(\mu(q_i)u, v\bigr)$ on~$\Lambda_{c'}$.
\end{proof}


After multiplying~$\delta$ by a unique suitable element
in~$\Q^\times_+$, which does not change~$H_\delta$, we can assume that
$\delta(\widehat{\Lambda}_0) \subset \widehat{\Lambda}_0$ and
$\delta(\widehat{\Lambda}_0)\not\subset p\widehat{\Lambda}_0$ for
every prime~$p$; we say that~$\delta$ is \emph{normalized} with
respect to~$\Lambda_0$. In this case, we define the \emph{isogeny
  degree} of~$H_\delta$ as the unique integer~$\isogdeg(\delta)\geq 1$
such that~$\isogdeg(\delta)^{-1}\det(\delta)$ is a unit in~$\Zhat$. In
other words,
\begin{displaymath}
  \isogdeg(\delta) = \# \bigl(
  \widehat{\Lambda}_0/\delta(\widehat{\Lambda}_0) \bigr).
\end{displaymath}
For a general~$\delta\in G(\A_f)$, we set
$\isogdeg(\delta) = \isogdeg(\lambda\delta)$
where~$\lambda\in\Q_+^\times$ is chosen such that~$\lambda\delta$ is
normalized with respect to~$\Lambda_0$.

\begin{cor}
  \label{cor:hecke-isog}
  Let~$\delta\in G(\A_f)$. Then, in the modular interpretation of
  Proposition~\ref{prop:modular-int-AV}, the Hecke
  correspondence~$H_\delta$ sends an abelian variety~$A$ with PEL
  structure to~$\Heckedeg(\delta)$ abelian
  varieties~$A_1,\ldots, A_{\Heckedeg(\delta)}$ such that for
  every~$1\leq i\leq \Heckedeg(\delta)$, there exists an isogeny $A_i\to A$ of
  degree~$\isogdeg(\delta)$.
\end{cor}

\begin{proof}
  We can assume that~$\delta$ is normalized with respect
  to~$\Lambda_0$. Then, in the result of Proposition~\ref{prop:hecke},
  each lattice~$\Lambda_i$ for $1\leq i\leq \Heckedeg(\delta)$ is a
  sublattice of~$\Lambda_c$ endowed with the same complex
  structure~$x$. Moreover, for every~$1\leq i\leq \Heckedeg(\delta)$,
  we have
  $\Lambda_c/\Lambda_i\simeq
  \widehat{\Lambda}_0/\delta(\widehat{\Lambda}_0)$, so the index of
  each~$\Lambda_i$ in~$\Lambda_c$ is~$\isogdeg(\delta)$.
\end{proof}

\paragraph{A relation between degrees}
For later purposes, we state an inequality
relating~$\Heckedeg(\delta)$ and a power
of~$\isogdeg(\delta)$. Since~$K\subset G(\A_f)$ is open, there exists
a smallest integer~$N\geq 1$ such that
\begin{displaymath}
  \bigl \{g\in G(\A_f) \cap \GL(\widehat{\Lambda}_0) \st g = 1\ \mathrm{mod}\
  N\widehat{\Lambda}_0 \bigr\} \subset K,
\end{displaymath}
that we call the~\emph{level} of~$K$ with respect
to~$\widehat{\Lambda}_0$.

\begin{prop}
  \label{prop:d-l-relation}
  There exists a constant~$C$ depending on~$K$ and~$\Lambda_0$ such
  that for every $\delta\in G(\A_f)$, we have
  $\Heckedeg(\delta)\leq C\,\isogdeg(\delta)^{(\dim V)^2}$. We can take
  $C = N^{(\dim V)^2}$, where~$N$ is the level of~$K$ with respect
  to~$\widehat{\Lambda}_0$.
\end{prop}

\begin{proof}
  We can assume that~$\delta$ is normalized with respect
  to~$\widehat{\Lambda}_0$.  Then~$K\cap\,\delta K \delta^{-1}$
  contains all the elements
  $g\in G(\A_f) \cap \GL(\widehat{\Lambda}_0)$ that are the identity
  modulo $\widehat{\Lambda} = \isogdeg(\delta) N \widehat{\Lambda}_0$. In
  other words we have a morphism
  $K\to \GL(\Lambda_0/N\isogdeg(\delta)\Lambda_0)$ whose kernel is
  contained in~$K\cap\delta K\delta^{-1}$.  This yields the result
  since
  $\#\GL(\Lambda_0/N\isogdeg(\delta)\Lambda_0)\leq (N\isogdeg(\delta))^{(\dim
    V)^2}.$
\end{proof}

\begin{rem}
  The upper bound on~$\Heckedeg(\delta)$ given in
  Proposition~\ref{prop:d-l-relation} is far from optimal in many
  cases: for instance, if~$\delta$ is normalized with respect
  to~$\widehat{\Lambda}_0$, if $\isogdeg(\delta)$ is prime to~$N$, and
  if moreover~$\delta$ normalizes the image of~$K$
  in~$\GL(\Lambda_0/N\Lambda_0)$, then
  $\Heckedeg(\delta) \leq \isogdeg(\delta)^{(\dim V)^2}$. But in
  general, the level of~$K$ does enter into account. As an example,
  take~$G = \GL_2$, $\delta = \tmat{0}{1}{1}{0}$, and
  \begin{displaymath}
    K = \bigr\{\tmat{a}{b}{c}{d}\in \GL_2(\Zhat)
    \st a = d = 1\text{ mod } N \text{ and } c = 0 \text{ mod } N\bigr\}.
  \end{displaymath}
  Then $\Heckedeg(\delta) = N$ even though~$\isogdeg(\delta)=1$. In the modular
  interpretation, the Hecke correspondence~$H_\delta$ has the effect
  of forgetting the initial~$K$-level structure entirely.
\end{rem}

\section{Modular equations on PEL Shimura varieties}
\label{sec:modeq}

This section presents a general definition of modular equations on PEL
Shimura varieties, generalizing three examples mentioned in the
introduction: the elliptic modular polynomials, and the modular
equations of Siegel and Hilbert type for abelian surfaces (see
§\ref{sub:siegel} and~§\ref{sub:hilbert}).

\subsection{The example of elliptic modular polynomials}
\label{sub:Phiell}

Elliptic modular polynomials are the simplest example of modular
equations. They are usually defined in terms of classical modular
forms~\cite[§11.C]{cox_PrimesFormNy2013}. In order to motivate the
general definition, we translate this definition in the adelic
language.

The underlying PEL datum is obtained by taking~$V = \Q^2$,
$\psi = \tmat{0}{1}{-1}{0}$, and~$B=\Q$ with~$*$ the trivial
involution. Then~$G = \GL_2$, and~$G(\Q)_+$ consists of all
rational~$2\times 2$ matrices with positive determinant. We
take~$\Lambda_0 = \Z^2$ and $K = \GL_2(\Zhat)$, so that~$\Sh_K(\C)$
has only one connected component~$\Scal$ (indexed by the identity
matrix) and the maximal order of~$B$ stabilizing~$\Lambda_0$
is~$\Order = \Z$. If we take the complex
structure~$x_0 = \tmat{0}{1}{-1}{0}$ as a base point, then~$X_+$ is
naturally identified with the Poincaré upper half plane~$\Half_1$,
with~$x_0$ corresponding to~$i\in \Half_1$. Then~$\Scal$ is identified
with the modular curve~$\SL_2(\Z)\backslash \Half_1$, and modular
forms on~$\Scal$ in the sense of~§\ref{sub:forms} correspond exactly
to modular forms of level~$\SL_2(\Z)$ on~$\Half_1$ in the classical
sense. The reflex field~$E(G,X_+)$ is equal to~$\Q$ in this case, and
the~$j$-invariant realizes an isomorphism between~$\Sh_K$ and the
affine line~$\A^1_\Q$; in particular~$j$ generates the function field
of~$\Scal$ over~$\Q$.

Let~$\ell$ be a prime number. Then the function on~$\Half_1$ given
by~$\tau\mapsto j(\tau/\ell)$ is invariant under the following
congruence subgroup of~$\SL_2(\Z)$:
\begin{displaymath}
  \Gamma^0(\ell) = \bigl\{\tmat{a}{b}{c}{d}\in \SL_2(\Z)\st b=0 \text{ mod }\ell\bigr\}.
\end{displaymath}
Therefore, the coefficients of the polynomial
\begin{displaymath}
  P_\ell(\tau) = \prod_{\gamma\in \Gamma^0(\ell)\backslash \SL_2(\Z)}
  \Bigl(Y - j(\linv\gamma\tau)\Bigr),\quad \text{for }\tau\in \Half_1
\end{displaymath}
are modular functions of level~$\SL_2(\Z)$. The elliptic modular
polynomial~$\Phi_\ell$ is the unique element of~$\C(X)[Y]$ satisfying
the relation~$\Phi_\ell(j(\tau),Y) = P_\ell(\tau)$ for
every~$\tau\in\Half_1$; actually~$\Phi_\ell\in \Z[X,Y]$. In other
words, we have a map
\begin{equation}
  \label{diag:classical-hecke}
  \Gamma^0(\ell)\backslash\Half_1\to \Scal\times \Scal, \quad \tau\mapsto (\tau, \tau/\ell),
\end{equation}
and the product~$\Scal\times \Scal$ is birational
to~$\Proj^1\times \Proj^1$ via~$(j,j)$. The modular
curve~$\Gamma^0(\ell)\backslash\Half_1$ is birational to its image
in~$\Proj^1\times \Proj^1$, and~$\Phi_\ell$ is an equation of this
image.

Remark that for every~$\tau\in \Half_1$, we have
\begin{displaymath}
  \tau/\ell = \delta^{-1}\tau, \quad\text{where}\quad
  \delta = \tmat{\ell}{0}{0}{1} \in G(\Q)_+.
\end{displaymath}
Therefore, if~$\tau\in \Half_1$ corresponds to a
point~$[x,I_2] \in \Sh_K(\C)$, then~$\tau/\ell$ corresponds to the
point~$[x,\delta]$.
Moreover~$\Gamma^0(\ell) = \SL_2(\Z) \cap \bigl(\delta
\SL_2(\Z)\delta^{-1}\bigr)$. Therefore the
map~\eqref{diag:classical-hecke} is precisely the Hecke
correspondence~$H_\delta$ given in diagram~\eqref{diag:hecke}.

The function~$\tau\mapsto j(\tau/\ell)$ corresponds to the modular
function
\begin{displaymath}
  \begin{matrix}
    j_\delta\defby & G(\Q)_+\backslash\bigl(G(\A_f)\times G(\R)_+\bigr) &\to &\C\\
    & [x,g] & \mapsto & j([x,g\delta]),
  \end{matrix}
\end{displaymath}
which is right-invariant under~$\delta K\delta^{-1}$. Let~$K''$ be a
normal subgroup of finite index in~$K$ contained
in~$K' = K \cap \,\delta K\delta^{-1}$. We let~$K$ act (on the left)
on the set of modular functions of level~$K''$ as follows: if~$k\in K$
and~$f$ is such a function, then we define
\begin{displaymath}
  \act{k}{f}\defby [x,g]\mapsto f([x,gk]).
\end{displaymath}
Since~$K'$ is contained in the stabilizer of~$j_\delta$, the
coefficients of the polynomial
\begin{equation}
  \label{eq:Qell}
  Q_\ell = \prod_{\gamma\in K/K'} \bigl( Y - \act{\gamma}{j_\delta}\bigr)
\end{equation}
are modular functions of level~$K$; the analogue of~$Q_\ell$ in the
classical world is exactly~$P_\ell$, as inversion induces a bijection
between right cosets of~$\Gamma^0(\ell)$ in~$\SL_2(\Z)$ and left
cosets of~$K'$ in~$K$. The general definition of modular equations
involves analogues of the product~\eqref{eq:Qell} for other Hecke
correspondences.

\subsection{General definition of modular equations}
\label{sub:modeq}


Let~$(G,X_+)$ be a PEL datum, let~$K$ be a compact open subgroup of
$G(\A_f)$, and let~$\Sigma$ be a finite group of automorphisms of~$G$
as in~§\ref{sub:forms}. Let~$n$ be the complex dimension of~$X_+$; we
assume that~$n\geq 1$. Let~$\Scal$,~$\Tcal$ be connected components
of~$\Sh_K(G,X_+)(\C)$, and let~$L$ be their field of definition.


To complete the picture, we also need to choose coordinates on~$\Scal$
and~$\Tcal$. Since the field~$L(\Scal)$ of modular functions
on~$\Scal$ has transcendence degree~$n$ over~$L$, the
field~$L(\Scal)^\Sigma$ of modular functions on~$\Scal$ that are
symmetric under~$\Sigma$ also has transcendence degree~$n$
over~$L$. Choose a transcendence basis $(j_1,\ldots, j_n)$
of~$L(\Scal)^\Sigma$ over~$L$, and another symmetric function
$j_{n+1}$ that generates the remaining finite extension, whose degree
is denoted by~$e$. On~$\Scal$, the function~$j_{n+1}$ satisfies a
minimal relation of the form
\begin{equation}
  \label{eq:jn+1}
  E(j_1,\ldots,j_{n+1}) = 0
  \qquad \text{where}\quad E = \sum_{k = 0}^{e}
  E_k(J_1,\ldots,J_n)\,J_{n+1}^{\,k} \in L[J_1,\ldots,J_{n+1}]
\end{equation}
and~$E$ is irreducible. If~$L(\Scal)^\Sigma$ is purely transcendental
over~$L$ (if~$\Sigma=\{1\}$, this means that~$\Scal$ is birational
to~$\Proj^n_L$), then we take~$j_{n+1}=1$, ignore eq.~\eqref{eq:jn+1},
and work with~$n$ invariants only.

We proceed similarly to define coordinates
on~$\Tcal$: no confusion will arise if we also denote them by
$j_1,\ldots, j_{n+1}$. We refer to all the data defined up to now as
the \emph{PEL setting}. Throughout the paper, our constants will
depend on this data only.

Given a PEL setting as above, let~$\delta\in G(\A_f)$ be an adelic
element of~$G$ defining an absolutely irreducible Hecke
correspondence~$H_\delta$ that intersects~$\Scal\times \Tcal$
nontrivially. We want to define explicit polynomials with coefficients
in~$L(\Scal)$, called the \emph{modular equations of level $\delta$},
describing~$H_\delta$ in the product~$\Scal\times \Tcal$. To do this,
we mimic the definition of elliptic modular polynomials in the
language of PEL Shimura varities given in~§\ref{sub:Phiell}. As
in~§\ref{sub:hecke}, we write~$K' = K\cap \, \delta K\delta^{-1}$.


Let~$K''$ be a normal subgroup of finite index in~$K$, contained
in~$K'$, and stabilized by~$\Sigma$. Let~$\Scal''$ be the
preimage of~$\Scal$ in~$\Sh_{K''}(\C)$.  There is a left action
of~$K \rtimes \Sigma$ on the space of modular functions
on~$\Scal''$, given by
\begin{displaymath}
  \act{(k,\sigma)}{f} \defby [x,g] \mapsto \act{\sigma}{f}([x,gk]).
\end{displaymath}
The modular functions that are invariant under~$K'\rtimes\{1\}$
(resp.~$K\rtimes\Sigma$) are exactly the rational functions
on~$H_\delta\cap(\Scal\times\Tcal)$ defined over~$\C$ (resp. the
rational functions on~$\Scal$ defined over~$\C$ and invariant
under~$\Sigma$).  The modular functions
\begin{displaymath}
  j_{i,\delta} \defby [x,g] \mapsto j_i([x,g\delta])
\end{displaymath}
for~$1\leq i\leq n+1$ are defined over~$L$ and generate the function
field of~$H_\delta\cap(\Scal\times\Tcal)$. We define the decreasing
chain of subgroups
\begin{displaymath}
  K \rtimes\Sigma = K_0 \supset K_1 \supset \cdots \supset K_{n+1}
  \supset K'
\end{displaymath}
as follows: for each $1\leq i\leq n+1$, the subgroup~$K_i$ is the
stabilizer of the functions
$j_{1,\delta},\ldots,j_{i,\delta}$. In~§\ref{sub:Phiell}, we
had~$K_0=K$ and~$K_1=K'$.

Galois theory applied to the Galois covering~$\Scal''\to\Scal$ tells
us that for every~$1\leq i\leq n+1$, the
field~$L(j_1,\ldots,j_{n+1},j_{1,\delta},\ldots,j_{i,\delta})$ is the
function field of the preimage of~$\Scal$ in the Shimura
variety~$\Sh_{K_i}$, and consists of all modular functions
on~$\Scal''$ defined over~$L$ that are invariant under~$K_i$. In other
words, we have a tower of function fields:
\begin{displaymath}
  \begin{tikzcd}
    L(j_1,\ldots,j_{n+1},j_{1,\delta},\ldots,j_{n+1,\delta}) =
    L(H_\delta \cap(\Scal\times\Tcal)) \\
    \vdots \ar[u, -, "{\text{degree } d_{n+1}}"] \\
    L(j_1,\ldots,j_{n+1},j_{1,\delta}) \ar[u, -, "{\text{degree }
      d_2}"] \\
    L(\Scal)^\Sigma. \ar[u, -, "{\text{degree } d_1}"]
  \end{tikzcd}
\end{displaymath}
where $d_i = [K_{i-1}:K_i]$ for~$1\leq i\leq n+1$. The modular
equations of level~$\delta$ are defining equations for the successive
extensions in the tower.

\begin{defn}
  \label{def:modeq}
  The \emph{modular equations} of level~$\delta$
  on the product~$\Scal\times \Tcal$ are the tuple
  $(\Psi_{\delta,1}, \Psi_{\delta,2},\ldots, \Psi_{\delta,n+1})$
  defined as follows: for each $1\leq m\leq n+1$,~$\Psi_{\delta,m}$ is
  the multivariate polynomial in the~$m$ variables $Y_1,\ldots, Y_m$
  defined by
  \begin{displaymath}
    \Psi_{\delta,m} = \sum_{\gamma\in K_0/K_{m-1}} \left( 
      \left(
        \prod_{i=1}^{m-1} \prod_{
        \gamma_i}
        \Bigl(Y_i - \act{\gamma_i}{j_{i,\delta}} \Bigr)
      \right)
      \prod_{\gamma_m \in K_{m-1}/K_m} \Bigl(Y_m -
        \act{\gamma\gamma_m}{j_{m,\delta}} \Bigr)
    \right)
  \end{displaymath}
  where the middle product is over all $\gamma_i\in K_0/K_i$ such that
  $\gamma_i = \gamma$ modulo~$K_{i-1}$, but $\gamma_i\neq \gamma$
  modulo~$K_i$. The expression for~$\Psi_{\delta,m}$ makes sense,
  because multiplying~$\gamma$ on the right by an element in~$K_{m-1}$
  only permutes the factors in the last product.
\end{defn}

In the case of the Hecke correspondence considered
in~§\ref{sub:Phiell}, the polynomial~$\Psi_{\delta,1}$ is
precisely~$Q_\ell$. The precise formula is inspired from preexisting
definitions of modular equations for abelian surfaces
\cite{broker_ModularPolynomialsGenus2009,milio_QuasilinearTimeAlgorithm2015,milio_ModularPolynomialsHilbert2020,martindale_HilbertModularPolynomials2020}. We
will return to these examples in~§\ref{sub:siegel}
and~§\ref{sub:hilbert}.

Let us give elementary properties of modular equations. First, we need
a lemma.

\begin{lem}
  \label{lem:all-CC}
  Let~$\gamma,\gamma'\in K_0$ and~$1\leq i\leq n+1$. Assume that
  the equality~$\gamma\cdot j_{i,\delta}= \gamma'\cdot j_{i,\delta}$
  holds on one connected component of~$\Scal''$. Then it
  holds on all connected components of~$\Scal''$.
\end{lem}

\begin{proof}
  Write~$\gamma = (k,\sigma)$ and~$\gamma' = (k',\sigma')$
  where~$k,k'\in K$ and~$\sigma,\sigma'\in \Sigma$.
  Let~$c\in G(\A_f)$ be an adelic element of~$G$ defining the connected
  component~$\Scal$ in~$\Sh_K(\C)$, so
  that~$\Scal = \Gamma_c\backslash X_+$ with
  $\Gamma_c = G(\Q)_+\cap cKc^{-1}$. By assumption, there exists an
  element $g\in G(\A_f)$ such that~$g=c$ in the double quotient space
  $G(\Q)_+\backslash G(\A_f)/K$, and
  \begin{equation}
    \label{eq:equality-one-CC}
    \forall x\in X_+,\ j_{i,\delta}\bigl([\sigma^{-1}(x), \sigma^{-1}(gk)]\bigr)
    = j_{i,\delta}\bigl([\sigma'^{-1}(x), \sigma'^{-1}(gk')]\bigr).
  \end{equation}
  Since~$H_\delta$ is absolutely irreducible, we have
  $G(\Q)_+\backslash G(\A_f)/K = G(\Q)_+\backslash G(\A_f)/K'$. Using
  the description of connected components of a PEL Shimura variety
  in~§\ref{sub:pel}, and the fact that the action~$\Sigma$
  leaves~$\nu$ invariant, we find that there exist
  $\gamma_1,\gamma_2\in G(\Q)_+$ such that $gk = \gamma_1 \sigma(c)$
  mod~$\sigma(K')$ and $gk' = \gamma_2 \sigma'(c)$
  mod~$\sigma'(K')$. Then equation~\eqref{eq:equality-one-CC}
  is equivalent to the following:
  \begin{equation}
    \label{eq:equality-all-CC}
    \forall x\in X_+,\ j_{i,\delta}\bigl([x,c]\bigr)
    = j_{i,\delta}\bigl([\sigma'^{-1}(\gamma_2^{-1}\gamma_1\sigma(x)), c]\bigr).
  \end{equation}
  Note that~$\gamma_2^{-1}\gamma_1$ is well-defined and independent
  of~$g$, up to multiplication on the left by an element
  of~$G(\Q)_+\cap \sigma'(cK'c^{-1})$, and on the right by an element
  of~$G(\Q)_+\cap\sigma(cK'c^{-1})$. Therefore
  equation~\eqref{eq:equality-all-CC} holds for every~$g\in G(\A_f)$
  such that $g=c$ in $G(\Q)_+\backslash G(\A_f)/K$. In other words,
  the equality $\gamma\cdot j_{i,\delta} = \gamma'\cdot j_{i,\delta}$
  holds on every connected component of~$\Scal''$.
\end{proof}


\begin{prop}
  \label{prop:modeq-evaluate}
  Let $1\leq m\leq n+1$, and let $\gamma\in K_0/K_{m-1}$. Then, up to
  multiplication by an element in
  $L(j_1,\ldots,j_{n+1},\gamma\cdot j_{1,\delta},\ldots,\gamma\cdot
  j_{m-1,\delta})^\times$, we have
  \begin{displaymath}
    \Psi_{\delta,m}(\act{\gamma}{j_{1,\delta}}\,,\ldots,\act{\gamma}{j_{m-1,\delta}}\,,
    Y_m) = \prod_{\gamma_m\in K_{m-1}/K_m} \Bigl(Y_m - \act{\gamma\gamma_m}{j_{m,\delta}}\Bigr).
  \end{displaymath}
\end{prop}

\begin{proof}
  By Definition~\ref{def:modeq}, the above equality holds true after
  multiplying the right hand side by
  \begin{displaymath}
    f = \prod_{i=1}^{m-1} \prod_{\substack{\gamma_i\in K_0/K_i\\\gamma_i\neq \gamma\\\gamma_i=\gamma\text{ mod } K_{i-1}}}
    \Bigl(\act{\gamma}{j_{i,\delta}} - \act{\gamma_i}{j_{i, \delta}} \Bigr)
  \end{displaymath}
  The function~$f$ a product of nonzero modular functions
  on~$\Scal''$ defined over~$L$.  In order to show
  that~$f\in L(j_1,\ldots,j_{n+1},\gamma\cdot
  j_{1,\delta},\ldots,\gamma\cdot j_{m-1,\delta})$, we check that~$f$
  is invariant under the action of~$\gamma K_{m-1}\gamma^{-1}$.  By
  definition of the subgroups~$K_i$, no factor of~$f$ is identically
  zero on~$\Scal''$. Therefore~$f$ is invertible by
  Lemma~\ref{lem:all-CC}.
\end{proof}

Let~$1\leq m\leq n+1$.  Proposition~\ref{prop:modeq-evaluate} implies
that up to scaling, the univariate polynomial
$\Psi_{\delta,m}(j_{1,\delta},\ldots,j_{m-1,\delta},Y_m)$ is the
minimal polynomial of the function~$j_{m,\delta}$ over the field
$L(j_1,\ldots,j_{n+1},j_{1,\delta},\ldots,j_{m-1,\delta})$. In other
words, when the multiplicative coefficient in
Proposition~\ref{prop:modeq-evaluate} does not vanish, which is
generically the case,~$\Psi_{\delta,m}$ provides all the possible
values for~$j_{m,\delta}$ once $j_1,\ldots, j_{n+1}$ and
$j_{1,\delta}, \ldots, j_{m-1,\delta}$ are known. In particular,
modular equations vanish on~$H_\delta$ as promised.

We could also define other modular equations~$\Phi_{\delta,m}$ for
which there is true equality in Proposition~\ref{prop:modeq-evaluate},
as in the case of the classical modular polynomial~$\Phi_l$, but they
have a more complicated expression. In practice, using the
polynomials~$\Psi_{\delta,m}$ is more convenient as they are typically
smaller.

\begin{prop}
  \label{prop:modeq-coeffs} Let $1\leq m\leq n+1$. The coefficients
  of~$\Psi_{\delta,m}$ lie in $L(j_1,\ldots,j_{n+1})$. The degree
  of~$\Psi_{\delta,m}$ in~$Y_m$ is~$[K_{m-1}:K_m]$, and for
  each~$1\leq i<m$, the degree of~$\Psi_{\delta,m}$ in~$Y_i$ is at
  most~$[K_{i-1}:K_i] - 1$.
\end{prop}

\begin{proof}
  It is clear from Definition~\ref{def:modeq} that the action of~$K_0$
  leaves~$\Psi_{\delta,m}$ invariant. Hence the coefficients
  of~$\Psi_{\delta,m}$ are rational functions on~$\Scal$ invariant
  under~$\Sigma$ and defined over~$L$, so the first statement
  holds. The second part is obvious.
\end{proof}

In general, using a nontrivial~$\Sigma$ increases the degree of
modular equations. This has a geometric interpretation: modular
equations describe the Hecke correspondence~$H_\delta$ and its
conjugates under~$\Sigma$ simultaneously.

Let~$J_1,\ldots, J_{n+1}$ be indeterminates, and let
$1\leq m\leq n+1$. By the equation~\eqref{eq:jn+1} satisfied
by~$j_{n+1}$ on~$\Scal$, there exists a unique element of the
ring~$L(J_1,\ldots,J_n)[J_{n+1}, Y_1,\ldots, Y_m]$ with degree at
most~$e-1$ in~$J_{n+1}$ which, when evaluated at $J_i = j_i$ for
$1\leq i\leq n+1$, yields~$\Psi_{\delta,m}$. In the sequel, we also
denote it by~$\Psi_{\delta,m}$ for simplicity. Therefore the
coefficients of~$\Psi_{\delta,m}$ will be either functions on~$\Scal$,
i.e.~as elements of~$L(j_1,\ldots,j_{n+1})$, or multivariate rational
fractions in the indeterminates~$J_1,\ldots,J_{n+1}$ that are
polynomial in~$J_{n+1}$ of degree at most~$e-1$, depending on the
context.

\begin{rem}
  \label{rem:simpler-modeqs}
  In several cases, the function~$j_{1,\delta}$ already generates the
  whole extension of function fields, so that
  $K_1 = \cdots = K_{n+1}=K'$,
  \begin{align*}
    \Psi_{\delta,1} &= \prod_{\gamma_1 \in K_0/K'} \bigl(Y_1 - \act{\gamma_1}{j_{1,\delta}}\bigr),
  \end{align*}
  and for every $2\leq m\leq n+1$,
  \begin{equation}
    \label{eq:simpler-Psidelta}
    \Psi_{\delta,m} = \sum_{\gamma\in K_0/K'}
                      \left(\Biggl(\prod_{\gamma_1\neq \gamma}
                      \bigl(Y_1 - \act{\gamma_1}{j_{1,\delta}}\bigr) \Biggr)
                      \bigl(Y_m - \act{\gamma}{j_{m,\delta}}\bigr) \right).
  \end{equation}
  In this case, for each~$2\leq m\leq n+1$, we
  have~$\Psi_{\delta,m}(j_{1,\delta}) =
  \partial_{Y_1}\Psi_{\delta,1}(j_{1,\delta}) (Y_m - j_{m,\delta})$,
  where~$\partial_{Y_1}$ denotes derivative with respect
  to~$Y_1$. Therefore~$\Psi_{\delta,m}$ is just the expression
  of~$j_{m,\delta}$ as an element of~$L(\Scal)^\Sigma[\,j_{1,\delta}]$
  in a compact representation inspired from
  \cite{gaudry_2adicCMMethod2006}.

  In this case, we often keep only the constant term in
  equation~\eqref{eq:simpler-Psidelta}, and consider the modular
  equations~$\Psi_{\delta,m}$ for~$2\leq m\leq n+1$ as elements of the
  ring~$L(J_1,\ldots,J_n)[J_{n+1},Y]$ with degree at most~$e$
  in~$J_{n+1}$, defined by
  \begin{displaymath}
    \Psi_{\delta,m}(j_1,\ldots,j_{n+1}) = \sum_{\gamma\in K_0/K'} \bigl(\act{\gamma}{j_{m,\delta}}\bigr)
    \prod_{\gamma_1\neq \gamma}\bigl(Y - \act{\gamma_1}{j_{1,\delta}} \bigr).
  \end{displaymath}
  Then, we simply
  have~$j_{m,\delta} =
  \Psi_{\delta,m}(j_{1,\delta})/\partial_{Y_1}\Psi_{\delta,1}(j_{1,\delta})$.
\end{rem}

\subsection{Modular equations of Siegel type for abelian surfaces}
\label{sub:siegel}

The Siegel modular varieties are prominent examples of PEL Shimura
varieties. They are moduli spaces for complex abelian varieties of
dimension~$g$ with a certain polarization and level structure.
Another example is given by the Hilbert modular varieties, for which
the PEL structure contains an additional real multiplication
embedding. In this subsection and the next, we explain how these
examples fit in the general setting of PEL Shimura varieties, and we
show that modular equations of Siegel and Hilbert type in dimension 2
\cite{milio_QuasilinearTimeAlgorithm2015,milio_ModularPolynomialsHilbert2020}
are special cases of modular equations as defined above.


\paragraph{Siegel moduli spaces.}
Let~$g\geq 1$. The \emph{Siegel modular variety} of dimension~$g$
\cite[§6]{milne_IntroductionShimuraVarieties2005} is obtained by
taking $B = \Q$, with trivial involution~$*$, and taking the
symplectic module $(V,\psi)$ to be $V = \Q^{2g}$ with
\begin{displaymath}
  \forall u, v\in V,\ \psi(u,v) = u^t \mat{0}{I_g}{-I_g}{0} v.
\end{displaymath}
Then $G = \GSp_{2g}$. The $\Q$-algebra~$B$ is simple of
type~(C). We can choose~$X_+$ to be the set of all complex structures
on~$V(\R)$ that are positive for~$\psi$
\cite[§6]{milne_IntroductionShimuraVarieties2005}, and we have
\begin{displaymath}
  G(\R)_+ = \{g\in G(\R) \st \mu(g)>0 \}.
\end{displaymath}
The reflex field is~$\Q$
\cite[§14]{milne_IntroductionShimuraVarieties2005}.  Generalizing the
example of modular curves, we can identify~$X_+$ with the Siegel upper
half-space~$\Half_g$ endowed with the usual action
of~$\GSp_{2g}(\R)_+$:
\begin{displaymath}
  \tmat{a}{b}{c}{d}\cdot \tau= (a\tau+b)(c\tau+d)^{-1}
\end{displaymath}
for every~$\tau\in \Half_g$ and $\tmat{a}{b}{c}{d}\in G(\R)_+$,
where~$a,b,c$, and~$d$ are~$g\times g$ blocks.

Let~$(e_1,\ldots,e_{2g})$ be the canonical basis of~$V(\Q)$. Choose
positive integers $D_1|\cdots|D_g$ such that~$D_1=1$, and
let~$\Lambda_0\subset V(\Q)$ be the lattice generated
by~$e_1,\ldots,e_g, D_1 e_{g+1},\ldots, D_g e_{2g}$. Then~the type of
the polarization~$\psi$ on~$\Lambda_0$ is a product of cyclic groups
of order~$D_1,\ldots,D_g$; we also say that~$\psi$ is of
type~$(D_1,\ldots,D_g)$. Let $K$ be a compact open subgroup
of~$G(\A_f)$ that stabilizes $\Lambda_0\otimes\Zhat$, and let~$\Scal$
be the connected component of~$\Sh_K(\C)$ defined by the identity
matrix in~$G(\A_f)$. Then~$\Scal$ is identified with the
quotient~$\Gamma\backslash \Half_g$, where
\begin{displaymath}
  \Gamma = \GSp_{2g}(\Q)_+ \cap K = \Sp_{2g}(\Q) \cap K.
\end{displaymath}

By Proposition~\ref{prop:modular-int-AV}, $\Scal$ is a moduli space
for polarized abelian varieties with polarization
type~$(D_1,\ldots,D_g)$ and level~$K$ structure such that $H_1(A,\Z)$
is isomorphic to the standard polarized lattice
to~$(\Lambda_0,\psi)$. This modular interpretation coincides with the
classical one
\cite[§8.1]{birkenhake_ComplexAbelianVarieties2004}. Also, modular
forms on~$\Scal$ can be identified with Siegel modular forms in
the classical sense, as we mentioned in~§\ref{sub:Phiell} in the
case~$g=1$.


\paragraph{Siegel modular equations.}
We now focus on the special case given by
\begin{displaymath}
  g = 2,\ D_1 = D_2 = 1,\ \Lambda_0 = \Z^{2g},\ K =
  \GSp_{2g}(\Zhat).
\end{displaymath}
Then~$\Sh_K(\C)$ has only one connected component defined over~$\Q$,
and classifies principally polarized abelian surfaces
over~$\C$. Modular forms on~$\Sh_K$ are identified with classical
Siegel modular forms of level~$\Sp_4(\Z)$. As shown by Igusa
\cite{igusa_SiegelModularForms1962}, the graded $\Q$-algebra of these
modular forms is generated by four elements of respective weights 4,
6, 10, and~12. These generators can be taken to be
$I_4, I_6', I_{10}$, and~$I_{12}$ in Streng's notation
\cite[p.\,42]{streng_ComplexMultiplicationAbelian2010}. The function
field of $\Sh_K$ over~$\Q$ is therefore generated by the three
algebraically independent \emph{Igusa invariants}:
\begin{displaymath}
  j_1 = \dfrac{I_4 I_6'}{I_{10}},\quad j_2 =  
  \dfrac{I_4^2 I_{12}}{I_{10}^2}, \quad j_3 = \dfrac{I_4^5}{I_{10}^2}.
\end{displaymath}

Let~$\ell$ be a prime, and consider the Hecke correspondence of level
\begin{displaymath}
  \delta = \mat{\ell I_2}{0}{0}{\,I_2} \qquad\text{as a 4} \times \text{4 matrix in
    2}\times \text{2 blocks}.
\end{displaymath}
The group $K\cap\, \delta K \delta^{-1} \cap G(\Q)_+$ is usually
denoted by~$\Gamma^0(\ell)$, and the degree of~$H_\delta$ is
\begin{displaymath}
  \Heckedeg(\delta) = \ell^3 + \ell^2 + \ell + 1.
\end{displaymath}
The Hecke correspondence~$H_\delta$ is absolutely irreducible, and
describes all principally polarized abelian surfaces $\ell$-isogeous
to a given one; the degree of these isogenies is
$\isogdeg(\delta) = \ell^2$.  In this case, the
function~$j_{1,\delta}$ generates the function field on the Hecke
correspondence \cite[Lem.~4.2]{broker_ModularPolynomialsGenus2009}, so
that $d_1 = \Heckedeg(\delta)$ and $d_2 = d_3 = 1$, in the notation
of~§\ref{sub:modeq}. The modular equations from
Definition~\ref{def:modeq} are called the Siegel modular equations of
level~$\ell$ in Igusa invariants. They have been computed for $\ell = 2$ and $\ell=3$
\cite{milio_QuasilinearTimeAlgorithm2015}.

\subsection{Modular equations of Hilbert type for abelian surfaces}
\label{sub:hilbert}


\paragraph{Hilbert moduli spaces.}
Let~$F$ be a totally real number field of degree~$g$ over~$\Q$, and
let $B = F$ with trivial involution~$*$. The $\Q$-algebra~$B$ is
simple of type~(C). Let $V = F^2$, which is a $\Q$-vector space of
dimension~$2g$, and define the symplectic form~$\psi$ on~$V$ as follows:
\begin{displaymath}
  \forall a,b,c,d\in F,\ \psi\bigl((a,b),(c,d)\bigr) = \Tr_{F/\Q}(ad-bc).
\end{displaymath}
Then~$(V,\psi)$ is a faithful symplectic $(B,*)$-module, where~$B$
acts on~$V$ by multiplication. The associated algebraic group is
$G = \GL_2(F)$. The~$g$ real embeddings of~$F$ induce identifications
\begin{displaymath}
  V(\R) = (\R^2)^g \quad \text{and}\quad G(\R) = \prod_{i=1}^g \GL_2(\R).
\end{displaymath}
The subgroup~$G(\R)_+$ consists of matrices with totally positive
determinant.

There is a particular complex structure~$x_0\in G(\R)$ on
$V(\R)$ given by
\begin{displaymath}
  x_0 = \Bigl(\tmat{0}{1}{-1}{0}\Bigr)_{1\leq i\leq g}.
\end{displaymath}
Let~$X_+$ be the~$G(\R)_+$-conjugacy class of~$x_0$.  Then
$(G,X_+)$ is called a \emph{Hilbert Shimura datum}. Its
reflex field is~$\Q$: see
\cite[§X.4]{vandergeer_HilbertModularSurfaces1988} when
$g=2$, and
\cite[Ex.~12.4]{milne_IntroductionShimuraVarieties2005} in
general. The domain~$X_+$ can be identified with
$\Half_1^g$, where~$\Half_1$ is the complex upper
half-plane, endowed with the action of~$\GL_2(\R)_+$ on each
coordinate.


Let~$\Z_F$ be the integer ring of~$F$, and take
$\Lambda_0 = \Z_F\oplus\Z_F^\dual$, where~$\Z_F^\dual$ is the dual
of~$\Z_F$ with respect to the trace form. Then the stabilizer
of~$\Lambda_0$ in~$B$ is~$\Z_F$, and~$\psi$ is principal
on~$\Lambda_0$.  Let~$K$ be a compact open subgroup
of~$\GL(\Lambda_0\otimes\Zhat)$.

\begin{rem}
  In the Hilbert setting, the group~$\mu(\Gamma_c)$ is not equal
  to~$\mu(\mathcal{E})$ in general. For instance, if
  $K=\GL(\Lambda_0\otimes\Zhat)$, and $c=\tmat{1}{0}{0}{1}$, then
  \begin{displaymath}
    \Gamma_c = G(\R)_+\cap K = \{ g\in \GL(\Lambda_0)\st \det(g) \text{
      is totally positive}\},
  \end{displaymath}
  so~$\mu(\Gamma_c)$ is the set of totally positive units in~$\Z_F$. On
  the other hand,~$\mu(\mathcal{E})$ is the set of all squares of units. For
  instance, if $g=2$, then $\mu(\mathcal{E}) = \mu(\Gamma_c)$ if and only if the
  fundamental unit in~$\Z_F$ has negative norm.
\end{rem}

We now assume that~$K$ has been chosen in such a way that
\begin{equation}
  \label{eq:valid-K}
  G(\Q)_+\cap K = \bigl\{g\in \GL(\Lambda_0)\st \mu(g)\in \Z_F^{\times 2}\bigr\}.
\end{equation}
The Shimura variety~$\Sh_K(G,X_+)(\C)$ has several connected
components: the narrow class group of~$F$ is a quotient of
$\pi_0(\Sh_K(\C))$
\cite[Cor.~I.7.3]{vandergeer_HilbertModularSurfaces1988}.  Let $\Scal$
be the connected component defined by the identity matrix
in~$G(\A_f)$. Then there is a natural isomorphism
\begin{displaymath}
  \Scal = (G(\Q)_+\cap K) \backslash \Half_1^g \simeq
  \SL(\Z_F\oplus\Z_F^\dual)\backslash \Half_1^g.
\end{displaymath}
By Proposition~\ref{prop:modular-int-AV}, the component~$\Scal$
parametrizes principally polarized abelian varieties with real
multiplication by~$\Z_F$ and level~$K$ structure such that~$H_1(A,\Z)$
is isomorphic to the polarized lattice~$(\Lambda_0,\psi)$ with its
additional data. The modular forms of weight~$w$ on~$\Scal$ are
identified with the classical Hilbert modular forms of
weight~$(w,w,\ldots,w)$ for~$F$ and
level~$\SL(\Z_F\oplus\Z_F^\dual)$
\cite[§4]{freitag_HilbertModularForms1990}.


In the special case~$g = 2$, let $\Sigma = \{1,\sigma\}$,
where~$\sigma$ is the involution of~$V$ coming from real conjugation
in~$F$. On~$G(\R)_+$, the involution~$\sigma$ acts as permutation of
the two factors. Modular forms that are symmetric under~$\Sigma$ are
symmetric Hilbert modular forms in dimension 2 in the usual sense
\cite[§1.3]{bruinier_HilbertModularForms2008}.

\paragraph{Hilbert modular equations.} Let~$F$ be a real quadratic
field, and assume moreover that the fundamental unit of~$F$ has
negative norm; then~$K=\GL(\Lambda_0\otimes\Zhat)$
satisfies~\eqref{eq:valid-K}.
Let~$\beta\in \Z_F$ be totally positive and prime, and let
\begin{displaymath}
  \delta = \mat{\beta}{0}{0}{1} \in G(\A_f).
\end{displaymath}
The Hecke correspondence~$H_\delta$ is absolutely irreducible, has
degree $\Heckedeg(\delta) = N_{F/\Q}(\beta) + 1$, and parametrizes
isogenies of degree~$\isogdeg(\delta) = N_{F/\Q}(\beta)$. One can
check that~$H_\delta$ intersects~$\Scal\times \Scal$
nontrivially. Being able to consider this Hecke correspondence is the
reason for our different choice of~$G$ in~§\ref{sec:pel} compared
to~\cite[§8]{milne_IntroductionShimuraVarieties2005}.


As invariants on~$\Scal$, one possibility is to use the pullback of
Igusa invariants by the forgetful map to the Siegel threefold,
i.e.~the Siegel moduli space for~$g=2$
\cite{lauter_ComputingGenusCurves2011}. They are symmetric with
respect to $\Sigma$, and the equation relating these three invariants
is the equation of the associated \emph{Humbert surface}, the image of
the Hilbert surface~$\Scal$ inside the Siegel threefold. In this case,
the modular equations describe simultaneously $\beta$- and
$\sigma(\beta)$-isogenies \cite{milio_ModularPolynomialsHilbert2020}.

In special cases, the field of $\Sigma$-invariant modular functions
can be generated by two elements called \emph{Gundlach
  invariants}. This reduction of the number of variables is
interesting in practice. For instance, if $F = \Q(\sqrt{5})$, then the
graded $\Q$-algebra of symmetric Hilbert modular forms is free over
three generators~$F_2,F_6$, and~$F_{10}$ of respective weights~$2$,
$6$, and $10$ \cite{gundlach_BestimmungFunktionenZur1963}; therefore,
$L(\Scal)^\Sigma = \Q(g_1,g_2)$ where the Gundlach invariants~$g_1$
and~$g_2$ are defined by
\begin{displaymath}
  g_1 = \dfrac{F_2^5}{F_{10 }},\quad g_2 = \dfrac{F_2^2 F_6}{F_{10}}.
\end{displaymath}
Moreover, $g_1$ and~$g_2$ are algebraically independent.  The
associated modular equations are called the Hilbert modular equations
of level~$\beta$ in Gundlach invariants for~$F=\Q(\sqrt{5})$, and
have been computed up to $N_{F/\Q}(\beta) = 59$
\cite{milio_DatabaseModularPolynomials}. They also describe both
$\beta$- and $\sigma(\beta)$-isogenies.

\section{Degree estimates for modular equations}
\label{sec:degree-bound}

We fix a PEL setting as in~§\ref{sub:modeq}; in particular we make a
choice of invariants~$j_1,\ldots,j_{n+1}$ on the Shimura
components~$\Scal$ and~$\Tcal$. Let $\delta\in G(\A_f)$, and assume
that the Hecke correspondence~$H_\delta$
intersects~$\Scal\times \Tcal$ nontrivially.  In
Definition~\ref{def:modeq}, we defined the modular equations
$\Psi_{\delta,1},\ldots,\Psi_{\delta,n+1}$; they are multivariate
polynomials in the variables $Y_1,\ldots, Y_{n+1}$
describing~$H_\delta$ and its conjugates under~$\Sigma$. Their
coefficients are uniquely determined rational fractions in
$L(J_1,\ldots,J_n)[J_{n+1}]$ of degree at most~$e$ in~$J_{n+1}$, where
the integer~$e$ is defined as in equation~\eqref{eq:jn+1}. The goal of
this section is to prove the upper bounds on the degree of the
coefficients of the modular equations~$\Psi_{\delta,m}$ given in the
first part of Theorem~\ref{thm:main}. We also give explicit variants
in the case of modular equations for abelian surfaces. As indicated in
the introduction, the proof works by identifying a denominator of the
modular equations, then by analyzing the degree of the rational
fractions we obtain when rewriting a quotient of modular forms of
bounded weights in terms of the invariants~$j_1,\ldots,j_{n+1}$.

\subsection{The common denominator of \texorpdfstring{$\Psi_{\delta,m}$}{}}
\label{sub:denom}


We keep the notation used in~§\ref{sub:modeq}: in particular
\begin{displaymath}
  K' = K\cap \,\delta K \delta^{-1},\quad K_0 =  K \rtimes\Sigma,
\end{displaymath}
and~$K''$ is a normal subgroup of finite index in~$K$, contained
in~$K'$ and stabilized by~$\Sigma$. The natural action of~$K_0$ on
modular functions of level~$K''$ extends to an action on modular forms.

For each $1\leq i\leq {n+1}$, fix a nonzero modular form~$\chi_i$
invariant under~$\Sigma$ and defined over~$L$ such that $\chi_i j_i$
is again a modular form (i.e.~has no poles); we say that~$\chi_i$ is a
\emph{denominator} of~$j_i$. This is possible by
Proposition~\ref{prop:mf-symmetric}. For each~$i$, the function
\begin{displaymath}
  \chi_{i,\delta}\defby [x,g] \mapsto \chi_i([x,g \delta])
\end{displaymath}
is a modular form of weight~$\wt(\chi_i)$ on the preimage of~$\Scal$
in~$\Sh_{K'}(\C)$. We define the functions~$g_{\delta,m}$ on~$\Scal$ for
$1\leq m\leq n+1$ as follows:
\begin{displaymath}
  \begin{aligned}
    g_{\delta,m} &= \prod_{i=1}^{m} \prod_{\gamma\in K_0/K'}
      \act{\gamma}{\chi_{i,\delta}}.
  \end{aligned}
\end{displaymath}


\begin{lem}
  \label{lem:denom-weight} For every $1\leq m\leq n+1$,
  the function~$g_{\delta,m}$ is a nonzero symmetric
  modular form on~$\Scal$, and
  \begin{displaymath}
    \wt(g_{\delta,m}) = (\#\Sigma) \Heckedeg(\delta) \sum_{i=1}^{m} \wt(\chi_i).
  \end{displaymath}
\end{lem}

\begin{proof}
  By construction, the function~$g_{\delta,m}$ is a modular form of
  level~$K''$ and weight $\sum_{i=1}^m \# (K_0/K')\, \wt(\chi_i)$. We
  have $\# (K_0/K') = (\#\Sigma) \Heckedeg(\delta)$. Each modular
  form~$\act{\gamma}{\chi_{i,\delta}}$ is nonzero on every connected
  component of~$\Sh_{K''}(\C)$ above~$\Scal$, hence~$g_{\delta,m}$ is
  nonzero as well.
  
  Acting by an element of~$K_0$ permutes the factors in the product
  defining~$g_{\delta,m}$, so~$g_{\delta,m}$ is in fact a symmetric
  modular form on~$\Scal$.
\end{proof}


\begin{prop}
  \label{prop:denom-valid} For every $1\leq m\leq n+1$, the
  coefficients of the multivariate
  polynomial~$g_{\delta,m} \Psi_{\delta,m}$ are symmetric modular
  forms on~$\Scal$.
\end{prop}

\begin{proof}
  By Definition~\ref{def:modeq}, the polynomial~$\Psi_{\delta,m}$ is a sum of terms of the form
  \begin{align*}
      \left(
        \prod_{i=1}^{m-1} \prod_{
        \gamma_i}
        \Bigl(Y_i - \act{\gamma_i}{j_{i,\delta}} \Bigr)
      \right)
      \prod_{\gamma_m \in K_{m-1}/K_m} \Bigl(Y_m -
        \act{\gamma\gamma_m}{j_{m,\delta}} \Bigr)
  \end{align*}
  where~$\gamma\in K_0$ is fixed, and the middle product is over all
  $\gamma_i\in K_0/K_i$ such that $\gamma_i = \gamma$
  modulo~$K_{i-1}$, but $\gamma_i\neq \gamma$ modulo~$K_i$. In this
  expression, all the cosets~$\gamma_i$ and~$\gamma\gamma_m$ are
  simultaneously disjoint as subsets of~$K_0/K'$. Each denominator is
  accounted for by some factor in the product defining~$g_{\delta,m}$,
  so the coefficients of~$g_{\delta,m}\Psi_{\delta,m}$ are modular
  forms.
\end{proof}


When the modular functions~$j_1,\ldots,j_{n+1}$ have similar
denominators, it is possible to make a better choice
for~$g_{\delta,m}$.

\begin{prop}
  \label{prop:better-g}
  Assume that there exists a modular form~$\chi$ on~$\Scal$ such that
  for every~$i$, we have $\chi_i = \chi^{\alpha_i}$ for some
  integer~$\alpha_i\geq 0$. Let $1\leq m\leq n+1$, and define
  \begin{displaymath}
    g_{\delta,m} = \Bigl(\prod_{\gamma\in K_0} \act{\gamma}{\chi_\delta}\Bigr)^\alpha,
    \qquad\text{where } \alpha = \max_{1\leq i\leq m} \alpha_i.
  \end{displaymath}
  Then~$g_{\delta,m}$ is a nonzero symmetric modular form on~$\Scal$,
  and
  \begin{displaymath}
    \wt(g_{\delta,m}) = (\#\Sigma) \, \Heckedeg(\delta)\, \alpha \wt(\chi).
  \end{displaymath}
  Moreover, the coefficients of~$g_{\delta,m}\Psi_{\delta,m}$ are
  symmetric modular forms on~$\Scal$.
\end{prop}

The proof is similar to that of Proposition~\ref{prop:denom-valid}, and
omitted.

\subsection{Writing quotients of modular forms in terms of invariants}
\label{sub:rewrite}

Let~$f/g$ be a quotient of symmetric modular forms of weight~$w$
on~$\Scal$. We show that when we rewrite such a quotient in terms of
the invariants~$j_1,\ldots,j_{n+1}$, the degree of the rational
fractions we obtain is bounded linearly in~$w$. To make the
proportionality constant explicit, we define the \emph{symmetric
  geometric complexity} of our invariants as follows.

\begin{defn}
  \label{def:gc}
  Let~$f_k$ for $1\leq k\leq r$ be nonzero generators over~$L$ for the
  graded ring of symmetric modular forms on~$\Scal$, with respective
  weights~$w_k$. For each $1\leq k\leq r-1$, let~$\beta_k\geq 1$ be
  the minimal integer such that
  \begin{displaymath}
    \beta_k w_k \in \Z w_{k+1} + \cdots + \Z w_r. 
  \end{displaymath}
  We can find nonzero modular forms
  $\lambda_k,\, \xi_k\in L[f_{k+1},\ldots,f_r]$ such that
  $\wt(\lambda_k) - \wt(\xi_k) = \beta_k w_k$. For
  every~$1\leq k\leq r-1$, the
  function~$\xi_k f_k^{\beta_k}/\lambda_k$ is a quotient of two
  symmetric modular forms of the same weight on~$\Scal$; hence there
  exist polynomials~$P_k,Q_k\in L[J_1,\ldots,J_{n+1}]$ such that
  \begin{displaymath}
    \dfrac{\xi_k f_k^{\beta_k}}{\lambda_k}
    = \dfrac{P_k(j_1,\ldots, j_{n+1})}{Q_k(j_1,\ldots,j_{n+1})}.
  \end{displaymath}
  Denote the total degrees of~$P_k$ and~$Q_k$ by~$\deg(P_k)$
  and~$\deg(Q_k)$ respectively. We define the \emph{symmetric
    geometric complexity} of~$j_1,\ldots,j_{n+1}$ relative to the
  choice of~$f_k,\lambda_k,\psi_k,P_k,Q_k$ to be the positive rational
  number given by, either
  \begin{enumerate}
  \item \label{case:gc-simple}
    \begin{displaymath}
      \left(1+\max_{1\leq k\leq r-1}
        \frac{\wt(\xi_k)}{\beta_k w_k}\right) \max_{1\leq k\leq r-1}
      \frac{\deg(P_k)}{\beta_k w_k + \wt(\xi_k)}\,,
    \end{displaymath}
    
    if the following conditions are satisfied: for
    every~$1\leq k\leq r-1$, the modular forms~$\lambda_k$ and~$\xi_k$
    are powers of~$f_r$ and~$f_{r-1}$ respectively (in
    particular~$\xi_{r-1}=1)$, and~$Q_{k}=1$; or
  \item \label{case:gc-gen}
    \begin{displaymath}
      \sum_{k= 1}^{r-1} \left( \dfrac{1}{\beta_k w_k} \max
      \bigl\{\deg(P_k),\deg(Q_k) \bigr\} \prod_{l=1}^{k-1}\left(1 +
        \dfrac{\wt(\xi_l)}{\beta_l w_l}\right) \right),
    \end{displaymath}
    otherwise.
  \end{enumerate}
  Note that formula~\ref{case:gc-simple}, when it applies, yields a
  smaller result than formula~\ref{case:gc-gen}.
   
  The \emph{symmetric geometric complexity} of~$j_1,\ldots,j_{n+1}$,
  denoted by~$\SGC(j_1,\ldots,j_{n+1})$, is the infimum of this
  quantity over all possible choices of modular
  forms~$f_k,\lambda_k,\xi_k$ and polynomials~$P_k,Q_k$.
\end{defn}

Given Definition~\ref{def:gc}, explicit upper bounds on the geometric
complexity are easy to obtain if a generating set of modular forms is
known. Note that the symmetric geometric complexity is invariant under
permutations of the invariants~$j_1,\ldots,j_{n+1}$, in contrast with
their \emph{geometric complexity} to be defined later, which takes
into account the fact that~$j_{n+1}$ is considered differently in
equation~\eqref{eq:jn+1}.

\begin{prop}
  \label{prop:rewrite}
  Let~$w\geq 0$, let $f,\, g$ be symmetric modular forms on~$\Scal$ of
  weight~$w$, and assume that~$g$ is nonzero. Then there exist
  polynomials~$P,\, Q\in L[J_1,\ldots, J_{n+1}]$ of total degree at
  most~$\SGC(j_1,\ldots,j_{n+1})w$ such that
  \begin{displaymath}
    \dfrac{f}{g} = \dfrac{P(j_1,\ldots,j_{n+1})}{Q(j_1,\ldots,j_{n+1})}.
  \end{displaymath}
  Moreover,~$Q$ can be chosen independently of~$f$.
\end{prop}

\begin{proof}
  We keep the notation used in~Definition~\ref{def:gc}, and make a
  choice of generators~$f_k$ for~$1\leq k\leq r$, modular forms
  $\lambda_k,\xi_k$ for~$1\leq k\leq r-1$, and
  polynomials~$P_k,Q_k\in L[J_1,\ldots,J_{n+1}]$
  for~$1\leq k\leq r-1$. Let~$C$ be symmetric geometric complexity
  of~$j_1,\ldots,j_{n+1}$ relative to this choice.

  Let~$f$,~$g$ be as in the proposition.  Then~$f$ and~$g$ can be
  expressed as a sum of monomial terms of the form
  \begin{displaymath}
    c f_1^{\alpha_1}\cdots f_r^{\alpha_r} \qquad\text{with } c\in L \text{ and } \sum_{k=1}^r
    \alpha_k w_k = w.
  \end{displaymath}
  We give algorithms to rewrite the fraction~$P/Q = f/g$ (currently a
  rational fraction in terms of the modular forms~$f_k$) as a fraction
  of invariants, and bound the total degree of the output.

  \paragraph{Case~\ref{case:gc-simple} of~Definition~\ref{def:gc}.} We
  assume that~$\lambda_k$ and~$\xi_k$ are powers of~$f_r$
  and~$f_{r-1}$ respectively for every $1\leq k\leq r-1$. In this
  case, for each~$1\leq k\leq r-2$, the integer~$\beta_k$ can be seen
  as the order of~$w_k$ in the group~$\Z/(\Z w_{r-1}+\Z w_r)$. We can
  write 
  \begin{displaymath}
    w = \sum_{k=1}^{r-2} s_k w_k \quad(\text{mod } \Z w_{r-1}+ \Z w_r)
  \end{displaymath}
  for some integers $0\leq s_k < \beta_k$, and this determines the
  integers~$s_k$ uniquely (if such a linear combination vanishes,
  considering the smallest nonzero~$s_k$ yields a contradiction). Then
  each monomial appearing in~$P$ and~$Q$ is divisible by
  $f_1^{s_1}\cdots f_{r-2}^{s_{r-2}}$. After simplifying by this
  common factor, we can assume that the common weight~$w$ of~$P$
  and~$Q$ satisfies~$w\in \Z w_{r-1}+\Z w_r$. Then, for each
  $1\leq k\leq r-2$, the exponent of~$f_k$ in each monomial of~$P$
  and~$Q$ is divisible by~$\beta_k$.
  For convenience, write
  \begin{displaymath}
    a = \max_{1\leq k\leq r-1} \frac{\wt(\xi_k)}{\beta_kw_k}.
  \end{displaymath}
  In order to rewrite~$P/Q$ in terms of invariants, we proceed as
  follows.
  \begin{enumerate}
  \item \label{step:case1-M} Multiply~$P$ and~$Q$ by~$f_{r-1}^{\floor{aw/\wt(f_{r-1})}}$.
  \item \label{step:case1-R}
    For each $1\leq k\leq r-2$, replace each occurence
    of~$f_k^{\beta_k}$ by~$\lambda_k P_k/\xi_k$ in~$P$ and~$Q$.
  \item \label{step:case1-D} Let~$0\leq s_{r-1}< \beta_{r-1}$ be such
    that~$w = s_{r-1}w_{r-1}\mod w_r$, and divide~$P$ and~$Q$
    by~$f_{r-1}^{s_{r-1}}$.
  \item \label{step:case1-R2} Replace each occurence of~$f_{r-1}^{\beta_{r-1}}$
    by~$\lambda_{r-1} P_{r-1}$ in~$P$ and~$Q$.
  \item \label{step:case1-D2} Finally, divide~$P$ and~$Q$
    by~$f_r^{(w-s_{r-1}w_{r-1})/w_r}$.
  \end{enumerate}

  This algorithm runs independently on each monomial of~$P$
  and~$Q$. Let~$M = c\prod_{k=1}^r f_k^{\alpha_k}$, with~$c\in L$, be
  such a monomial after step~\ref{step:case1-M}. Let us show that the
  exponent of~$f_{r-1}$ in~$M$ remains nonnegative after
  step~\ref{step:case1-R}. In this step, we introduce a denominator
  given by
  \begin{displaymath}
    \prod_{k=1}^{r-2} \xi_k^{\alpha_k/\beta_k} = \prod_{k=1}^{r-2}
    f_{r-1}^{\frac{\wt(\xi_k)\alpha_k}{\wt(f_{r-1})\beta_k}}.
  \end{displaymath}
  We have
  \begin{displaymath}
    \sum_{k=1}^{r-2} \frac{\wt(\xi_k)\alpha_k}{\wt(f_{r-1})\beta_k}
    \leq a \sum_{k=1}^{r-2} \frac{\alpha_k w_k}{\wt(f_{r-1})}
    \leq \frac{aw}{\wt(f_{r-1})},
  \end{displaymath}
  hence
  \begin{displaymath}
    \sum_{k=1}^{r-2} \frac{\wt(\xi_k)\alpha_k}{\wt(f_{r-1})\beta_k}
    \leq \floor{\frac{aw}{\wt(f_{r-1})}} \leq \alpha_{r-1} \qquad \text{by step~\ref{step:case1-M}}
  \end{displaymath}
  because the left hand side is an integer. Therefore, at the end of
  step~\ref{step:case1-R},~$M$ belongs to the polynomial
  ring~$L[J_1,\ldots,J_{n+1}][f_{r-1},f_r]$. Hence, we
  have~$M\in L[J_1,\ldots,J_{n+1}][f_{r-1}^{\beta_{r-1}},f_r]$ after
  step~\ref{step:case1-D}, and finally~$M\in L[J_1,\ldots,J_{n+1}]$
  after step~\ref{step:case1-D2}. 
  
  It remains to bound the total degree of~$M$ after
  step~\ref{step:case1-D2}. To do this, we consider the total weight
  of~$M$ in~$f_1,\ldots,f_{r-1}$. For each~$1\leq k\leq r-1$, the
  modular form~$\lambda_k$ is a power of~$f_r$; hence
  replacing~$f_k^{\beta_k}$ by~$\lambda_k P_k/\xi_k$ in
  steps~\ref{step:case1-R} or~\ref{step:case1-R2} reduces this weight
  by~$\beta_k w_k + \wt(\xi_k)$, and increases the total degree of~$M$
  in~$J_1,\ldots,J_{n+1}$ by at most~$\deg(P_k)$. At the beginning of
  step~\ref{step:case1-R}, the total weight of~$M$
  in~$f_1,\ldots,f_{r-1}$ is at most $(1+a)w$. Therefore the total
  degree of~$M$ in~$J_1,\ldots,J_{n+1}$ at the end of the algorithm is
  bounded above by
  \begin{displaymath}
    (1+a)w \max_{1\leq k\leq r-1} \frac{\deg(P_k)}{\beta_k w_k+\deg(\xi_k)} = Cw.
  \end{displaymath}

  \paragraph{Case~\ref{case:gc-gen} of~Definition~\ref{def:gc}.} In
  the general case, we perform replacements and simplifications in a
  sequential way.

  We start by defining integers~$z_k, d_k$ for~$0\leq k\leq r-1$
  and~$s_k, a_k$ for~$1\leq k\leq r-1$ by induction as follows:
  \begin{itemize}
  \item $z_0 = w$ and $d_0 = 0$;
  \item For each~$1\leq k\leq r$, the integer~$0\leq s_k< \beta_k$ is
    defined by the relation
    \begin{displaymath}
      z_{k-1} = s_k w_k \quad(\text{mod } \Z w_{k+1} +\cdots + \Z w_r);
    \end{displaymath}
  \item $a_k =\displaystyle \floor{\frac{z_{k-1}}{\beta_k w_k}}$ for
    each $1\leq k\leq r-1$;
  \item $z_k = z_{k-1} - s_k w_k + a_k \wt(\xi_k)$ for each
    $1\leq k\leq r-1$; $\phantom{\displaystyle\frac12}$
  \item $d_k = d_{k-1} + a_k \max\{\deg(P_k),\deg(Q_k)\}$ for each $1\leq k\leq r-1$.
  \end{itemize}

  In order to rewrite~$P/Q$ in terms of invariants, we use the
  following algorithm. For~$k=1$ up to~$k=r-1$, do:
  \begin{enumerate}
  \item \label{step:case2-D} Divide~$P$ and~$Q$ by~$f_k^{s_k}$;
  \item \label{step:case2-R} Replace each occurence of~$f_k^{\beta_k}$
    by $\dfrac{\lambda_k P_k}{\xi_k Q_k}$ in~$P$ and~$Q$;
  \item \label{step:case2-M} Multiply~$P$ and~$Q$ by
    $\xi_k^{a_k} Q_k^{a_k}$.
  \end{enumerate}
  Finally, simplify the remaining occurences of~$f_r$. We prove the
  following statement~$(H_k)$ by induction for every
  $1\leq k\leq r$:

  \claimemph{$(H_k)$ At the beginning of the~$k$-th loop,~$P$ and~$Q$
    are elements of $L[J_1,\ldots,J_{n+1}][f_k,\ldots,f_r]$ of
    weight~$z_{k-1}$, with total degree at most~$d_{k-1}$ in
    $J_1,\ldots,J_{n+1}$, such that
    \begin{displaymath}
      \dfrac{f}{g} = \dfrac{P(j_1,\ldots,j_{n+1})}{Q(j_1,\ldots,j_{n+1})}.
    \end{displaymath}}

  The statement~$(H_1)$ is true by definition of~$z_0$ and~$d_0$;
  assume that~$(H_k)$ is true for some~$k\geq 1$. Then we see,
  in order, that during the $k$-th loop:
  \begin{itemize}
  \item $z_{k-1} \in \sum_{i=k}^r \Z w_i$, so~$s_k$ is well defined.
  \item In each monomial of~$P$ and~$Q$, the exponent of~$f_k$ is of
    the form $a\beta_k + s_k$ for some integer $a\leq a_k$. Therefore
    step~\ref{step:case2-D} is an exact division, and after
    step~\ref{step:case2-R} there are no more occurences of~$f_k$
    in~$P$ or~$Q$.
  \item After step~\ref{step:case2-M},~$P$ and~$Q$ are elements of
    $L[J_1,\ldots,J_{n+1}][f_{k+1},\ldots,f_r]$ of weight
    \begin{displaymath}
      z_{k-1} - s_k w_k + a_k \wt(\xi_k) = z_k.
    \end{displaymath}
  \end{itemize}

  It remains to show that the degree of~$P,\,Q$ in
  $J_1,\ldots,J_{n+1}$ is bounded by~$d_k$ after
  step~\ref{step:case2-M}. This comes from the following observation:
  during the $k$-th loop, we only multiply the polynomials in
  $J_1,\ldots, J_{n+1}$ already present by $P_k^b Q_k^{a_k - b}$ for
  some $0\leq b\leq a_k$. This proves our claim~$(H_k)$ for all
  $1\leq k\leq r$.

  At the end of the algorithm, all the occurences of~$f_r$ cancel
  out. Therefore we obtain polynomials~$P$ and~$Q$ of total degree at
  most~$d_{r-1}$ such that
  \begin{displaymath}
    \dfrac{f}{g} = \dfrac{P(j_1,\ldots, j_{n+1})}{Q(j_1,\ldots, j_{n+1})}.
  \end{displaymath}
  By induction, we obtain
  \begin{displaymath}
    z_k \leq w\prod_{l=1}^k \left(1 + \dfrac{\wt(\xi_l)}{\beta_l w_l} \right)
  \end{displaymath}
  and
  \begin{displaymath}
    d_{r-1} \leq \sum_{k= 1}^{r-1} \left( \dfrac{w}{\beta_k w_k}
      \max\{\deg(P_k),\deg(Q_k)\} \prod_{l=1}^{k-1}\left(1 +
        \dfrac{\wt(\xi_l)}{\beta_l w_l}\right) \right) = C w.
  \end{displaymath}

  In both cases~\ref{case:gc-simple} and~\ref{case:gc-gen}, the
  algorithm runs independently on the numerator and denominator,
  hence~$Q$ can be chosen independently of~$f$.
\end{proof}

\subsection{Degree bounds in canonical form}
\label{sub:canonical}

Recall that the modular function~$j_{n+1}$ satisfies
eq.~\eqref{eq:jn+1}: we have~$E(j_1,\ldots,j_{n+1})=0$ where
\begin{displaymath}
  E = \sum_{k = 0}^{e}
  E_k(J_1,\ldots,J_n)\,J_{n+1}^{\,k} \in L[J_1,\ldots,J_n,J_{n+1}]
\end{displaymath}
has degree~$e$ in~$J_{n+1}$ and is irreducible. Let~$d_E$ denote the
total degree of~$E$ in the variables~$J_1,\ldots,J_n$.  In this
section, we work in the ring~$L(J_1,\ldots,J_n)[J_{n+1}]$
modulo~$E$. We say that a fraction $R\in L(J_1,\ldots,J_{n+1})$ is in
\emph{canonical form} if~$R$ is a polynomial in~$J_{n+1}$ of degree at
most~$e-1$.


\begin{prop}
  \label{prop:canonical}
  Let~$d\geq 0$, let~$P,\,Q\in L[J_1,\ldots, J_{n+1}]$ be polynomials
  of total degree at most~$d$, and assume that $Q(j_1,\ldots,j_{n+1})$
  is not identically zero. Let $R\in L(J_1,\ldots,J_n)[J_{n+1}]$ be the
  fraction in canonical form such that~$P/Q = R\mod E$.  Then the
  total degree of~$R$ in~$J_1,\ldots,J_n$ is bounded above by
  $(e+2d_E)d$.
\end{prop}

\begin{proof}
  In this proof, degrees and coefficients are taken with respect to
  the variable~$J_{n+1}$ unless otherwise specified.  First, we invert
  the denominator~$Q$. Consider the resultant
  \begin{displaymath}
    Z = \Res_{J_{n+1}}(Q, E) \in L[J_1,\ldots,J_n],
  \end{displaymath}
  which is nonzero by hypothesis. Let
  $U, V\in L[j_1,\ldots, j_{n+1}]$ be the associated Bézout
  coefficients, so that
  \begin{displaymath}
    Z = UQ + VE.
  \end{displaymath}
  The inverse of~$Q$ modulo~$E$ is~$U/Z$, so we have
  $P/Q = UP/Z \mod E$.

  It is well-known that~$Z$ (resp.~$Q$) has a polynomial expression of
  degree~$e$ (resp.~$e-1$) in the coefficients of~$Q$, and
  degree~$\deg(Q)$ in the coefficients of~$E$. Since the total degree
  of~$Q$ is at most~$d$, the total degrees of~$Z$ and~$UP$ in
  $J_1,\ldots,J_{n}$ are bounded above by~$d(e+d_E)$. The degree
  of~$UP$ in~$J_{n+1}$ is at most $d+e-1$.
  
  Now, we reduce~$UP/Z$ modulo~$E$ to obtain a numerator of degree at
  most~$e-1$ in~$J_{n+1}$.  We can decrease this degree by 1 by
  multiplying above and below by~$E_e(J_1,\ldots,J_n)$ and using
  the relation
  \begin{displaymath}
    E_e J_{n+1}^{e} = - \sum_{k=0}^{e-1} E_k J_{n+1}^{k}\qquad \mod E.
  \end{displaymath}
  When doing so, the total degree in~$J_1,\ldots,J_n$ increases by at
  most~$d_E$. This operation is done at most~$d$ times; therefore the
  result has total degree at most~$(e+2d_E)d$ in~$J_1,\ldots,J_n$
  and degree at most~$e-1$ in~$J_{n+1}$.
\end{proof}

\begin{defn}
  We define the \emph{geometric complexity} of the invariants~$j_1,\ldots,j_{n+1}$ to be
  \begin{displaymath}
    \GC(j_1,\ldots,j_{n+1}) = (e+2d_E) \SGC(j_1,\ldots,j_{n+1}) +e-1.
  \end{displaymath}
  This quantity depends on the choice of~$j_{n+1}$ as a distinguished
  invariant.
\end{defn}

\begin{prop}
  \label{prop:gc}
  Let~$w\geq 0$, let $f, g$ be symmetric modular forms on~$\Scal$ of
  weight~$w$, and assume that~$g$ is nonzero.
  Let~$R\in L(J_1,\ldots J_n)[J_{n+1}]$ be the rational fraction in
  canonical form such that
  \begin{displaymath}
    \frac{f}{g} = R(j_1,\ldots,j_{n+1}).
  \end{displaymath}
  Then the total degree of~$R$ in~$J_1,\ldots,J_{n+1}$ is bounded above by
  $\GC(j_1,\ldots,j_{n+1}) w$.
\end{prop}

\begin{proof}
  Combine Propositions~\ref{prop:rewrite} and~\ref{prop:canonical}.
\end{proof}

We are ready to prove the first part of Theorem~\ref{thm:main} on
degree bounds for modular equations, with an explicit expression for
the constant~$C_1$.

\begin{thm}
  \label{thm:degree-bound-proved}
  Let~$H_\delta$ be an absolutely irreducible Hecke correspondence
  on~$\Scal\times \Tcal$ defined by an adelic element~$\delta$ of~$G$,
  and let~$\Heckedeg(\delta)$ be the degree of~$H_\delta$. For
  each~$1\leq i\leq n+1$, let~$\chi_i$ be a denominator of~$j_i$ as
  in~§\ref{sub:denom}. Let~$1\leq m\leq n+1$. Finally, let
  \begin{displaymath}
    C_1 = GC(j_1,\ldots,j_{n+1})\,(\#\Sigma)  \sum_{i=1}^m \wt(\chi_i).
  \end{displaymath}
  Then there exists a polynomial~$D_m\in L[J_1,\ldots,J_n]$ of total
  degree at most~$C_1\Heckedeg(\delta)$ such that~$D_m\Psi_{\delta,m}$
  is a polynomial in~$J_1,\ldots,J_{n+1},Y_1,\ldots,Y_m$ whose total
  degree in~$J_1,\ldots,J_{n+1}$ is also bounded above
  by~$C_1 \Heckedeg(\delta)$. In particular,
  if~$\Frac\in L(J_1,\ldots,J_n)[J_{n+1}]$ is a coefficient
  of~$\Psi_{\delta,m}$, then the total degree of~$\Frac$ is bounded
  above by~$C_1 \Heckedeg(\delta)$.
\end{thm}

\begin{proof}
  Let~$g_{\delta,m}$ be the modular form on~$\Scal$ defined
  in~§\ref{sub:denom}, and let~$\Frac$ be a coefficient
  of~$\Psi_{\delta,m}$. By Proposition~\ref{prop:denom-valid}, the
  modular function~$\Frac(j_1,\ldots,j_{n+1})$ is of the
  form~$f/g_{\delta,m}$, where~$f$ is a modular form on~$\Scal$ of
  weight~$\wt(g_{\delta,m})$. By Lemma~\ref{lem:denom-weight}, we have
  \begin{displaymath}
    \wt(g_{\delta,m}) =  (\#\Sigma) \Heckedeg(\delta) \sum_{i=1}^{m} \wt(\chi_i),
  \end{displaymath}
  so the degree bound on~$\Frac$ follows from
  Proposition~\ref{prop:gc}. By Proposition~\ref{prop:rewrite}, the
  denominator can be chosen independently of the coefficient
  of~$\Psi_{\delta,m}$ we consider, hence the existence of a common
  denominator~$D_m$ of the correct total degree.
\end{proof}

\subsection{Explicit degree bounds in dimension 2}
\label{sub:degree-genus-2}


Our methods provide new results about the degrees of the coefficients
of modular equations of Siegel and Hilbert type for abelian surfaces,
introduced in~§\ref{sub:siegel} and~§\ref{sub:hilbert}
respectively. In the Hilbert case, we restrict to the quadratic
field~$F = \Q(\sqrt{5})$, and consider modular equations in terms of
Gundlach invariants.

In both cases, we can take~$j_{n+1}=1$ and~$E = J_{n+1}-1$ in the
notation of~§\ref{sub:modeq}. Then the notions of geometric complexity
and symmetric geometric complexity coincide.

\begin{lem}
  \label{lem:rewrite-siegel}
  Let~$j_1,j_2$, and~$j_3$ denote the Igusa invariants on the Siegel
  threefold $\Sp_4(\Z)\backslash\Half_2$, as defined
  in~§\ref{sub:siegel}. Then we have
  \begin{displaymath}
    \GC(j_1,j_2,j_3,1)\leq \frac16.
  \end{displaymath}
\end{lem}

\begin{proof}
  Recall that the graded $\Q$-algebra of Siegel modular forms of
  level~$\Sp_4(\Z)$ is generated by
  \begin{displaymath}
    f_1 = I_6',\quad f_2 = I_{12}, \quad f_3 = I_4, \quad\text{and}\quad f_4 = I_{10}.
  \end{displaymath}
  We are in case~\ref{case:gc-simple} of Definition~\ref{def:gc}, since
  \begin{displaymath}
    \frac{I_6'I_4}{I_{10}} = j_1, \quad \frac{I_{12}I_4^2}{I_{10}^2} = j_2,
    \quad\text{and}\quad  \frac{I_4^5}{I_{10}^2} = j_3.
  \end{displaymath}
  The definition gives
  \begin{displaymath}
    \SGC(j_1,j_2,j_3,1) \leq \left(1 + \frac23\right) \cdot \frac{1}{10} = \frac16.
  \end{displaymath}
\end{proof}

\begin{prop}
  \label{prop:degree-siegel}
  Let~$\ell$ be a prime number, and let~$\Psi_{\ell,m}$
  for~$1\leq m\leq 3$ denote the Siegel modular equations of
  level~$\ell$ in Igusa invariants. Let~$\Frac\in \Q(J_1,J_2,J_3)$ be
  a coefficient of~$\Psi_{\ell,1}$ (resp.~$\Psi_{\ell,2}$
  or~$\Psi_{\ell,3}$). Then the total degree of~$\Frac$ is bounded above
  by $5 \Heckedeg(\ell)/3$ (resp.~$10 \Heckedeg(\ell)/3$), where
  $\Heckedeg(\ell) = \ell^3 + \ell^2 + \ell + 1$.
\end{prop}

\begin{proof}
  The integer~$\Heckedeg(\ell)$ is the degree of the Hecke
  correspondence. The denominators of~$j_1,j_2,$ and~$j_3$ can be
  taken to be the modular forms~$I_{10},I_{10}^2$, and
  $I_{10}^2$. Let~$g_{\ell,m}$ for~$1\leq m\leq 3$ be the common
  denominators of the modular equations~$\Psi_{\ell,m}$ defined in
  Proposition~\ref{prop:better-g}, so that
  $g_{\ell,2} = g_{\ell,3} = g_{\ell,1}^2$ and
  $\wt(g_{\ell,1}) = 10 \Heckedeg(\ell)$.

  Then~$\Frac(j_1,j_2,j_3)$ is the quotient of two modular forms of
  degree~$10 \Heckedeg(\ell)$ (resp.~$20 \Heckedeg(\ell)$) on~$\Scal$,
  by Proposition~\ref{prop:better-g}. Therefore the result follows from
  Lemma~\ref{lem:rewrite-siegel} and Proposition~\ref{prop:gc}.
\end{proof}

\begin{lem}
  \label{lem:rewrite-hilbert}
  Let~$F = \Q(\sqrt{5})$, and let~$g_1,g_2$ denote the Gundlach
  invariants on the Hilbert
  surface~$\SL(\Z_F\oplus\Z_F^\dual) \backslash\Half_1^2$, as defined
  in~§\ref{sub:hilbert}. Then we have
  \begin{displaymath}
    \GC(g_1,g_2,1) \leq \frac16.
  \end{displaymath}
\end{lem}

\begin{proof}
  Choose~$F_6,F_2$, and~$F_{10}$ as generators of the graded
  $\Q$-algebra of Hilbert modular forms of level
  $\SL(\Z_F\oplus\Z_F^\dual)$. We have
  \begin{displaymath}
    \frac{F_6F_2^2}{F_{10}} = g_2\quad\text{and}\quad \frac{F_2^5}{F_{10}}=g_1.
  \end{displaymath}
  Therefore we are in case~\ref{case:gc-simple} of
  Definition~\ref{def:gc}, and
  \begin{displaymath}
    \GC(g_1,g_2,1)\leq \left(1+\frac23 \right)\cdot\frac{1}{10} = \frac16.
  \end{displaymath}
\end{proof}

\begin{prop}
  \label{prop:degree-hilbert}
  Let~$F = \Q(\sqrt{5})$, let $\beta\in \Z_F$ be a totally positive
  prime, and let~$\Psi_{\beta,m}$ for~$m\in\{1,2\}$ denote the Hilbert
  modular equations of level~$\beta$ in Gundlach
  invariants. Let~$\Frac\in \Q(J_1,J_2)$ be a coefficient
  of~$\Psi_{\beta,1}$ or~$\Psi_{\beta,2}$. Then the total degree
  of~$\Frac$ is bounded above by~$10 \Heckedeg(\beta)/3$, where
  $\Heckedeg(\beta) = N_{F/\Q}(\beta)+1$.
\end{prop}

\begin{proof}
  The integer~$\Heckedeg(\beta)$ is the degree of the Hecke
  correspondence, and the automorphism group~$\Sigma$ used to define
  the Hilbert modular equations has order 2. We can take the
  modular~$F_{10}$ as denominator of both~$g_1$ and~$g_2$; the common
  denominators $g_{\beta,1} = g_{\beta,2}$ from
  Proposition~\ref{prop:better-g} have weight~$20 \Heckedeg(\beta)$,
  so the result follows from Lemma~\ref{lem:rewrite-siegel} and
  Proposition~\ref{prop:gc}.
\end{proof}

The degree bounds in Propositions~\ref{prop:degree-siegel}
and~\ref{prop:degree-hilbert} are both reached experimentally. In the
Siegel case with $\ell=2$, the maximum degree is~25; in the Hilbert
case with $N_{F/\Q}(\beta) = 41$, the maximum degree is~140
\cite{milio_DatabaseModularPolynomials}.

\section{Height estimates for modular equations}
\label{sec:height-bound}

Another important information when manipulating modular equations,
besides their degrees, is the size of their coefficients. More
precisely, we use the notion of \emph{heights} of elements,
polynomials and rational fractions over a number field. The goal of
this section is to prove part~\ref{item:height-bound} of
Theorem~\ref{thm:main}, giving height bounds on coefficients of
modular equations.

As mentioned in the introduction, the proof is inspired by existing
works on elliptic modular
polynomials~\cite{pazuki_ModularInvariantsIsogenies2019}. First, we
study the heights of modular equations evaluated at well-chosen
points, using the fact that the underlying Hecke correspondence
describes isogenous abelian varieties. Then we apply the main result
of~\cite{kieffer_UpperBoundsHeights2020}, which gives a tight
relation between the height of a rational fraction and the heights of
sufficiently many of its evaluations.

\subsection{Definition of heights}
\label{sub:heights}

Let us recall the well-known definitions. We use the following notation:
\begin{itemize}
\item $L$ is a number field of degree~$d_L$ over~$\Q$;
\item $\V_L^0$ (resp.~$\V_L^\infty$) is the set of all nonarchimedean
  (resp.~archimedean) places of~$L$; and
\item $\V_L = \V_L^0\sqcup \V_L^\infty$ is the set of all places of~$L$.
\end{itemize}
For each place~$v$ of~$L$,
\begin{itemize}
\item $L_v$ (resp.~$\Q_v$) denotes the completion of~$L$ (resp.~$\Q$) at~$v$,
\item $d_v = [L_v:\Q_v]$ denotes the local extension degree of~$L/\Q$
  at~$v$, and
\item $|\cdot|_v$ denotes the normalized absolute value associated
  with~$v$.
\end{itemize}
We normalize the nonarchimedean absolute values of~$L$ in the
following way: for each $v\in \V_L^0$, if~$p\in \mathcal{P}_\Q$ is the
prime below~$v$, then $|p|_v = 1/p$.

The (absolute logarithmic Weil) \emph{height} of projective
tuples, affine tuples, elements, poly\-nomials and rational
fractions over~$L$ is defined as follows.

\begin{defn}~ Let~$n\geq 1$, and let~$y_0,\ldots,y_n\in L$.
  \label{def:heights}
  \begin{enumerate}
  \item The \emph{projective height} of~$(y_0:\cdots:y_n)\in \Proj^n_L$ is
    \begin{displaymath}
      \h(y_0:\cdots:y_n) = \sum_{v\in\V_L} \frac{d_v}{d_L} \log
      \bigl(\max_{0\leq i\leq n} |y_i|_v \bigr).
    \end{displaymath}
  \item The \emph{affine height} of~$(y_1,\ldots,y_n)\in L^n$ is the
    projective height of~$(1:y_1:\cdots:y_n)$:
    \begin{displaymath}
      \h(y_1,\ldots,y_n) = \sum_{v\in\V_L} \frac{d_v}{d_L}
      \log \bigl(\max\{1, \max_{1\leq i\leq n} |y_i|_v\} \bigr).
    \end{displaymath}
    In particular, for every~$y\in L$, we have
    \begin{displaymath}
      \h(y) = \sum_{v\in\V_L}
      \dfrac{d_v}{d_L} \log \bigl( \max\{1, |y|_v\} \bigr).
    \end{displaymath}
  \item Let~$P\in L[Y_1,\ldots,Y_n]$ be a multivariate polynomial
    over~$L$, and write
    \begin{displaymath}
      P = \sum_{k = (k_1,\ldots,k_n)\in \N^n} c_k Y_1^{k_1}\cdots Y_n^{k_n}.
    \end{displaymath}
    Let~$v\in\V_L$. We write 
    \begin{displaymath}
      |P|_v =
      \max_{k\in \N^n} |c_k|_v
    \end{displaymath}
    and
    \begin{displaymath}
      \h(P) = \sum_{v\in\V_L} \dfrac{d_v}{d_L} \log
      \bigl(\max\{1, |P|_v\} \bigr).
    \end{displaymath}
    In other words,~$h(P)$ is the height of the affine tuple
    formed by all the coefficients of~$P$.
  \item Let $\Frac\in L(Y_1,\ldots,Y_n)$ be a multivariate rational
    fraction over~$L$, and choose coprime
    polynomials~$P,Q\in L[Y_1,\ldots,Y_n]$ such that $\Frac =
    P/Q$. Then we define~$\h(\Frac)$ as the height of the projective
    tuple formed by all the coefficients of~$P$ and~$Q$.
  \end{enumerate}
\end{defn}

Here are a few elementary properties of heights.
\begin{enumerate}
\item Projective heights are well defined, by the product formula
  \cite[Lem.~B.2.1(a)]{hindry_DiophantineGeometry2000}. Therefore the
  height of a fraction is also well defined.
\item Heights are independent of the ambient number field
  \cite[Lem.~B.2.1(c)]{hindry_DiophantineGeometry2000}, by another
  application of the product formula. In particular we note that
  \begin{displaymath}
    \sum_{v\in \V_L^\infty} \frac{d_v}{d_L}=1.
  \end{displaymath}
\item If~$L=\Q$, then Definition~\ref{def:heights} coincides with the
  naive one given in the introduction.
\end{enumerate}
  
Informally, the height of an element $y\in L$ measures the amount of
information needed to represent $y$.

\subsection{Heights, evaluations and roots}
\label{sub:height-frac}

In this section, we state relations between
\begin{enumerate}
\item The height of a univariate polynomial over~$L$ and the height
  of its roots;
\item The height of a multivariate polynomial or multivariate rational
  fraction over~$L$ with the heights of its values at special points.
\end{enumerate}
Several of the statements are easy consequences of the formul\ae\ from
Definition~\ref{def:heights}, while others are more intricate and are
proved by the author in a separate
paper~\cite{kieffer_UpperBoundsHeights2020}.

Let us start with the evaluation of polynomials; the following
proposition is a slight generalization of
\cite[Prop.~B.7.1]{hindry_DiophantineGeometry2000}.

\begin{prop}
  \label{prop:multivariate-evaluation}
  Let~$d\geq 0$, let~$P\in L[Y_1,\ldots, Y_n]$ be a polynomial of
  total degree at most~$d$, let $1\leq m\leq n$, and let
  $y_1,\ldots, y_m\in L$. Write
  $Q = P(y_1,\ldots, y_m, Y_{m+1}, \ldots, Y_n)$. Then
  \begin{displaymath}
    \h(Q) \leq \h(P) + m\log(d+1) + d \h(y_1,\ldots,y_n).
  \end{displaymath}
  More generally, if\, $\mathcal{I}_1\sqcup\cdots\sqcup \mathcal{I}_r$
  is a partition of~$\Zint{ 1, m}$, and if~$d_k\geq 0$ denotes an
  upper bound on the total degree of~$P$ in the variables~$Y_i$ for
  $i\in \mathcal{I}_k$, then
  \begin{displaymath}
    \h(Q) \leq \h(P) + \sum_{k=1}^r (\# \mathcal{I}_k) \log(d_k +
    1 ) + \sum_{k=1}^r d_k \h\bigl((y_i)_{i\in \mathcal{I}_k}\bigr).
  \end{displaymath}
\end{prop}

\begin{proof}
  It is enough to prove the second statement. If $v\in\V_L^0$, we have
  \begin{displaymath}
    \bigl|P(y_1,\ldots,y_m, Y_{m+1},\ldots, Y_n) \bigr|_v \leq |P|_v \prod_{k=1}^r
    \Bigl (\max\bigl\{1, \max_{i\in \mathcal{I}_k} |y_i|_v \bigr\}\Bigr)^{d_k}.
  \end{displaymath}
  If $v\in\V_L^\infty$, the same estimate holds after multiplying the
  right hand side by the number of possible monomials in
  $Y_1,\ldots, Y_m$, which is
  \begin{displaymath}
    \prod_{k=1}^r (d_k + 1)^{\# \mathcal{I}_k}.
  \end{displaymath}
  Taking logarithms and summing gives the result.
\end{proof}

As a consequence, we can bound the height of a monic polynomial by the
height of its roots.

\begin{prop}
  \label{prop:pol-roots}
  Let $Q\in L[Y]$ be monic of degree~$d$, and let
  $\alpha_1,\dots,\alpha_d$ be its roots in the algebraic
  closure of~$L$. Then
  \begin{displaymath}
    \h(Q)\leq \sum_{i=1}^d \h(\alpha_k) + d\log 2.
  \end{displaymath}
\end{prop}

\begin{proof}
  Apply Proposition~\ref{prop:multivariate-evaluation} on the multivariate
  polynomial
  \begin{displaymath}
    P = \prod_{k=1}^d (Y_{d+1} - Y_k)
  \end{displaymath}
  with $m = d$, $y_k = \alpha_k$, and $\mathcal{I}_k = \{k\}$. Since the
  coefficients of~$P$ all belong to $\{-1, 0, 1\}$, we have $\h(P) = 0$.
\end{proof}

Conversely, the height of a univariate polynomial over~$L$ controls
the height of its roots.

\begin{prop}
  \label{prop:pol-root}
  Let $P\in L[Y]\backslash\{0\}$, and let~$\alpha$ be a root of~$P$.  Then
  \begin{displaymath}
    h(\alpha) \leq h(P) + \log(2).
  \end{displaymath}
\end{prop}

\begin{proof}
  We reproduce the proof given in a lecture by F.~Pazuki. We can
  assume that~$P$ is monic.  Let~$v\in \V_L$. We want to show that
  $|\alpha|_v\leq |P|_v$ if~$v\in\V_L^0$, and $|\alpha|_v\leq 2|P|_v$
  if~$v\in \V_L^\infty$. Since~$P$ is monic, we always
  have~$|P|_v\geq 1$. Write~$P = X^n + \sum_{k=0}^{n-1} c_k Y^k$, for
  some~$n\geq 1$.

  If~$v\in \V_L^0$, we can assume that~$|\alpha|_v\geq 1$. Then
  \begin{displaymath}
    |\alpha|_v = \left|\sum_{i=0}^{n-1} c_k \alpha^k\right|_v \leq |P|_v |\alpha|_v^{n-1},
  \end{displaymath}
  so~$|\alpha|_v\leq |P|_v$.

  If~$v\in \V_L^\infty$, we can assume that~$|\alpha|_v\geq 2$. Then,
  by the triangle inequality, we obtain
  \begin{displaymath}
    |\alpha|_v\leq |P|_v |\alpha_v|^{n-1}\left(1 + \frac{1}{|\alpha|_v} 
      + \cdots +\frac{1}{|\alpha|_v^{n-1}}\right) \leq 2|\alpha|_v^{n-1}|P|_v,
  \end{displaymath}
  so~$|\alpha|_v\leq 2|P|_v$. Taking logarithms and summing over all
  places of~$L$ yields the result.
\end{proof}

We now turn to the more difficult questions of giving upper bounds on
the height of a polynomial or rational fraction from its values at
special points. Our choice is to consider (almost) consecutive
integers.

\begin{prop}[({\cite[Prop.~1.1]{kieffer_UpperBoundsHeights2020}})]
  \label{prop:main-poly}
  Let $\Zint{A,B}$ be an interval in~$\Z$.  Write $D=B-A$ and
  $M=\max\{\abs{A},\abs{B}\}$.  Let~$d\geq 1$, let~$P\in L[Y]$ be a
  univariate polynomial of degree at most~$d$, let~$N\geq d+1$, and
  let $y_1,\ldots,y_N$ be distinct elements of~$\Zint{A,B}$.
  Let~$H\geq 0$, and assume that $\h(P(y_i))\leq H$ for every
  $1\leq i\leq N$. Then we have
  \begin{displaymath}
    \h(P) \leq \frac{N}{N-d} H + D\log(D) + d\log(2M) + \log(d+1).
  \end{displaymath}
\end{prop}

Note that the bound on~$\h(P)$ is of the order of~$dH$ when~$N=d+1$,
as suggested by the Lagrange interpolation formula. On the other hand,
if we take for instance~$N=2d$, then the bound on~$\h(P)$ is roughly
in~$O(H)$. This remark will be crucial in~§\ref{sub:height-end}, when
we consider the evaluation of multivariate polynomials in each
variable successively.

\begin{prop}[({\cite[Prop.~1.2]{kieffer_UpperBoundsHeights2020}})]
  \label{prop:main-frac}
  Let $\Zint{A,B}$ be an interval in~$\Z$.  Write $D=B-A$ and
  $M=\max\{\abs{A},\abs{B}\}$. Let~$d\geq 1$, and let~$\Frac\in L(Y)$
  be a univariate rational fraction of degree at most~$d$. Let~$S$ be
  a subset of~$\Zint{A,B}$ containing no poles of~$\Frac$,
  let~$\eta\geq 1$, and let $H\geq \max\{4, \log(2M)\}$.  Assume that
  \begin{enumerate}
  \item $\h(\Frac(y))\leq H$ for every $y\in S$.
  \item $S$ contains at least $D/\eta$ elements.
  \item $D \geq \max\{\eta d^3 H, 4 \eta d d_L\}$.
  \end{enumerate}
  Then we have
  \begin{displaymath}
    \h(F) \leq H + C_L\eta d\log (\eta dH) + d \log(2M) + \log(d+1),
  \end{displaymath}
  where~$C_L$ is a constant depending only on~$L$. We can take
  $C_\Q=960$.
\end{prop}

The bound on~$\h(\Frac)$ given in Proposition~\ref{prop:main-frac} is
roughly in~$O(H)$ as well, but the number of evaluation points that we
have to consider is bounded from below in terms of~$H$.

\subsection{Heights of abelian varieties}
\label{sub:heights-AV}

We fix a PEL setting as in~§\ref{sub:modeq}, and keep the notation
used there. We also write $\Scal = \Gamma\backslash X_+$,
where~$\Gamma$ is a subgroup of~$G(\Q)_+$.

Different types of heights can be defined for an abelian variety~$A$
over~$\Qbar$. The \emph{Faltings height}~$h_F(A)$ is defined in
\cite[§3]{faltings_EndlichkeitssaetzeFuerAbelsche1983} in terms of
Arakelov degrees of metrized line bundles on~$A$. If~$A$ is given a
principal polarization~$\Lcal$, and $r\geq 2$ is an even integer, we
can also define the \emph{Theta height of level~$r$} of~$(A,\Lcal)$,
denoted by~$h_{\Theta,r}(A,\Lcal)$, as the projective height of
level~$r$ theta constants of~$(A,\Lcal)$
\cite[Def.~2.6]{pazuki_ThetaHeightFaltings2012}. Finally, if~$A$ is an
abelian variety with PEL structure over~$\Qbar$ given by a
point~$z\in \Scal$ where $j_1,\ldots, j_{n+1}$ are well defined, we
can define the \emph{$j$-height} of~$A$ as
\begin{displaymath}
  h_j(A) = h\bigl(j_1(A),\ldots, j_{n+1}(A)\bigr).
\end{displaymath}
We also write $\Hbar_F(A) = \max\{1,h_F(A)\}$ and define
$\Hbar,\Hbar_{\Theta,r}$, and~$\Hbar_j$ similarly.

The goal of this section is to relate the~$j$-heights of isogenous
abelian varieties, under mild conditions related to the geometry of
the moduli space. Such a relation is known for instance in the case of
elliptic curves, taking the usual $j$-invariant as coordinate
\cite[Thm.~1.1]{pazuki_ModularInvariantsIsogenies2019}. To this end,
we relate the $j$-height with the Faltings height, since the latter
behaves well with respect to isogenies. Theta heights are an
intermediate step between concrete values of invariants and the
Faltings height. More precisely, we use the two following results.

\begin{prop}
  \label{prop:faltings-isog}
  Let~$A$,~$A'$ be abelian varieties over~$\Qbar$, and assume that an
  isogeny $\varphi\from A\to A'$ exists. Then
  \begin{displaymath}
    \bigl |h_F(A) - h_F(A') \bigr| \leq \dfrac{1}{2} \log(\deg \varphi).
  \end{displaymath}
\end{prop}

\begin{proof}
  This is a consequence of
  \cite[Lem.~5]{faltings_EndlichkeitssaetzeFuerAbelsche1983}.
\end{proof}

\begin{thm}[{(\cite[Cor.~1.3]{pazuki_ThetaHeightFaltings2012})}]
  \label{thm:theta-faltings}
  For every~$g\geq 1$, and every even~$r\geq 2$, there exists a
  constant~$C(g,r)$ such that the following holds. Let~$(A,\Lcal)$ be
  a principally polarized abelian variety of dimension~$g$ defined
  over~$\Qbar$. Then
  \begin{displaymath}
    \Bigl|\Hbar_{\Theta,r}(A,\Lcal) - \dfrac{1}{2}\Hbar_F(A)\Bigr| \leq
    C(g,r) \log \bigl(\min\{\Hbar_F(A),
    \Hbar_{\Theta,r}(A,\Lcal)\}
    +2 \bigr).
  \end{displaymath}
  We can take
  \begin{displaymath}
    C(g,r) = 1000 r^{2g} \log^5(r^{2g}).
  \end{displaymath}
\end{thm}

\subsection{Relating the $j$-height and the Faltings height}
\label{sub:j-faltings}

Using Theorem~\ref{thm:theta-faltings}, we can prove that the
$j$-height and the Faltings height of a generic abelian variety with
PEL structure are related.

\begin{prop}
  \label{prop:j-faltings}
  There exists a nonzero polynomial~$P\in L[Y_1,\ldots Y_{n+1}]$ and a
  positive constant~$C$ such that the following holds: if~$A$ is the
  abelian variety with PEL structure associated with a
  point~$z\in \Scal$ where $j_1,\ldots, j_{n+1}$ are well defined and
  $P(j_1,\ldots,j_{n+1}) \neq 0$, and if~$A$ is defined over~$\Qbar$, then
  \begin{displaymath}
    \dfrac{1}{C} \Hbar_F(A) \leq
    \Hbar_j(A) \leq C \Hbar_F(A).
  \end{displaymath}
\end{prop}

\begin{proof}
  By \cite[Thm.~5.17]{milne_IntroductionShimuraVarieties2005}, we can
  write $\Scal = \Gamma'\backslash X_+$ where~$\Gamma'$ is a
  congruence subgroup of~$G^\mathrm{der}$. Since
  $G^\mathrm{der}\subset \ker(\det)$, it embeds into~$\GSp_{2g}(\Q)$,
  where $2g=\dim_\Q V$. Therefore, by
  \cite[Thm.~5.16]{milne_IntroductionShimuraVarieties2005}, we can
  find a congruence subgroup~$\Gamma''$ of~$G^\mathrm{der}$ and an
  even integer~$r\geq 4$ such that~$\Gamma''\backslash X_+$ embeds in
  the moduli space~$\mathcal{A}_{\Theta,r}$ of principally polarized
  abelian varieties of dimension~$g$ with level~$r$ Theta
  structure. We have a diagram
  \begin{equation}
    \label{diag:j-faltings}
    \begin{tikzcd}
      & \widetilde{\Scal} = \widetilde{\Gamma}\backslash X_+
      \ar{ld}[swap]{p'}
      \ar{rd}{p''} & & \\
      \Scal = \Gamma'\backslash X_+ & & \Scal'' = \Gamma''
      \backslash X_+ \ar[r, hook, "\iota"] &
      \mathcal{A}_{\Theta, r}
    \end{tikzcd}
  \end{equation}
  where $\widetilde{\Gamma} = \Gamma'\cap\Gamma''$. The maps~$p'$
  and~$p''$ are finite coverings. All the varieties and maps in this
  diagram are defined over~$\Qbar$.

  The modular interpretation of diagram~\eqref{diag:j-faltings} is the
  following. Let~$(\Lambda, \psi)$ be the standard polarized lattice
  associated with the connected component~$\Scal$, as in
  Proposition~\ref{prop:modular-int-lattices}. We can find a
  sublattice~$\Lambda''\subset \nolinebreak \Lambda$, and
  $\lambda\in \Q^\times$ such that~$(\Lambda'',\lambda\psi)$ is
  principally polarized. A point~$z\in \Scal$ defines a complex
  structure~$x$ on $\Lambda\otimes\R = V(\R)$, up to action
  of~$\Gamma$. Lifting~$z$ to $\widetilde{z}\in\widetilde{\Scal}$
  corresponds to considering~$x$ up to action of~$\widetilde{\Gamma}$
  only, and this group leaves~$\Lambda''$ and its level~$r$ Theta
  structure stable. Then the image of~$\widetilde{z}$
  in~$\mathcal{A}_{\Theta, r}$ is then given by
  $(\Lambda'', x, \lambda\psi)$.

  In particular, if $\widetilde{z}\in \widetilde{\Scal}$, and
  if~$A$ and~$A''$ are the abelian varieties corresponding to the
  points~$p'(\widetilde{z})\in \Scal$
  and~$\iota \circ p''(\widetilde{z})\in \mathcal{A}_{\Theta, r}$
  respectively, then~$A$ and~$A''$ are linked by an isogeny of
  degree~$d = \#(\Lambda/\Lambda'')$.  Hence, by
  Proposition~\ref{prop:faltings-isog} and
  Theorem~\ref{thm:theta-faltings}, we have
  \begin{align*}
    \bigl|\Hbar_F(A) - 2\Hbar_{\Theta,r}(A'')\bigr|
    &\leq
      \dfrac{\log(d)}{2} + C(g,r) \log\left(\min \{\Hbar_F(A),
      \Hbar_{\Theta,r}(A'')\} + 2 + \dfrac{\log(d)}{2}\right) \\
    &\leq C_F \min\{\Hbar_F(A), \Hbar_{\Theta,r}(A'')\}
  \end{align*}
  with~$C_F = (2 + \log(d)) C(g,r)$. Therefore
  \begin{align}
    \label{eq:h-faltings}
    \Hbar_F(A) \leq (2 + C_F) \Hbar_{\Theta,r}(A''),\qquad
    \Hbar_{\Theta,r}(A'') \leq \frac{1+C_F}{2} \Hbar_F(A).
  \end{align}

  Now we relate the Theta height and the~$j$-height using relation
  between modular functions; the genericity hypothesis encoded in the
  polynomial~$P$ appears in this step. Denote by
  $\theta_0,\ldots, \theta_k$ the Theta constants of level~$r$. They
  define a projective embedding of~$\mathcal{A}_{\Theta, r}$,
  therefore the pullbacks of
  $\theta_1/\theta_0,\ldots, \theta_k/\theta_0$ generate the function
  field of~$\Scal''$. By definition, $j_1,\ldots, j_{n+1}$ are
  coordinates on~$\Scal$. To ease notation, we identify all these
  functions with their pullbacks to~$\widetilde{\Scal}$.

  By the primitive element theorem, there exists a function~$f$
  on~$\widetilde{\Scal}$ such that both $(j_1,\ldots,j_{n+1},f)$
  and~$(\theta_1/\theta_0,\ldots, \theta_k/\theta_0, f)$ are
  generating families for the function field of~$\widetilde{\Scal}$
  over~$\Qbar$. We choose
  polynomials
  \begin{displaymath}
    P_J\in \Qbar[Y_1,\ldots,Y_{n+1},X] \quad\text{and}\quad P_\Theta\in \Qbar[Y_1,\ldots,Y_k,X]
  \end{displaymath}
  such that~$P_J(j_1,\ldots,j_{n+1},X)$
  and~$P_\Theta(\theta_1/\theta_0,\ldots,\theta_k/\theta_0,X)$ are
  (non necessarily monic) minimal polynomials of~$f$ over the function
  fields of~$\Scal$ and~$\Scal''$ respectively. We also choose
  polynomials $N_{J,i}, D_{J,i}\in \Qbar[Y_1,\ldots,Y_k, X]$
  for each~$1\leq i\leq n+1$, and
  $N_{\Theta,i}, D_{\Theta,i} \in \Qbar[Y_1,\ldots,Y_{n+1},X]$
  for each~$1\leq i\leq k$, such that the following equalities hold
  on~$\widetilde{\Scal}$:
  \begin{align*}
    j_i &= \frac{N_{J,i}}{D_{J,i}}(\theta_1/\theta_0,\ldots,\theta_k/\theta_0,f)
          &\text{for each } 1\leq i\leq n+1, \text{ and}\\
    \theta_i/\theta_0 &= \frac{N_{\Theta,i}}{D_{\Theta,i}}(j_1,\ldots,j_{n+1},f)
                        &\text{for each } 1\leq i\leq k.
  \end{align*}
  Let~$\widetilde{F}$ be the smallest Zariski closed subset
  of~$\widetilde{\Scal}$ such that outside~$\widetilde{F}$, the
  following properties are all satisfied:
  \begin{itemize}
  \item all the functions~$f$, $j_i$ for~$1\leq i\leq n+1$ and
    $\theta_i/\theta_0$ for~$1\leq i\leq k$ are well defined;
  \item the polynomials~$P_J(j_1,\ldots,j_{n+1},X)$
    and~$P_\Theta(\theta_1/\theta_0,\ldots,\theta_k/\theta_0, X)$ do
    not vanish;
  \item the
    quantities~$D_{J,i}(\theta_1/\theta_0,\ldots,\theta_k/\theta_0,f)$
    for~$1\leq i\leq k$ and $D_{\Theta,i}(j_1,\ldots,j_{n+1},f)$
    for~$1\leq i\leq k$ do not vanish.
  \end{itemize}
  Then~$\widetilde{F}$ has codimension~$1$ in~$\widetilde{\Scal}$,
  hence $\U = \Scal\backslash p'(\widetilde{F})$ is open dense in
  $\Scal$. Let $P\in L[j_1,\ldots, j_{n+1}]$ be a polynomial such that
  $\{P\neq 0\}\subset \U$.

  Let $z\in \Scal$ be a point where $j_1,\ldots, j_{n+1}$ are well
  defined, take values in~$\Qbar$, and satisfy
  $P(j_1,\ldots, j_{n+1})\neq 0$. We look at the
  diagram~\eqref{diag:j-faltings}, from left to right. Lift~$z$ to a
  point~$\widetilde{z}\in \widetilde{\Scal}$; by construction,
  $\widetilde{z}\notin \widetilde{F}$. By
  Propositions~\ref{prop:multivariate-evaluation}
  and~\ref{prop:pol-root}, we have
  \begin{equation}
    \label{eq:h-f}
    \Hbar\bigl(j_1(\widetilde{z}), \ldots,
    j_{n+1}(\widetilde{z}), f(\widetilde{z})\bigr)
    \leq C \Hbar\bigl(j_1(z),\ldots, j_{n+1}(z)\bigr)
  \end{equation}
  with~$C = \h(P_J) + (n+1)\log(d_J+1) +d_J+1$, where~$d_J$ denotes the
  total degree of~$P_J$ in~$Y_1\ldots,Y_{n+1}$. Writing
  $z'' = p''(\widetilde{z})$, we also have for every $1\leq i\leq k$,
  \begin{equation}
    \label{eq:h-theta}
    \Hbar(\theta_i/\theta_0(\widetilde{z})) \leq C \Hbar\bigl(j_1(\widetilde{z}), \ldots,
    j_{n+1}(\widetilde{z}), f(\widetilde{z})\bigr)
  \end{equation}
  with
  \begin{align*}
    C &= \h(N_{\Theta,i}) + \h(D_{\Theta,i}) +
    (n+2)\bigl(\log(\deg(N_{\Theta,i})+1)+\log(\deg(D_{\Theta,i})+1)\bigr)\\
    &\quad
    + \deg(N_{\Theta,i}) + \deg(D_{\Theta,i}),
  \end{align*}
  where~$\deg$ denotes the
  total degree. Combining equations~\eqref{eq:h-f}
  and~\eqref{eq:h-theta}, we obtain
  \begin{displaymath}
    \Hbar\Bigl(\dfrac{\theta_1}{\theta_0}(z''), \ldots,
    \dfrac{\theta_k}{\theta_0}(z'')\Bigr)\leq C_\Theta \Hbar\bigl(j_1(z), \ldots,
    j_{n+1}(z)\bigr)
  \end{displaymath}
  where~$C_\Theta$ has an explicit expression in terms of the heigts
  and degrees of the polynomials~$P_J$ and~$N_{\Theta,i},D_{\Theta,i}$
  for~$1\leq i\leq k$.  Equivalently, in the notation above, we have
  \begin{displaymath}
    \Hbar_{\Theta,r}(A'')\leq C_\Theta \Hbar_j(A),
  \end{displaymath}
  so by~\eqref{eq:h-faltings}
  \begin{displaymath}
    \Hbar_F(A)\leq (2+C_F)C_\Theta \Hbar_j(A).
  \end{displaymath}
  Going through the diagram from right to left
  gives the reverse inequality
  \begin{displaymath}
    \Hbar_j(A) \leq \frac{(1+C_F)C_J}{2} \Hbar_F(A)
  \end{displaymath}
  where~$C_J$ is defined in a similar way to~$C_\Theta$ in terms of
  the polynomials~$P_\Theta$ and~$N_{J,i}, D_{J,i}$
  for~$1\leq i\leq n+1$.
\end{proof}

Assume that the integers~$r$ and~$d$, the modular function~$f$, as
well as the polynomials~$P_J,P_\Theta,N_{J,i},D_{J,i},N_{\Theta,i},$
and~$D_{\Theta,i}$ can be explicitly determined. Then both the
polynomial~$P$ and the constant~$C$ in
Proposition~\ref{prop:j-faltings} can be determined explicitly as
well. We will do this computation in a slightly different way
in~§\ref{sub:height-explicit} in the case of Igusa invariants on the
Siegel threefold.

From now on, we define~$\U$ to be the Zariski open set in~$\Scal$ where
$j_1,\ldots, j_{n+1}$ are well defined and
$P(j_1,\ldots, j_{n+1})\neq 0$.

\begin{cor}
  \label{cor:j-isog}
  Let~$C$ be the constant from Proposition~\ref{prop:j-faltings},
  let~$z$ and~$z'$ be points of~$\U$ and let~$A$ and~$A'$ be the
  abelian varieties with PEL structure associated with~$z$ and~$z'$
  respectively. Assume that~$A$ and~$A'$ are defined over~$\Qbar$, and
  are linked by an isogeny of degree~$d$. Then
  \begin{displaymath}
    \Hbar_j(A') \leq C^2(\Hbar_j(A) + \log d).
  \end{displaymath}
\end{cor}

\begin{proof}
  Combine Propositions~\ref{prop:faltings-isog}
  and~\ref{prop:j-faltings}.
\end{proof}

\begin{rem}
  We can presumably do better than Corollary~\ref{cor:j-isog}. For
  instance, when studying $j$-invariants of isogenous elliptic curves,
  one can prove that $|h(j(E)) - h(j(E'))|$ is bounded by logarithmic
  terms \cite[Thm.~1.1]{pazuki_ModularInvariantsIsogenies2019}. This
  is also the kind of bound provided by
  Theorem~\ref{thm:theta-faltings}. The rough estimate in
  Corollary~\ref{cor:j-isog} is sufficient for our purposes, but has
  the drawback that the constants we derive from it are very
  pessimistic.
\end{rem}

\subsection{Heights of evaluated modular equations}
\label{sub:interpolation}

Let~$\U$ (resp.~$\U'$) be an open set of~$\Scal$ (resp.~$\Tcal$) where
a relation between the~$j$-height and the Faltings height holds, as in
Proposition~\ref{prop:j-faltings}. Define $\U_\delta\subset \Scal$ to
be the Zariski open set of all points~$[x,g]\in \Scal$ such that
$[x,g]\in \U$, and moreover the images of~$[x,g]$ under the
(symmetrized) Hecke correspondence~$H_\delta$ all lie in~$\U'$: in
other words $[\sigma(x), \sigma(g k \delta)]\in \U'$ for every
$(k,\sigma)\in \nolinebreak K_0/K_{n+1}$, in the notation
of~§\ref{sub:modeq}. Finally, we define~$\V_\delta\subset L^n$ to be
the Zariski open set of all points~$(j_1,\ldots, j_n)$ where the
equation~\eqref{eq:jn+1} given by~$E(j_1,\ldots,j_n, J_{n+1})$ has~$e$
distinct roots and the following property holds: if~$j_{n+1}$ is a
root of~\eqref{eq:jn+1}, then $(j_1,\ldots, j_{n+1})$ are the
invariants of some point $z\in \U_\delta$.  In particular, the modular
equations~$\Psi_{\delta,m}$ do not have poles on~$\V_\delta$.

\begin{lem}
  \label{lem:equation-V}
  There exist a positive constant~$C$ independent of~$\delta$, and a
  nonzero polynomial $P_\delta \in L[J_1,\ldots,J_n]$ of total degree
  at most~$C \Heckedeg(\delta)$ such that
  $\{P_\delta(j_1,\ldots,j_n)\neq 0\}\subset \V_\delta$.
\end{lem}

\begin{proof}
  Let~$E\in L[J_1,\ldots,J_{n+1}]$ be the polynomial defined
  in~§\ref{sub:modeq}, of degree~$e$ in~$J_{n+1}$, so that the
  equation satisfied by~$j_{n+1}$ on~$\Scal$ takes the
  form~$E(j_1,\ldots,j_{n+1}) = 0$.
  
  Let~$R$ be the the resultant of~$E$ and its derivative with respect
  to~$J_{n+1}$. If~$R$ does not vanish at $(j_1,\ldots, j_n)\in L^n$,
  then the polynomial $E(j_1,\ldots,j_n,J_{n+1})\in L[J_{n+1}]$
  has~$e$ distinct roots.

  Similarly, there is a polynomial~$Q\in L[J_1,\ldots, J_{n+1}]$ such
  that every tuple $(j_1,\ldots,j_{n+1})$ satisfying~\eqref{eq:jn+1}
  and such that $Q(j_1,\ldots,j_{n+1})\neq 0$ lies in the image
  of~$\Scal$. Let~$R'$ be the resultant of~$Q$ and~$E$ with respect
  to~$J_{n+1}$. If~$R'$ does not vanish at $(j_1,\ldots, j_n)$, then
  for every root~$j_{n+1}$ of~$E(j_1,\ldots,j_n,J_{n+1})$, the tuple
  $(j_1,\ldots,j_{n+1})$ lies in the image of~$\Scal$.

  Let~$\lambda, \lambda'$ be symmetric modular forms on~$\Scal$
  and~$\Tcal$ respectively, defined over~$L$, such that
  $\{\lambda\neq 0\}\subset \U$ and $\{\lambda'\neq 0\}\subset
  \U'$. These modular forms can be chosen independently of~$\delta$. As
  in~§\ref{sub:denom}, we construct the modular form
  \begin{displaymath}
    \lambda^\delta = \lambda \prod_{\gamma\in K_0/K'} \act{\gamma}{\lambda'_\delta}
  \end{displaymath}
  where~$\lambda'_\delta$ is the modular
  form~$[x,g]\mapsto \lambda'([x,g\delta])$ of level~$K'$. The modular
  form~$\lambda^\delta$ is defined over~$L$ and has weight
  \begin{displaymath}
    \wt(\lambda^\delta) = \wt(\lambda) + (\#\Sigma) \Heckedeg(\delta) \wt(\lambda').
  \end{displaymath}
  
  Modular forms realize a projective embedding of~$\Scal$ by
  Theorem~\ref{thm:mf-algebraic}; therefore, possibly after increasing
  the weight by a constant independent of~$\delta$, we can find a
  symmetric modular form~$\xi$ defined over~$L$ such that
  $\wt(\lambda^\delta) = \wt(\xi)$ and the divisors
  of~$\lambda^\delta$ and~$\xi$ have no common codimension~1
  components.  By Proposition~\ref{prop:gc}, if we write
  \begin{displaymath}
    \dfrac{\lambda^\delta}{\xi} = \sum_{k=0}^{e-1} R_k(j_1,\ldots, j_n) j_{n+1}^{\,k}
    \qquad \text{where } R_k\in L(J_1,\ldots,J_n),
  \end{displaymath}
  then $\deg R_k\leq \GC(j_1,\ldots,j_{n+1})\wt(\lambda^\delta)$ for
  every~$0\leq k\leq e-1$. Taking the resultant of the
  polynomials~$\sum R_k J_{n+1}^k$ and~$E$ with respect to~$J_{n+1}$
  yields a rational fraction~$R''\in L(J_1,\ldots,J_n)$ of total
  degree at most
  \begin{displaymath}
    (e-1) d_E + e\max_{0\leq k\leq e-1} \deg(R_k),
  \end{displaymath}
  where~$d_E$ denotes the total degree of~$E$ in~$j_1,\ldots,j_n$.
  If~$R'$,~$R''$ are well defined and do not vanish at
  $(j_1,\ldots,j_n)$, then for every root~$j_{n+1}$
  of~\eqref{eq:jn+1}, the tuple~$(j_1,\ldots,j_{n+1})$ comes from a
  point~$z\in \U_\delta$.

  We take~$P_\delta$ to be the product of~$R$, $R'$, and the numerator
  of~$R''$. The polynomials~$R$ and~$R'$ are independent of~$\delta$,
  and the degree of~$R''$ is bounded above linearly
  in~$\Heckedeg(\delta)$.
\end{proof}

If upper bounds on the degree of equations defining~$\U$ and~$\U'$ are
explicitly known, together with the polynomials~$E$ and~$Q$, then the
proof of Lemma~\ref{lem:equation-V} allows us to determine a valid
constant~$C$ explicitly.

\begin{prop}
  \label{prop:height-modeq-evaluation}
  There exists a constant~$C$, independent of~$\delta$, such that the
  following holds.  Let $(j_1,\ldots,j_{n})\in \V_\delta$, and let
  $1\leq m\leq n+1$. Then
  \begin{displaymath}
    h \bigl(\Psi_{\delta,m}(j_1,\ldots,j_n) \bigr) \leq
    C \Heckedeg(\delta)
    \bigl(\Hbar(j_1,\ldots,j_n) + \log\isogdeg(\delta) \bigr).
  \end{displaymath}
\end{prop}

\begin{proof}
  Let~$\mathcal{J}$ be the set of roots of equation~\eqref{eq:jn+1} at
  $(j_1,\ldots,j_n)$, and let $j_{n+1}\in \mathcal{J}$. Let~$[x,g]$ be
  a point of~$\Scal$ describing an abelian variety~$A$ with PEL
  structure whose invariants are $(j_1,\ldots,j_{n+1})$. For
  every~$\sigma\in \Sigma$, denote by~$A_\sigma$ the abelian variety
  with PEL structure associated with the
  point~$[\sigma(x),\sigma(g)]$. Then for every
  $\gamma = (\sigma,k)\in K_0/K_m$, the
  point~$[\sigma(x),\sigma(gk\delta)]$ describes an abelian
  variety~$A_\gamma$ which is related to~$A_\sigma$ by an isogeny of
  degree~$\isogdeg(\sigma(\delta)) = \isogdeg(\delta)$, by
  Corollary~\ref{cor:hecke-isog}. Therefore, by
  Corollary~\ref{cor:j-isog}, we have
  \begin{displaymath}
    \Hbar \bigl(\act{\gamma}{j_{1,\delta}}([x,g]),\ldots,
    \act{\gamma}{j_{n+1,\delta}}([x,g]) \bigr)
    \leq C (\Hbar \bigl(j_1,\ldots, j_{n+1}) + \log \isogdeg(\delta) \bigr).
  \end{displaymath}
  where the constant~$C$ is positive and independent of~$\delta$. By
  Definition~\ref{def:modeq}, the
  polynomial $\Psi_{\delta,m}(j_1,\ldots,j_n,j_{n+1})\in
  L[Y_1,\ldots,Y_m]$ is the evaluation of a certain multivariate
  polynomial at the values~$\gamma\cdot j_{i,\delta}([x,g])$, for
  $1\leq i\leq m$ and $\gamma\in K_0/K_i$, each appearing with
  degree~1. The number of such values is
  \begin{displaymath}
    d_1 + d_1d_2 + \cdots + d_1\cdots d_m \leq m\,(\#\Sigma) \Heckedeg(\delta).
  \end{displaymath}
  Therefore, by Proposition~\ref{prop:multivariate-evaluation}, we have
  \begin{displaymath}
    \begin{aligned}
      h\bigl(\Psi_{\delta,m}(j_1,\ldots,j_{n+1})\bigr)
      & \leq m \, (\#\Sigma)
      \Heckedeg(\delta) \log(2) + m\, (\#\Sigma) \Heckedeg(\delta)\,
      C \bigl(\Hbar(j_1,\ldots,j_{n+1}) +
      \log\isogdeg(\delta) \bigr) \\
      &\leq C' \Heckedeg(\delta) \bigl(\Hbar(j_1,\ldots,j_{n+1}) +
      \log\isogdeg(\delta) \bigr).
    \end{aligned}
  \end{displaymath}
  where~$C$ and~$C'$ denote explicit constants independent
  of~$\delta$.  In order to obtain $\Psi_{\delta,m}(j_1,\ldots,j_n)$,
  we interpolate a polynomial of degree~$e-1$ in~$j_{n+1}$
  where~$\mathcal{J}$ is the set of interpolation points. By
  Propositions~\ref{prop:multivariate-evaluation}
  and~\ref{prop:pol-root}, we have
  \begin{displaymath}
    h(j_{n+1}) \leq C \Hbar(j_1,\ldots,j_n) \qquad\text{for every }
    j_{n+1}\in \mathcal{J},
  \end{displaymath}
  where~$C$ is a constant independent on~$\delta$. The result follows
  by applying Proposition~\ref{prop:main-poly} with~$N=d+1$.
\end{proof}

The proof of Proposition~\ref{prop:height-modeq-evaluation} provides
an explicit value of~$C$ if the constant from
Corollary~\ref{cor:j-isog} is known.

\subsection{Heights of coefficients of modular equations}
\label{sub:height-end}

We are ready to prove upper bounds on the heights of modular equations
(the second part of Theorem~\ref{thm:main}) using
Proposition~\ref{prop:height-modeq-evaluation} and the results on
heights of fractions given in~§\ref{sub:height-frac}. From now on, we
add subscripts to constants: for instance~$C_{\ref{prop:j-faltings}}$
denotes a constant \emph{larger than~$1$} such that
Proposition~\ref{prop:j-faltings} holds with this value
of~$C$. Moreover, we denote by~$C_{\log}$ a constant independent
of~$\delta$ such
that~$\log \Heckedeg(\delta)\leq
C_{\log}\max\{1,\log\isogdeg(\delta)\}$. By
Proposition~\ref{prop:d-l-relation}, we can take
$C_{\log} = (\dim V)^2 + \log(C_{\ref{prop:d-l-relation}})$, where~$V$
denotes the~$\Q$-vector space defining the PEL datum.

\begin{defn}
  \label{def:tree}
  We call an \emph{$(n, N_1, N_2)$-evaluation tree} a rooted tree of
  depth~$n$, arity~$N_1$ at depths~$0,\ldots,n-2$, and arity~$N_2$ at
  depth~$n-1$, such that every vertex but the root is labeled by an
  element of~$\Z$ and the sons of every vertex are distinct.

  Let~$T$ be an $(n,N_1,N_2)$-evaluation tree, and let
  $1\leq k\leq n$. The~$k$-th \emph{evaluation set}~$\mathcal{I}_k(T)$
  of~$T$ is the set of points~$(y_1,\ldots,y_k)\in\Z^k$ such
  that~$y_1$ is a son of the root, and~$y_{i+1}$ is a son of~$y_i$ for
  every $1\leq i\leq k-1$. We say that~$T$ is \emph{bounded by~$M$} if
  the absolute value of every vertex is bounded above by~$M$. We say
  that~$T$ has \emph{amplitude~$(D_1, D_2)$} if for every vertex~$y$
  of depth~$0\leq r\leq n-2$ (resp.\ depth~$n-1$) in~$T$, the sons
  of~$y$ lie in an integer interval of amplitude at most~$D_1$
  (resp.~$D_2$); by definition, the amplitude of~$\Zint{A,B}$
  is~$B-A$.

  Let~$T$ be an $(n,N_1,N_2)$-evaluation tree, let
  $a = (a_1,\ldots,a_n)\in\Z^n$, and let~$M\geq 1$ be an
  integer. Let~$\Frac$ be a coefficient of~$\Psi_{\delta,m}$ for some
  $1\leq m\leq n+1$, seen as a polynomial in the
  variables~$J_{n+1},Y_1,\ldots,Y_m$;
  hence~$\Frac\in L(J_1,\ldots,J_n)$. Write $\Frac = P/Q$ in
  irreducible form, and let $d = \deg(\Frac)$; assume that $d\geq
  1$. We say that~$T,a$ and~$M$ are \emph{valid evaluation data
    for~$\Frac$} if the following conditions are satisfied:
  \begin{enumerate}
  \item \label{cd:large-M} $T$ and~$a$ are bounded by~$M$
  \item We have
      $M \geq 2 B \log^2(B+1)$, 
    where
    \begin{displaymath}
      B= 4 C_{\ref{thm:degree-bound-proved}}^3 C_{\ref{prop:height-modeq-evaluation}}
      \Heckedeg(\delta)^4 \max\{1,\log \isogdeg(\delta)\}.
    \end{displaymath}
  \item $N_1 = 2d$ and $N_2\geq M$.
  \item $T$ has amplitude~$(4d, 2M)$.
  \item \label{cd:V-delta}
    For every $(y_1,\ldots,y_n)\in \mathcal{I}_n(T)$, the point
    \begin{displaymath}
      (j_1,\ldots,j_n) = (y_1y_n + a_1,\ldots, y_{n-1}y_n +
      a_{n-1}, y_n + a_n)
    \end{displaymath}
    belongs to~$\V_\delta$.
  \item \label{cd:coprime} For every
    $(y_1,\ldots,y_{n-1})\in \mathcal{I}_{n-1}(T)$, the two
    polynomials~$P$ and~$Q$ evaluated at the tuple
    $(y_1 Y + a_1,\ldots, y_{n-1}Y + a_{n-1}, Y + a_n)$ are coprime
    in~$L[Y]$.
  \item $Q(a_1,\ldots,a_n)\neq 0$.
  \end{enumerate}
\end{defn}

\begin{lem}
  \label{lem:interpolation-data}
  There exists a constant~$C$, independent of~$\delta$, such that the
  following holds.  Let~$\Frac$ be a coefficient of~$\Psi_{\delta,m}$
  of degree $d\geq 1$. Then there exist valid evaluation
  data~$(T, a, M)$ for~$\Frac$ such that
  \begin{equation}
    \label{eq:choice-M}
    C \Heckedeg(\delta)^4 \max\{1,\log^3(\isogdeg(\delta))\}\leq M<
    C \Heckedeg(\delta)^4 \max\{1,\log^3(\isogdeg(\delta))\} +1
  \end{equation}
  and~$M\geq 4d[L:\Q]$. We can take
  \begin{displaymath}
    C = \max\{C_1,C_2,C_3\}
  \end{displaymath}
  where
  \begin{align*}
    C_1 &= 24 C_{\ref{thm:degree-bound-proved}}^3 C_{\ref{prop:height-modeq-evaluation}}
    \bigl(4 C_{\log}
    + \log( 24 C_{\ref{thm:degree-bound-proved}}^3 C_{\ref{prop:height-modeq-evaluation}}) +1\bigr),\\
    C_2 &= 14 C_{\ref{thm:degree-bound-proved}}^2 +
          5C_{\ref{lem:equation-V}}, \quad\text{and}\quad
          C_3 = 4 C_{\ref{thm:degree-bound-proved}} [L:\Q].
  \end{align*}
\end{lem}

\begin{proof}
  Let~$M$ be as in~\eqref{eq:choice-M}. Condition~\ref{cd:large-M} in
  Definition~\ref{def:tree} holds because~$C\geq C_1$.

  We start by constructing the vector~$a$. Note that~$M\geq
  2d+1$. Since~$Q$ is nonzero, and has degree at most~$d$ in~$Y_1$, we
  can find $a_1\in\Z$ such that $|a_1|\leq M$ and the
  polynomial~$Q(a_1,Y_2,\ldots,Y_n)$ is nonzero. Iterating, we find a
  vector $a = (a_1,\ldots,a_n)$ bounded by~$M$ such that
  $Q(a_1,\ldots,a_n)\neq 0$.

  We now build the evaluation tree~$T$ down from the
  root. Let~$P_\delta$ be an equation for the complement of~$\V_\delta$
  as in Lemma~\ref{lem:equation-V}, and define
  \begin{displaymath}
    R_\delta = P_\delta(Y_1 Y_n + a_1,\ldots,
    Y_{n-1} Y_n + a_{n-1}, Y_n + a_n)
  \end{displaymath}
  which is a nonzero polynomial of degree at
  most~$2 C_{\ref{lem:equation-V}} \Heckedeg(\delta)$. Let~$R$ be the
  resultant with respect to~$Y_n$ of the two polynomials
  \begin{displaymath}
    P(Y_1 Y_n + a_1,\ldots,
    Y_{n-1} Y_n + a_{n-1}, Y_n + a_n)
  \end{displaymath}
  and
  \begin{displaymath}
    Q(Y_1 Y_n + a_1,\ldots,
    Y_{n-1} Y_n + a_{n-1}, Y_n + a_n).
  \end{displaymath}
  The polynomial~$R$ is nonzero and has total degree at
  most~$4d^2$.

  We want to choose~$2d$ values of~$y_1$, lying in an interval with
  amplitude at most~$4d$, such that neither~$R_\delta$ nor~$R$
  vanishes when evaluated at $Y_1 = y_1$; this nonvanishing condition
  excludes at
  most~$4d^2+ 2 C_{\ref{lem:equation-V}} \Heckedeg(\delta)$ possible
  values of~$y_1$. At least one of the integer intervals of the
  form~$\Zint{5kd, (5k+4)d}$
  for~$0\leq k\leq 2d+C_{\ref{lem:equation-V}} \Heckedeg(\delta)/d$
  contains at least~$2d$ valid choices of~$y_1$. Then~$\abs{y_1}$ is
  always bounded above
  by~$5(2d^2+C_{\ref{lem:equation-V}} \Heckedeg(\delta))+4d\leq M$,
  because~$C\geq C_2$.

  We iterate this procedure to construct~$T$ up to depth~$n-1$ with
  the right arity, bound and amplitude, such that the evaluations
  of the polynomials~$R_\delta$ and~$R$ are nonzero at every
  point $(y_1,\ldots,y_{n-1})\in \mathcal{I}_{n-1}(T)$.

  We conclude by constructing~$n$-th level of~$T$.  Let
  $(y_1,\ldots,y_{n-1})\in \mathcal{I}_{n-1}(T)$. Then, as before, at
  most~$4d^2+ 2 C_{\ref{lem:equation-V}} \Heckedeg(\delta) \leq M$
  values for~$y_n$ are forbidden as they make either~$R_\delta$ or~$R$
  vanish. This leaves at least~$M$ available values for~$y_n$
  in~$\Zint{-M,M}$.

  For every~$(y_1,\ldots,y_n)\in \mathcal{I}_n(T)$, the nonvanishing
  of the polynomials~$R_\delta$ and~$R$ at~$(y_1,\ldots,y_n)$
  guarantees conditions~\ref{cd:V-delta} and~\ref{cd:coprime} of
  Definition~\ref{def:tree} respectively. Finally, the
  inequality~$C\geq C_3$ ensures that~$M\geq 4d[L:\Q]$.
\end{proof}

\begin{thm}
  \label{thm:height-bound-proved}
  Let~$H_\delta$ be an absolutely irreducible Hecke correspondence
  on~$\Scal\times \Tcal$ defined by an element~$\delta\in G(\A_f)$,
  and let~$d(\delta)$ be the degree of~$H_\delta$.
  Let~$\Frac\in L(J_1,\ldots,J_n)$ be a coefficient of one of the
  modular equations~$\Psi_{\delta,m}$ for $1\leq m\leq n+1$. Then the
  height of~$\Frac$ is bounded above by~$C \Heckedeg(\delta)$,
  where~$C$ is a constant independent of~$\delta$; more precisely we
  can take
  \begin{align*}
    C = 2^{n-1} &\bpar{2 C_{\ref{prop:height-modeq-evaluation}} (1+C'') + 2 C_{\ref{prop:main-frac}} C_{\ref{thm:degree-bound-proved}}
                  \bpar{\log(4 C_{\ref{thm:degree-bound-proved}} C_{\ref{prop:height-modeq-evaluation}})
                  + 2 C_{\log} +1 +C''} \\
                &\quad + 4 C_{\ref{thm:degree-bound-proved}}
                  (\log(C_{\ref{thm:degree-bound-proved}})
                  + C_{\log}) + 2C_{\ref{thm:degree-bound-proved}} (\log(2)+C'')
                  + 2\log(2 C_{\ref{thm:degree-bound-proved}})+2
                  },
  \end{align*}
  where
  $C'' = 3 + \log(2 C_{\ref{lem:interpolation-data}}) + 4 C_{\log}$.
\end{thm}

\begin{proof}
  By Lemma~\ref{lem:interpolation-data}, there exist valid evaluation
  data $(T,a,M)$ for~$\Frac$ such that the inequality
  $M \leq C_{\ref{lem:interpolation-data}}
  \Heckedeg(\delta)^4\max\{1,\log^3\isogdeg(\delta)\}+1$ holds. After
  scaling~$P$ and~$Q$ by an element of~$L^\times$, we can assume that
  $Q(a_1,\ldots,a_n) = 1$.

  Let $(y_1,\ldots,y_{n-1})\in \mathcal{I}_{n-1}(T)$, and write
  \begin{displaymath}
    \widetilde{\Frac}(Y) = \Frac(y_1Y + a_1,\ldots y_{n-1}Y +
    a_{n-1}, Y + a_n).
  \end{displaymath}
  For every son~$y_n$ of~$y_{n-1}$ in~$T$, we have
  \begin{displaymath}
    h\bigl(y_1y_n +
    a_1,\ldots, y_{n-1}y_n + a_n\bigr)\leq \log\bigl((M+1)M\bigr) \leq 2\log(M+1).
  \end{displaymath}
  Therefore, by Proposition~\ref{prop:height-modeq-evaluation},
  \begin{align*}
    \h(\widetilde{\Frac}(y_n))
    &\leq C_{\ref{prop:height-modeq-evaluation}}
      \Heckedeg(\delta) \bpar{2\log(M+1) +
      \log\isogdeg(\delta)}\\
    &\leq 2  C_{\ref{prop:height-modeq-evaluation}} \Heckedeg(\delta)
      \bpar{\log(M+1) + \max\{1,\log\isogdeg(\delta)\}}.
  \end{align*}
  Denote this last quantity by~$H$. We have~$H\geq 4$
  and~$H\geq \log(2M)$. Moreover, in the notation of
  Definition~\ref{def:tree}, the inequality~$M\geq 2B\log^2(B+1)$
  ensures that
  \begin{displaymath}
    \frac{M}{\log(M+1)}\geq B
    \geq d^3\bpar{4 C_{\ref{prop:height-modeq-evaluation}}
      \Heckedeg(\delta) \max\{1,\log \isogdeg(\delta)\}}.
  \end{displaymath}
  Therefore~$M\geq d^3 H$.
  
  We are in position to apply Proposition~\ref{prop:main-frac} for the
  univariate rational fraction~$\widetilde{F}$ on the
  interval~$\Zint{-M,M}$, with~$\eta=2$, using the sons
  of~$(y_1,\ldots,y_{n-1})$ in~$T$ as evaluation points. We obtain
  \begin{align*}
    \h(\widetilde{\Frac}) &\leq H + 2 C_{\ref{prop:main-frac}} d\log(2dH)
                           + d\log(2M) + \log(d+1)\\
    &\leq C' \Heckedeg(\delta) \max\{1,\log \isogdeg(\delta)\},
  \end{align*}
  where~$C'$ is a constant independent of~$\delta$. In order to obtain
  an explicit expression for~$C'$, we note that
  \begin{displaymath}
    \log(M+1)\leq C'' \max\{1,\log \isogdeg(\delta)\}
  \end{displaymath}
  where~$C''$ is defined as in the statement of the theorem.  We check
  that we can take
  \begin{align*}
    C'&= 2 C_{\ref{prop:height-modeq-evaluation}} (1+C'')
        + 2 C_{\ref{prop:main-frac}} C_{\ref{thm:degree-bound-proved}}
        \bpar{\log(4 C_{\ref{thm:degree-bound-proved}} C_{\ref{prop:height-modeq-evaluation}})
        + 2 C_{\log} +1 +C''}\\
      &\qquad + C_{\ref{thm:degree-bound-proved}} (\log(2)+C'')
        + \log(2 C_{\ref{thm:degree-bound-proved}})+1.
  \end{align*}

  In the second part of the proof, we relate the height
  of~$\widetilde{\Frac}$ with the height of~$\Frac$. The quotient
  \begin{displaymath}
    \dfrac{P(y_1Y + a_1,\ldots, y_{n-1} Y + a_{n-1}, Y +
      a_n)}{Q(y_1 Y + a_1,\ldots, y_{n-1} Y + a_{n-1}, Y + a_n)}
  \end{displaymath}
  is a way to write~$\widetilde{\Frac}$ in irreducible form in~$L(Y)$,
  and has a coefficient equal to 1. Therefore~$\h(\widetilde{\Frac})$
  is the affine height of the coefficients appearing in the
  quotient. Hence
  \begin{displaymath}
    \h\bigl(P(y_1 Y_n + a_1, \ldots, y_{n-1} Y_n + a_{n-1}, Y_n + a_n)\bigr)
    \leq C' \Heckedeg(\delta) \max\{1, \log\isogdeg(\delta)\}
  \end{displaymath}
  for every $(y_1,\ldots,y_{n-1})\in \mathcal{I}_{n-1}(P)$, and the
  same inequality holds for~$Q$. Since $N_1= 2d$, we can interpolate
  successively the variables~$y_{n-1},\ldots, y_1$, using
  Proposition~\ref{prop:main-poly} with~$2d$ evaluation points at each
  vertex of the tree~$T$. Finally we obtain
  \begin{align*}
    \h(\Frac) &\leq 2^{n-1}\bpar{C' \Heckedeg(\delta) \max\{1,\log \isogdeg(\delta)\}
                + 4d\log(4d) + d\log(2M) + \log(d+1)} \\
              &\leq 2^{n-1} \bpar{C' + 4 C_{\ref{thm:degree-bound-proved}}
                (\log(C_{\ref{thm:degree-bound-proved}})
                + C_{\log}) + C_{\ref{thm:degree-bound-proved}} (\log(2)+C'') \\
              &\qquad\qquad + \log(2 C_{\ref{thm:degree-bound-proved}}) +1
                } \Heckedeg(\delta) \max\{1,\log\isogdeg(\delta)\}.
  \end{align*}
\end{proof}

\subsection{Explicit height bounds in dimension 2}
\label{sub:height-explicit}

In this final section, we derive explicit height bounds for modular
equations of Siegel type for abelian surfaces. Our first aim is to
provide an explicit value for the constant in
Corollary~\ref{cor:j-isog}, using Theta constants of level~4 as an
intermediate step. To relate Theta heights and $j$-heights in this
setting, we use Mestre's algorithm and Thomae's formul\ae\ instead of
writing out polynomials~$N_{J,i},D_{J,i},N_{\Theta,i}$,
and~$D_{\Theta,i}$ as in the proof of
Proposition~\ref{prop:j-faltings}.

\begin{prop}
  \label{prop:j-theta-as}
  Let $A$ be a principally polarized abelian surface defined
  over~$\Qbar$ where~$j_1,j_2,j_3$ are well defined, and assume
  that~$j_3(A)\neq 0$. Then we have
  \begin{displaymath}
    \h_j(A) \leq 40 \h_{\Theta,4}(A) + 12 \quad\text{and}\quad
    \h_{\Theta,4}(A) \leq 200 \h_j(A) + 1000.
  \end{displaymath}
\end{prop}

\begin{proof}
  Recall the expression of Igusa invariants in terms of the Siegel
  modular forms $I_4,I_6',I_{10},$ and~$I_{12}$:
  \begin{equation}
    \label{eq:igusa-cov}
    j_1 = \dfrac{I_4 I_6'}{I_{10}},\quad
    j_2 = \dfrac{I_4^2I_{12}}{I_{10}^2},\quad\text{and}\quad j_3 = \dfrac{I_4^5}{I_{10}^2}.
  \end{equation}
  These modular forms have a polynomial expression in terms of theta
  constants of level~$4$: see for instance
  \cite[§II.7.1]{streng_ComplexMultiplicationAbelian2010}. The total
  degrees of the polynomials giving~$I_4,I_6',I_{10}$ and~$I_{12}$
  are~$8, 12, 20$ and~$24$ respectively; they contain
  respectively~$10, 60, 1$ and~$15$ monomials, and their height is
  zero. Up to scaling, we may assume that the first theta
  constant~$\theta_0$ takes the value~$1$. Then, by
  Proposition~\ref{prop:multivariate-evaluation}, we have
  \begin{displaymath}
    \h(I_4^5 ,\, I_4 I_6' I_{10},\, I_4^2 I_{12},\, I_{10}^2) \leq 5\log(10) + 40 \h_{\Theta,4}(A),
  \end{displaymath}
  hence the first inequality
  \begin{displaymath}
    \h_j(A)\leq 40 \h_{\Theta,4}(A) + 12.
  \end{displaymath}

  For the second inequality, we follow Mestre's algorithm
  \cite{mestre_ConstructionCourbesGenre1991}. Starting from
  $j_1(A), j_2(A)$ and~$j_3(A)$, Mestre's algorithm constructs a
  hyperelliptic curve $y^2 = f(x)$ whose Jacobian is isomorphic to~$A$
  over~$\Qbar$. Up to scaling~$f$, we may take $I_{10} = 1$ in
  equation~\eqref{eq:igusa-cov}. Then we see that $j_1(A), j_2(A)$
  and~$j_3(A)$ are realized by values of $I_2, I_4, I_6'$,
  and~$I_{10}$ in~$\Qbar$ such that
  \begin{displaymath}
    \h(I_2,I_4,I_6',I_{10}) \leq h_j(A).
  \end{displaymath}
  
  The roots of~$f$ are the intersection points of a conic and a cubic
  in~$\Proj^2$ whose equations are given explicitly in terms
  of~$I_2,I_4,I_6$, and~$I_{10}$.  In order to obtain the equation
  $\sum_{i,j=1}^3 c_{ij} z_i z_j = 0$ of the conic, we start from
  Mestre's equation $\sum_{i,j=1}^3 A_{ij} x_i x_j = 0$ and substitute
  the expressions of $A,B,C$, and~$D$ in terms of $I_2,I_4,I_6'$,
  and~$I_{10}$.  Then we multiply by~$2^{11}3^{13}5^{14}$ and make the
  substitutions
  \begin{displaymath}
    z_1 = 202500x_1,\quad z_2 = 225x_2,\quad z_3 = x_3.
  \end{displaymath}
  Then, each coefficient~$c_{ij}$ has an expression as a multivariate
  polynomial in~$I_2,I_4$, and~$I_6'$ (recall that $I_{10} = 1$) of
  total degree at most~$7$; its coefficients are integers whose
  absolute values are bounded by~$324\cdot 10^6$. By
  Proposition~\ref{prop:multivariate-evaluation}, we have
  \begin{displaymath}
    \h\bpar{(c_{ij})_{1\leq i,j\leq 3}}
    \leq 7(h_j(A) + \log(3)) + 19.6 + 3\log(8) \leq 7h_j(A) + 33.6.
  \end{displaymath}
  If we restrict to $c_{11}, c_{12}$, and~$c_{22}$, then we obtain a
  smaller upper bound, since the total degree and the height of
  coefficients are at most~$5$ and~$18.3$ respectively.  Similarly,
  the cubic equation, denoted by
  $\sum_{1\leq i\leq j\leq k\leq 3} c_{ijk}z_iz_jz_k = 0$, has total
  degree at most~11 in $I_2,I_4$, and~$I_6'$, and has integer
  coefficients whose heights are at most~$33.5$.

  In order to find the hyperelliptic curve equation~$f$, we
  parametrize the conic. Let us show that it contains a point~$P_0$
  defined over~$\Qbar$ such that $\h(P_0)\leq 5h_j(A) + 29.9$. We can
  assume that $c_{11}\neq 0$; otherwise we take $P_0 =
  (1:0:0)$. Let~$\alpha$ be a root of the monic polynomial
  \begin{displaymath}
    \alpha^2 + \dfrac{c_{12}}{c_{11}} \alpha + \dfrac{c_{22}}{c_{11}} =0. 
  \end{displaymath}
  The point $P_0 = (\alpha:1:0)$ belongs to the conic, and by
  Proposition~\ref{prop:pol-root},
  \begin{displaymath}
    \begin{aligned}
      \h(P_0) = \h(\alpha) &\leq \h(c_{11}, c_{12}, c_{22})
      + \log(2) \\
      &\leq 5(h_j(A) + \log(3)) + 18.3 + 3\log(6) + \log(2)\\
      &\leq
      5h_j(A) + 29.9.
    \end{aligned}
  \end{displaymath}
  
  We parametrize the conic using~$P_0$ as a base point; for
  simplicity, we continue to assume that~$c_{11}\neq 0$. For
  $(u:v)\in\Proj^1(\Qbar)$, the point $(z_1:z_2:z_3)$ defined by
  \begin{displaymath}
    \begin{aligned}
      z_1 &= \alpha(c_{11}u^2 + c_{13}uv + c_{33}v^2) -
      u((2c_{11}\alpha + c_{12})u + (c_{13}\alpha + c_{23})v),
      \\
      z_2 &= c_{11}u^2 + c_{13} uv + c_{33}v^2, \quad\text{and} \\
      z_3 &= - v((2c_{11}\alpha + c_{12})u + (c_{13}\alpha + c_{23})v)
    \end{aligned}
  \end{displaymath}
  runs through the conic. Substituting these expressions in the cubic
  equation gives the curve equation~$f$. The polynomials we obtain
  have total degrees at most~$29$ in $I_2,I_4$, and~$I_6'$; they have
  degree at most~$3$ in~$\alpha$; and their coefficients are integers
  whose heights are bounded above by~$86.9$. Therefore, by
  Proposition~\ref{prop:multivariate-evaluation} (separating
  $I_2,I_4,I_6'$ from $\alpha$), we have
  \begin{displaymath}
    \begin{aligned}
      \h(f) &\leq 29(h_j(A) + \log(3)) + 86.9 + 3(5h_j(A) + 29.9) +
      3\log(30) + \log(4) \\
      &\leq 44h_j(A) + 220.1.
    \end{aligned}
  \end{displaymath}
  Making~$f$ monic does not change its height.

  Thomae's formul\ae~\cite[IIIa.8.1]{mumford_TataLecturesTheta1984}
  give an expression of the Theta constants of level~4 of~$A$ in terms
  of roots of~$f$: if~$\theta$ is one of these Theta constants,
  then~$\theta^4$ is a product of 18 differences of roots of~$f$ (up
  to a common multiplicative factor). Therefore, by
  Proposition~\ref{prop:pol-root}, we obtain
  \begin{displaymath}
    h_{\Theta,4}(A,L) \leq \tfrac{1}{4}\cdot 18 (\h(f)+\log(4)) \leq
    198 h_j(A) + 997.
  \end{displaymath}
\end{proof}

As a consequence, we obtain an explicit analogue of
Corollary~\ref{cor:j-isog} in the case of isogenies between
principally polarized abelian surfaces.

\begin{prop}
  \label{prop:isog-as}
  Let~$A$ and~$A'$ be principally polarized abelian surfaces
  over~$\Qbar$ where~$j_1,j_2,j_3$ are well defined, and assume
  that~$j_3(A)j_3(A')\neq 0$. Let~$d\geq 1$ be an integer. If~$A$
  and~$A'$ are linked by an isogeny of degree~$d$, then we have
  \begin{displaymath}
    \Hbar_j(A') \leq 8000 \Hbar_j(A) + 1.08\cdot 10^{11} \log(\Hbar_j(A))
    + 1.67\cdot 10^{12} + 20\log d.
  \end{displaymath}
\end{prop}

\begin{proof}
  By Theorem~\ref{thm:theta-faltings} and
  Proposition~\ref{prop:faltings-isog} and~\ref{prop:j-theta-as}
  (noting that $C(2,4)\leq 1.35\cdot 10^9$), we have
  \begin{align*}
    \Hbar_{\Theta,4}(A) &\leq 200 \Hbar_j(A) + 1000,\\
    \tfrac12 \Hbar_F(A) &\leq \Hbar_{\Theta,4}(A) + C(2,4) \log(\Hbar_{\Theta,4}(A)+2) \\
                        &\leq 200 \Hbar_j(A) + C(2,4) \log(1202) + C(2,4)\log (\Hbar_j(A)), \\
    \tfrac12 \Hbar_F(A')&\leq \tfrac12 \Hbar_F(A) + \tfrac14\log\ell,\\
    \Hbar_{\Theta,4}(A') &\leq \tfrac12 \Hbar_F(A') + C(2,4) \log(\Hbar_F(A')+2) \\
                        &\leq 200 \Hbar_j(A) + C(2,4)\log(1202)
                          + 2 C(2,4)\log(\Hbar_j(A)) + \tfrac14\log\ell \\
                        &\qquad + C(2,4) \log\bpar{402+2 C(2,4)\log(1202)
                          + C(2,4) + \tfrac12\log\ell},\\
                        &\leq 200 \Hbar_j(A) + 2 C(2,4) \log(\Hbar_j(A))
                          + 4.17\cdot 10^{10} + \tfrac12\log\ell, \quad\text{and}\\
    \Hbar_j(A') &\leq 40 \Hbar_{\Theta,4}(A) +12\\
                        &\leq 8000 \Hbar_j(A) + 80 C(2,4) \log\Hbar_j(A)
                          + 1.67\cdot 10^{12} + 20\log\ell.
  \end{align*}
\end{proof}

In Lemma~\ref{lem:equation-V}, we take~$\lambda = I_4$ and
$\lambda' = I_4 I_{10}$. We have
\begin{displaymath}
  \wt(\lambda^\delta) = 14\Heckedeg(\delta)+4,
\end{displaymath}
which is greater than~16, the minimum weight for which Siegel modular
forms define a projective embedding of~$\Scal$. Hence~$\xi$ can be chosen
to be a modular form of weight~$\wt(\lambda^\delta)$. The
fraction~$R''$ has degree at most~$\frac73(\Heckedeg(\delta)+1)$ by
Lemma~\ref{lem:rewrite-siegel}; this is also an upper bound
on~$\deg(P_\delta)$.

We also mimic the proof of
Proposition~\ref{prop:height-modeq-evaluation} in the Siegel
case. Let~$[x,g]$ be a point of~$\Scal$ with Igusa
invariants~$(j_1,j_2,j_3)\in\V_\delta$.  For each~$1\leq m\leq 3$, by
Remark~\ref{rem:simpler-modeqs}, the
polynomial~$\Psi_{\delta,m}(j_1,j_2,j_3)$ is the evaluation of a
multivariate polynomial in~$2 \Heckedeg(\delta)$ variables. Moreover,
the Hecke correspondence describes isogenies of degree~$\ell^2$. By
Proposition~\ref{prop:isog-as}, we have
\begin{equation}
  \label{eq:height-eval-siegel}
  \h\bpar{\Psi_{\delta,m}(j_1,j_2,j_3)}\leq 2 \Heckedeg(\delta) 
  \bpar{8000 \Hbar(j_1,j_2,j_3) + 1.08\cdot 10^{11} \log(\Hbar_j(A))
    + 1.67\cdot 10^{12} + 40\log\ell}.
\end{equation}
Therefore, we can take
\begin{displaymath}
  C_{\ref{prop:height-modeq-evaluation}} = 3.35 \cdot 10^{12}.
\end{displaymath}
Moreover, we have~$\Heckedeg(\delta) = \ell^3+\ell^2+\ell+1$
and~$\isogdeg(\delta) = \ell^2$. Hence we can take
\begin{displaymath}
  C_{\log} = \frac32 + \log(2) \leq 2.2.
\end{displaymath}
We also take
\begin{align*}
  C_{\ref{prop:main-frac}} &= 960 &\text{because } L=\Q,\\
  C_{\ref{thm:degree-bound-proved}} &= \frac{10}{3} &\text{by Proposition~\ref{prop:degree-siegel}},
                                                      \text{ and}\\
  C_{\ref{lem:equation-V}} &= 15 &\text{since } \Heckedeg(\delta)\geq 15.
\end{align*}
In Lemma~\ref{lem:interpolation-data}, we can take
\begin{displaymath}
  C_{\ref{lem:interpolation-data}} = 1.36 \cdot 10^{17}
\end{displaymath}
and in Theorem~\ref{thm:height-bound-proved}, we can take
\begin{displaymath}
  C_{\ref{thm:height-bound-proved}} = 1.42 \cdot 10^{15}.
\end{displaymath}
Since~$d(\delta)\leq 2\ell^3$
and~$\max\{1,\log\ell(\delta)\}\leq 2\log(\ell)$, we obtain the
following result.

\begin{thm}
  \label{thm:explicit-height-siegel}
  Let~$\ell\geq 1$ be a prime number, and
  let~$\Frac\in \Q(J_1,J_2,J_3)$ be a coefficient of one of the Siegel
  modular equations of level~$\ell$ in Igusa invariants. Then we have
  \begin{displaymath}
    \h(\Frac) \leq 5.68 \cdot 10^{15} \ell^3 \log(\ell).
  \end{displaymath}
\end{thm}

In order to obtain tighter height bounds on Siegel modular equations,
we could repeat the computations of~§\ref{sub:height-end} using an
expression of the form~\eqref{eq:height-eval-siegel} for the height of
evaluated modular equations, instead of the simpler formula used in
Proposition~\ref{prop:height-modeq-evaluation}. However we cannot hope
to obtain a constant in Theorem~\ref{thm:explicit-height-siegel} that
is much smaller than~$C(2,4) \simeq 1.35\cdot 10^9$ using our
methods. Experimentally, we observe that the tighter
inequalities~$h(\Frac)\leq 48.7\,\ell^3\log(\ell)$
and~$h(\Frac)\leq 43.6\,\ell^3\log(\ell)$ hold for~$\ell=2$
and~$\ell=3$ respectively.

We could also give an analogue of
Theorem~\ref{thm:explicit-height-siegel} in the case of modular
equations of Hilbert type for~$\Q(\sqrt{5})$ in Gundlach
invariants. To replace Proposition~\ref{prop:j-theta-as}, we would use
the relations between Gundlach and Igusa invariants (see for
instance~\cite[§2.3]{milio_ModularPolynomialsHilbert2020}) and the
explicit curve equation given by
\cite[Prop.~A.4]{kieffer_ComputingIsogeniesModular2019}. We leave the
precise calculations for future work.

\begin{acknowledgement}
  The author thanks Fabien Pazuki and his Ph.D.~advisors, Damien
  Robert and Aurel Page, for answering the author's questions. The
  author also thanks the anonymous referees for helpful comments.
  Finally, acknowledgments are due to Aurel Page for his careful
  proofreading of an earlier version of the paper.
\end{acknowledgement}

\bibliographystyle{abbrv}
\bibliography{bounds-modular-equations.bib}

\affiliationone{
  Jean Kieffer

  Institut de Mathématiques de Bordeaux
  
  351 cours de la Libération
  
  33400 Talence

  France
   \email{jean.kieffer@math.u-bordeaux.fr}}

\end{document}